\documentclass[12pt,a4paper,twoside]{article}
\usepackage[latin1]{inputenc}
\usepackage[T1]{fontenc}
\usepackage{graphicx}
\usepackage{latexsym}
\usepackage{amsmath,amssymb,amsthm}
\usepackage{bbm}
\usepackage{ifpdf}
\usepackage{exscale}
\usepackage{fancyhdr}
\usepackage{tocloft}
\usepackage{hyperref}

\usepackage[DIV12]{typearea}

\newcommand{\e}{\ensuremath{\varepsilon}}
\newcommand{\dt}{\ensuremath{\textup{d}}}
\newcommand{\dist}{\operatorname{d}}
\newcommand{\diam}{\operatorname{diam}}
\newcommand{\dL}{\operatorname{d}_L}
\newcommand{\dLk}{\operatorname{d}_{L+r}}
\newcommand{\dtL}{\operatorname{\tilde{d}}}
\newcommand{\sh}{\operatorname{Sh}}

\newcommand{\czero}{\operatorname{\bf{A1}}}
\newcommand{\cone}{\operatorname{\bf{A0}}}
\newcommand{\ctwo}{\operatorname{\bf{C1}}}

\newcommand{\good}{\operatorname{Good}_L}
\newcommand{\onebad}{\operatorname{OneBad}_L}
\newcommand{\onebadi}{\operatorname{OneBad}_L^{(i)}}
\newcommand{\onebadj}{\operatorname{OneBad}_L^{(j)}}
\newcommand{\manybad}{\operatorname{ManyBad}_L}
\newcommand{\bbad}{\operatorname{BdBad}_{L,r}}

\newcommand{\badP}{\ensuremath{\mathcal{B}}}
\newcommand{\badPo}{\ensuremath{\mathcal{B}^\star}}

\newtheorem{theorem}{Theorem}[section]
\newtheorem{lemma}{Lemma}[section]
\newtheorem{proposition}{Proposition}[section]
\newtheorem{corollary}{Corollary}[section]
\newenvironment{prooof}{\noindent{\bf Proof:}}{%
  \hspace*{\fill}$\Box$\par\vskip2ex}
\newenvironment{prooof2}{\noindent{\bf Proof }}{%
  \hspace*{\fill}$\Box$\par\vskip2ex}

\newcommand{\pibm}{\ensuremath{\pi^{\textup{BM}}}}
\newcommand{\phbm}{\ensuremath{\hat{\pi}^{\textup{BM}}}}
\newcommand{\phibm}{\ensuremath{\phi}^{\textup{BM}}}

\newcommand{\ph}{\ensuremath{\hat{\pi}}}
\newcommand{\pho}{\ensuremath{\breve{\pi}}}
\newcommand{\pt}{\ensuremath{\tilde{\pi}}}

\newcommand{\Ph}{\ensuremath{\hat{\Pi}}}
\newcommand{\Pho}{\ensuremath{\breve{\Pi}}}

\newcommand{\Pt}{\ensuremath{\tilde{\Pi}}}
\newcommand{\Ptg}{\ensuremath{\tilde{\Pi}^g}}
\newcommand{\Phg}{\ensuremath{\hat{\Pi}^g}}
\newcommand{\Phgo}{\ensuremath{\breve{\Pi}^g}}
\newcommand{\Pg}{\ensuremath{\Pi^g}}

\newcommand{\gh}{\ensuremath{\hat{g}}}
\newcommand{\gho}{\ensuremath{\breve{g}}}
\newcommand{\gt}{\ensuremath{\tilde{g}}}
\newcommand{\Gh}{\ensuremath{\hat{G}}}
\newcommand{\Gho}{\ensuremath{\breve{G}}}
\newcommand{\Ghg}{\ensuremath{\hat{G}^g}}
\newcommand{\Ghgo}{\ensuremath{\breve{G}^g}}
\newcommand{\Gt}{\ensuremath{\tilde{G}}}
\newcommand{\Gtg}{\ensuremath{\tilde{G}^g}}
\newcommand{\gbm}{\ensuremath{g^{\textup{BM}}}}

\newcommand{\Prw}{\operatorname{P}}
\newcommand{\Pbm}{\operatorname{P}^{\textup{BM}}}
\newcommand{\Erw}{\operatorname{E}}
\newcommand{\pP}{\ensuremath{\mathbb{P}}}
\newcommand{\pQ}{\ensuremath{\mathbb{Q}}}
\newcommand{\pE}{\ensuremath{\mathbb{E}}}
\newcommand{\pF}{\ensuremath{\mathcal{F}}}
\newcommand{\pG}{\ensuremath{\mathcal{G}}}
\newcommand{\bj}{\ensuremath{{\bf j}}}

\newcommand{\ctime}{\operatorname{\bf{C2}}}
\newcommand{\goods}{\operatorname{Good}_L^{\textup{\tiny sp}}}
\newcommand{\bads}{\operatorname{Bad}_L^{\textup{\tiny sp}}}
\newcommand{\goodt}{\operatorname{Good}_L^{\textup{\tiny tm}}}
\newcommand{\onebadt}{\operatorname{OneBad}_L^{\textup{\tiny tm}}}
\newcommand{\manybadt}{\operatorname{ManyBad}_L^{\textup{\tiny tm}}}
\newcommand{\nottoobadt}{\operatorname{NotTooBad}_L^{\textup{\tiny tm}}}
\newcommand{\badPs}{\ensuremath{\mathcal{B}}_L^{\textup{\tiny sp}}}
\newcommand{\badPt}{\ensuremath{\mathcal{B}}_L^{\textup{\tiny tm}}}

\theoremstyle{definition}
\newtheorem{remark}{Remark}[section]

\renewcommand{\theenumi}{{\rm (\roman{enumi})}}

\begin{document}

\addtocontents{toc}{\protect\thispagestyle{empty}}
\pagestyle{empty}
\thispagestyle{empty}

\newpage
 \pagestyle{fancy} 
 \renewcommand{\headrulewidth}{0mm}
  \chead[\leftmark]{\rightmark}      
  \rhead[]{\thepage}
  \lhead[\thepage]{}
  \cfoot{} 
 \pagenumbering{arabic}

  \title{Long-time Behavior of Random Walks in Random Environment}
\author{Erich Baur\footnote{Universit\"at Z\"urich. Email:
    erich.baur@math.uzh.ch.}}

\date{}
\maketitle 
\thispagestyle{empty}
\begin{abstract}
We study behavior in space and time of random walks in an i.i.d. random
environment on $\mathbb{Z}^d$, $d\geq 3$. It is assumed that the measure
governing the environment is isotropic and concentrated on environments
that are small perturbations of the fixed environment corresponding to
simple random walk. We develop a revised and extended version of the paper 
of Bolthausen and Zeitouni (2007) on exit laws from large balls, which, as
we hope, is easier to follow. Further, we study mean sojourn times in balls.
  
This work is part of the author's PhD thesis under the supervision of Erwin Bolthausen. A
generalization of the results on exit measures to certain anisotropic
random walks in random environment is available
at~\href{http://arxiv.org/abs/1309.3169}{\tt arXiv:1309.3169 [math.PR]}.
\end{abstract}
\tableofcontents 

   \section{Introduction}
\label{intro}
\subsection{The  model}
\subsubsection{General description}
Consider the integer lattice $\mathbb{Z}^d$ with unit vectors $e_i$, whose
$i$th component equals $1$. We let $\mathcal{P}$ be the set of probability
distributions on $\{\pm e_i : i = 1,\ldots,d\}$. Given a probability
measure $\mu$ on $\mathcal{P}$, we equip
$\Omega=\mathcal{P}^{\mathbb{Z}^d}$ with its natural product $\sigma$-field
$\pF$ and the product measure $\pP_\mu = \mu^{\otimes\mathbb{Z}^d}$. Each
element $\omega\in\Omega$ yields transition probabilities of a nearest
neighbor Markov chain on $\mathbb{Z}^d$, the {\it random walk in random
  environment} (RWRE for short), via
\begin{equation*}
p_\omega(x,x+e) = \omega_x(e),\quad e\in\{\pm e_i : i = 1,\ldots,d\}.
\end{equation*}
We write $\Prw_{x,\omega}$ for the ``quenched'' law of the canonical Markov
chain $(X_n)_{n\geq 0}$ with these transition probabilities, starting at
$x\in\mathbb{Z}^d$.  The probability measure
\begin{equation*}
P = \int_\Omega \Prw_{0,\omega}\pP(\dt\omega)
\end{equation*}
is commonly referred to as averaged or ``annealed'' law of the RWRE started
at the origin. 
\subsubsection{Additional requirements}
We study asymptotic properties of the RWRE in dimension $d\geq 3$ when
the underlying environments are small perturbations of the fixed
environment $\omega_x(\pm e_i) =1/(2d)$ corresponding to simple or standard random walk.
In order to fix a perturbative regime, we introduce the following condition.
\begin{itemize}
\item Let $0<\varepsilon<1/(2d)$. We say that $\cone(\e)$ holds if $\mu(\mathcal{P}_\e) = 1$,
where 
\begin{equation*}
\mathcal{P}_\e = \left\{q\in\mathcal{P} : \left|q(\pm e_i) -
  1/(2d)\right| \leq \e\mbox{ for all }i=1,\ldots,d\right\}. 
\end{equation*}
\end{itemize}
The perturbative behavior concerns the  behavior of the RWRE
when $\cone(\e)$ holds for small $\e$. However, even for arbitrarily small
$\e$, such walks can behave very differently compared to simple random
walk. This motivates a further ``centering''
restriction on $\mu$.
\begin{itemize}
\item We say that $\czero$ holds if 
  $\mu$ is invariant under all orthogonal transformations fixing the
  lattice $\mathbb{Z}^d$, i.e. if $O:\mathbb{R}^d\rightarrow\mathbb{R}^d$ is any orthogonal matrix
  that maps $\mathbb{Z}^d$ onto itself, then the laws of
  $(\omega_0(Oe))_{|e|=1}$ and $(\omega_0(e))_{|e|=1}$ coincide.
\end{itemize}
If $\czero$ holds, $\pP_\mu$ is called {\it isotropic}. 
\subsection{Informal description of the results}
In the following, we write $\pP$ instead of $\pP_\mu$ and denote by $\pE$
the corresponding expectation.
\subsubsection{Exit laws from balls}
In the main part of this work, we investigate the RWRE exit distribution
from the ball $V_L=\{x\in\mathbb{Z}^d : |x| \leq L\}$ when the radius $L$
is large. Assuming $\czero$  and $\cone(\e)$ for small $\e$, we show that the exit law of the walk,
started from a point $x$ with 
$|x|\leq L/5$, is close to that of simple random walk. 
More precisely, using the multiscale analysis introduced in 
Bolthausen and Zeitouni~\cite{BZ}, we prove that
if the radius $L$ tends to infinity, then
\begin{enumerate}
\item The difference of the two exit laws measured in total variation stays
  small as $L$ increases (but does not tend to zero, due to boundary
  effects) (Theorem~\ref{def:main-theorem} (i)).
\item The distance between the two exit laws converges to zero if they are
  convolved with an additional smoothing kernel on a scale increasing
  arbitrarily slowly with $L$ (Theorem~\ref{def:main-theorem} (ii)).
\item The RWRE exit measure of boundary portions of size $\geq (L/(\log
  L)^{15})^{d-1}$ can be bounded from above by that of simple random
  walk. Evaluated on segments of size $\geq (L/(\log L)^{6})^{d-1}$, the
  two measures agree up to a multiplicative error, which tends to one as
  $L$ increases (Theorem~\ref{local-thm-exitmeas}).
\end{enumerate}
The first two parts already appeared in~\cite{BZ}, which serves as the basis
for our work. However, for reasons explained below, it was of great interest to find a somewhat different approach.

\subsubsection{Mean sojourn times}
The results on exit laws can be used to prove transience of the RWRE
(Corollary~\ref{def:main-transience}), and they provide an invariance
principle up to time transformation. Getting complete control over time is
a major open problem, and in that direction, we look in Section~\ref{times}
at mean holding or sojourn times in balls. Our basic insight is that
exceptionally small or large times can only be produced by spatially
atypical regions. Consequently, the philosophy behind our approach is to
derive statements on sojourn times from estimates on exit laws. However,
our results on exit distributions seem not quite sufficient to handle the
presence of strong traps, i.e. regions where the RWRE cannot escape for a
long time with high probability. We therefore make an additional assumption
which guarantees that the mass of environments producing very large times
is sufficiently small. Let $\tau_L = \inf\{n\geq 0 : X_n\notin V_L\}$ be
the first exit time of the RWRE from the ball $V_L$, and denote by
$\Erw_{0,\omega}$ the expectation with respect to $\Prw_{0,\omega}$.
\begin{itemize}
\item We say that {\bf A2} holds if for large $L$,
\begin{equation*}
  \pP\left(\Erw_{0,\omega}\left[\tau_L\right]>(\log L)^4L^2\right)\leq
  L^{-8d}. 
\end{equation*}
\end{itemize}
Assuming this additional condition, we prove 
\begin{enumerate}
\setcounter{enumi}{3}
\item For almost all
environments, the normalized quenched mean time
$\Erw_{0,\omega}[\tau_L]/L^2$ is finally contained in a small interval
around one, where the size of the interval converges to zero if the
disorder $\varepsilon$ tends to zero (Proposition~\ref{def:times-main-prop}
and Theorem~\ref{def:times-thm2}).  
\end{enumerate}
We believe that {\bf A2} follows from $\cone(\e)$ and $\czero$, even with a
faster decay of the probability. It remains an open (and possibly
challenging) problem to prove this. An example where {\bf A2} trivially
holds true is given in Remark~\ref{remark-balanced}.

\subsection{Discussion of this work}
\label{intro-purpose}
The part on exit measures should be seen as a corrected and extended
version of Bolthausen and Zeitouni~\cite{BZ}. Most of the ideas can already
be found there, and also our proofs sometimes follow
those of~\cite{BZ}. However, our focus lies more on Green's function
estimates on ``goodified'' environments, which are developed in
Section~\ref{super}. Partly based on (unpublished) notes of Bolthausen,
this section is entirely new, and the results obtained make the proofs of
the main statements more transparent. The core
statement is Lemma~\ref{def:superlemma}, which gives a bound on the (coarse
grained) RWRE Green's function, for a large class of environments. As such
estimates were only partially present in~\cite{BZ}, the authors had to
repeatedly consider higher order expansions in terms of Green's functions
coming from simple random walk, which led to serious problems, for example
in Sections 4.3 and 4.4 in~\cite{BZ}. 

The reason for developing a new approach was twofold: On the
one hand, it seemed difficult to fix these problems 
ad hoc. On the other hand, we aimed at establishing a solid basis for
future work on this topic, in particular in the direction of a central
limit theorem. Further new points of this work can be summarized as follows.
\begin{itemize}
\item We give either new proofs of the statements in~\cite{BZ} or we revise
  the old ones. For example, the proofs leading to the main results on the
  exit measures in Sections~\ref{smv-exits} and~\ref{nonsmv-exits} are
  based on our new techniques. These include the bounds on Green's
  functions, the use of parametrized coarse graining schemes
  and the concept of goodified environments, which goes back to~\cite{BZ}
  and is further elaborated here.
\item The appendix is completely rewritten. In this part, the main
  corrections concern the proof of the key Lemma~\ref{def:hittingprob}
  (Lemma 3.4 in~\cite{BZ}), where different case had to be
  considered. Also, we provide a lower bound on exit probabilities
  (Lemma~\ref{def:hittingprob} (iii)), which was already implicitly used
  in~\cite{BZ}, but not proved.
\item We obtain local estimates for the exit
  measures (Theorem~\ref{local-thm-exitmeas}). The global estimates in 
  total variation distance are extended to starting points $|x|\leq L/5$.
\item The results on the exit distributions are used to control the
  mean sojourn time of the RWRE in balls, under an
  extra assumption on $\pP$.
\end{itemize}
To improve readability, we overview the main steps of this work in Section~\ref{readingguide}.

\subsection{A brief history}
The literature on random walks in random environment is vast,
and we do by no means intend to give a full overview here. Instead, we 
point at some cornerstones and focus on results which are relevant for our
particular model. For a more detailed survey, the reader is invited to
consult the lecture notes of Sznitman~\cite{SZ-LN},~\cite{SZ-LN2} and
Zeitouni~\cite{ZT-StF},~\cite{ZT-TRev}, and also the overview article of
Bogachev~\cite{BG}.

Recall the general model defined at the very beginning under ``General
description''. We additionally assume that the environment is {\it
  uniformly elliptic}, i.e. there exists $\kappa>0$ such that
$\pP$-almost surely, $\omega_x(e)\geq \kappa$ for all $x,e\in\mathbb{Z}^d$,
$|e| = 1$. Note that in the perturbative regime, this is automatically
true.

The natural condition of uniform ellipticity can sometimes be relaxed to
mere ellipticity $\omega_x(e)>0$ for $x,e\in\mathbb{Z}^d$, $|e| = 1$. Also,
it often suffices to require $\pP$ to be
stationary and ergodic instead of being ``i.i.d.''. 

\subsubsection{Dimension $d=1$}
Early interest in models of RWRE can be traced back to the 60's in the
context of biochemistry, where they were used as a toy model for DNA
replication, cf. Chernov~\cite{CHE} and
Temkin~\cite{TEM}. Solomon~\cite{SOL} started a rigorous mathematical
analysis in dimension $d=1$. He proved that if
\begin{equation*}
\pE\left[\log\rho\right]\neq
0,\quad\mbox{where }\rho=\omega_0(-1)/\omega_0(1),
\end{equation*}
then the RWRE is $P$-almost surely transient, whereas in the case
$\pE\left[\log\rho\right]=0$, the walk is $P$-a.s. 
recurrent. Further, he obtained almost sure existence of the limit speed
$v=\lim_{n\rightarrow\infty} X_n/n$,
\begin{equation*}
v = \left\{\begin{array}{l@{\quad }l}
\frac{1-\pE[\rho]}{1+\pE[\rho]} & \mbox{if\ }\pE[\rho]<1\\
\frac{1-\pE[\rho^{-1}]}{1+\pE[\rho^{-1}]}
& \mbox{if\ }\pE[\rho^{-1}]<1\\
0&\mbox{otherwise}
\end{array}\right..
\end{equation*}
His results already reveal some surprising features of the model. For
example, it can happen that $v=0$, but nonetheless the RWRE is transient
(note that this is impossible for a Markov chain with stationary
increments, according to Kesten~\cite{KES2}).  Also, if
$\overline{v}=\pE[\omega_0(1)-\omega_0(-1)]$ denotes the mean local drift,
it is possible that $|v|< |\overline{v}|$. Such slowdown effects, caused by
traps reflecting impurities in the medium, come again to light in limit
theorems for the RWRE under both the quenched and the annealed
measure. In~\cite{KKS}, Kesten, Kozlov and Spitzer proved that in the
transient case under the annealed law, both diffusive and sub-diffusive
behavior can occur, depending on a critical exponent connected to
hitting times. However, the strongest form of sub-diffusivity appears in the
recurrent case with non-degenerate site distribution $\mu$, for which
Sinai~\cite{SIN} proved that after $n$ steps, the RWRE is typically at
distance of order only $(\log n)^2$ away from the starting point. His
analysis shows that the walk spends most of the time at the bottom of
certain valleys. The limit law of $X_n/(\log n)^2$ is given by the
distribution of a functional of Brownian motion, cf. Kesten~\cite{KES} and
Golosov~\cite{GOL}. Let us finally mention that slowdown phenomena also
show up when studying probabilities of atypical events like large
deviations, see e.g.~\cite{GRHO},~\cite{DPZ},~\cite{GAZEIT}.
\subsubsection{Dimensions $d\geq 2$}
While the one-dimensional picture is quite complete, many
questions remain open in higher dimensions, including a classification into
recurrent/transient behavior, existence of a limit speed and invariance
principles. The main difficulties come from the non-Markovian character 
under the annealed measure and the fact that the RWRE is irreversible under the
quenched measure as soon as $d\geq 2$. 

Let us illustrate one prominent open problem, the directional zero-one law.
For an element $l$ from the unit sphere $\mathbb{S}^{d-1}$, denote the
event that the RWRE is transient in direction $l$ by
\begin{equation*}
A_l = \left\{\lim_{n\rightarrow\infty}X_n\cdot l=\infty\right\}.
\end{equation*}
Kalikow proved in~\cite{KAL} that $P(A_l\cup A_{-l}) \in\{0,1\}$. Is it
also true that $P(A_l) \in \{0,1\}$ ? The answer is affirmative in
dimension $d=1,2$ (\cite{SOL} for $d=1$, Merkl and Zerner~\cite{MERZER} for
$d=2$), but unknown for higher dimensions. It is known that a limit speed
$v\in\mathbb{R}^d$ (possibly zero) exists if $P(A_l)\in\{0,1\}$ for every
$l\in\mathbb{S}^{d-1}$, cf. Sznitman and Zerner~\cite{SZZER}.

Much progress has been made in characterizing models which exhibit
ballistic behavior, that is when the limit velocity $v$ is an almost
sure constant vector different from zero.
Here Sznitman's conditions $(T_\gamma)$,
$\gamma\in(0,1]$,  give a criterion for ballisticity and lead to an
invariance principle under the annealed measure $P$,
see Sznitman~\cite{SZ1},~\cite{SZ2} and also his lecture notes~\cite{SZ-LN2}.
When $d\geq 4$ and the disorder is small, a quenched invariance principle
has been shown by Bolthausen and Sznitman~\cite{BOLTSZ}.  A stronger
ballisticity condition was given earlier by Kalikow~\cite{KAL}. However, as
examples in Sznitman~\cite{SZ4} demonstrate, Kalikow's condition does not completely
describe ballistic behavior in dimensions $d\geq 3$. A handy and complete
characterization of ballisticity has still to be found. For recent
developments, see the work of Berger~\cite{BER} and
Berger, Drewitz, Ram\'irez~\cite{BERDRERAM}. In~\cite{BERDRERAM}, it is
conjectured that in dimensions $d\geq 2$, a RWRE which is transient in all
directions $l$ out of an open subset $U\subset\mathbb{S}^{d-1}$ is
ballistic (for an i.i.d uniformly elliptic environment).

Turning to ballistic behavior in the perturbative
regime, Sznitman shows in~\cite{SZ4} that for $0<\eta<5/2$ in dimension $d=3$
or for $0<\eta<3$ in dimensions $d\geq 4$, there exists $\e_0 = \e_0(d,\eta)$
such that if $\cone(\e)$ is fulfilled for some $\e\leq \e_0$ and the mean
local drift under the static measure satisfies
\begin{equation*}
\pE[d(0,\omega)\cdot e_1] > \left\{\begin{array}{l@{\quad }l} 
\e^{5/2-\eta} &\mbox{if } d=3\\
\e^{3-\eta} &\mbox{if }d\geq 4
\end{array},\right.\quad\mbox{where } d(0,\omega) = \sum_{|e|=1}e\omega_0(e),
\end{equation*}
then the RWRE is ballistic in direction $e_1$, i.e. $v\cdot e_1\neq
0$. Moreover, a functional limit theorem holds under $P$. In~\cite{BSZ}, Bolthausen,
Sznitman and Zeitouni consider RWRE in
dimensions $d\geq 6$ where the projection onto at least five components
behaves as simple random walk. Among other things, examples are constructed
under $\cone(\e)$ for which $\pE[d(0,\omega)] \neq 0$, but $v= 0$ $(d\geq
7)$, and a quenched invariance principle is proved when $d\geq 15$. On the
other hand, it can happen that $\pE[d(0,\omega)] = 0$ but $v\neq 0$. As
a further remarkable result of~\cite{BSZ}, it can even happen that $0\neq
v=-c\pE[d(0,\omega)]$ for some $c>0$, which exemplifies that
the environment acts on the path of the walk in a highly nontrivial
way. Large deviations of $X_n/n$ are studied in Varadhan~\cite{VAR}.

Concerning non-ballistic behavior, much is known for the class of {\it balanced} RWRE
when $\pP(\omega_0(e_i) = \omega_0(-e_i))=1$ for all $i=1,\ldots,d$. One
first notices that the walk is a martingale, which readily leads to limit
speed zero. Employing the method of environment viewed from the particle,
Lawler proves in~\cite{LAWbalanced} that for $\pP$-almost all $\omega$,
$X_{\lfloor n\cdot\rfloor}/\sqrt{n}$ converges in
$\Prw_{0,\omega}$-distribution to a non-degenerate Brownian motion with
diagonal covariance matrix. Moreover, the RWRE is recurrent in dimension
$d=2$ and transient when $d\geq 3$, see~\cite{ZT-StF}. Recently, within the
i.i.d. setting, diffusive behavior has been shown in the mere elliptic case
by Guo and Zeitouni~\cite{GUOZT}
and in the non-elliptic case by Berger and Deuschel~\cite{BERDEU}.

Our study of random walks in random environment in the perturbative regime
under the isotropy condition $\czero$ aims at a quenched central limit
theorem, showing that in dimensions $d\geq 3$, the RWRE is asymptotically
Gaussian, on $\pP$-almost all environments $\omega$. Such an invariance
principle has already been shown by Bricmont and Kupiainen~\cite{BK}, who
introduced condition $\czero$. However, it is of interest to find a
self-contained new proof. A continuous counterpart of
this model, isotropic diffusions in a random environment which are small
perturbations of Brownian motion, has been investigated by Sznitman and
Zeitouni in~\cite{SZZ}. They prove transience and a full quenched
invariance principle in dimensions $d\geq 3$.

\subsection{Open problems for our model and ongoing work}
As we already pointed out above, with respect to a central limit theorem
one still needs to find ways to handle large times, which are in a certain
sense excluded by Assumption {\bf A2}. In this direction, a more complete
picture of exit laws could prove helpful, including sharper estimates for
the appearance of balls with an atypical exit measure. A further task is to
combine space and time estimates in the right way.

In the direction of a fully perturbative theory it would be desirable to
replace the isometry condition $\czero$ by the requirement that $\mu$ is
just invariant under reflections mapping a unit vector to its inverse. Then
the RWRE exit law from a ball should be close to that of some
$d$-dimensional symmetric random walk. The relaxed condition on $\mu$
would, for example, include the class of walks that are balanced in one
coordinate direction, where time can be controlled much easier. This is
work in progress. 

Another open problem is the case of small isotropic perturbations in
dimension $d=2$.  One expects diffusive behavior as in dimensions $d\geq
3$, but there is no rigorous result yet. In principle, one might try to
follow a similar multiscale approach for the exit measures as it is
presented below. But the same perturbation argument shows that unlike
dimensions $d\geq 3$, the disorder does not contract in leading
order. Therefore, one has to look closer at higher order terms. While for
$d\geq 3$, the nonlinear terms in the perturbation expansion for the
Green's function can be estimated in a somewhat crude way once the right
scales are found, it seems that in dimension $d=2$, at least terms up to
order three have to be carefully taken into account.  \newpage

   \section{Basic notation and main results}
\label{prel}
\subsection{Basic notation}
Our purpose here is to cover the most relevant notation which will be used
throughout this text. Further notation will be introduced later on when needed.
\subsubsection{Sets and distances}
We have $\mathbb{N} = \{0,1,2,3,\ldots\}$ and $\mathbb{R}_+ =
\{x\in\mathbb{R} : x \geq 0\}$. For a set $A$, its complement is denoted by
$A^c$. If $A\subset \mathbb{R}^d$ is measurable and non-discrete, we write
$|A|$ for its $d$-dimensional Lebesgue measure. Sometimes, $|A|$ denotes
the surface measure instead, but this will be clear from the context. If
$A\subset \mathbb{Z}^d$, then $|A|$ denotes its cardinality.

For $x\in\mathbb{R}^d$, $|x|$ is the Euclidean norm. If $A, B \subset
\mathbb{R}^d$, we set $\dist(A,B) = \inf\{|x-y| : x\in A,\; y\in B\}$ and
$\diam(A) = \sup\{|x-y| : x,y \in A\}$. Given $L > 0$, let
$V_L=\{x\in\mathbb{Z}^d : |x| \leq L\}$, and for $x\in\mathbb{Z}^d$,
$V_L(x) = V_L + x$. For Euclidean balls in $\mathbb{R}^d$ we write
$C_L=\{x\in \mathbb{R}^d : |x| < L\}$ and for $x\in\mathbb{R}^d$,
$C_L(x) = x + C_L$.

If $V\subset \mathbb{Z}^d$, then $\partial V =\{x\in V^c\cap\mathbb{Z}^d
: \dist(\{x\},V) = 1\}$ is the outer boundary, while in the case of a
non-discrete set $V\subset\mathbb{R}^d$, $\partial V$ stands for the usual
topological boundary of $V$ and $\overline{V}$ for its closure. For
$x\in\overline{C}_L$, we set $\dL(x) = L-|x|$. Finally, for $0\leq
a<b\leq L$, the ``shell'' is defined by
\begin{equation*}
\sh_L(a,b)=\{x\in V_L : a\leq \dL(x) < b\},\ \ \sh_L(b) = \sh_L(0,b).
\end{equation*}
\subsubsection{Functions}
If $a,b$ are two real numbers, we set $a\wedge b = \min\{a,b\}$, $a\vee
b=\max\{a,b\}$. The largest integer not greater than $a$ is denoted by
$\lfloor a\rfloor$. As usual, set $1/0 =\infty$. For us, $\log$ is the
logarithm to the base $e$. For $x,z\in\mathbb{R}^d$, the Delta function
$\delta_{x}(z)$ is defined to be equals one for $z=x$ and zero
otherwise. If $V\subset \mathbb{Z}^d$ is a set, then $\delta_V$ is the
probability distribution on the subsets of $\mathbb{Z}^d$ satisfying
$\delta_V(V') = 1$ if $V'=V$ and zero otherwise.

Given two functions $F,G :
\mathbb{Z}^d\times\mathbb{Z}^d\rightarrow\mathbb{R}$, we write $FG$ for
the (matrix) product $FG(x,y) = \sum_{u\in\mathbb{Z}^d}F(x,u)G(u,y)$,
provided the right hand side is absolutely summable. $F^k$ is the $k$th
power defined in this way, and $F^0(x,y) = \delta_x(y)$. $F$ can also
operate on functions $f:\mathbb{Z}^d\rightarrow\mathbb{R}$ from the
left via $Ff(x) = \sum_{y\in\mathbb{Z}^d}F(x,y)f(y)$. 

We use the symbol $1_W$ for the indicator
function of the set $W$. By an abuse of notation, $1_W$ will also denote the kernel
$(x,y)\mapsto 1_W(x)\delta_x(y)$.  If $f:\mathbb{Z}^d\rightarrow
\mathbb{R}$, $||f||_1 = \sum_{x\in\mathbb{Z}^d}|f(x)| \in [0,\infty]$ is
its $L^1$-norm. When $\nu : \mathbb{Z}^d\rightarrow \mathbb{R}$ is a
(signed) measure, $||\nu||_1$ is its total variation norm.

Let $U\subset\mathbb{R}^d$ be a bounded open set, and let
$k\in\mathbb{N}$. For a $k$-times continuously differentiable function
$f:U\rightarrow\mathbb{R}$, that is $f\in C^k(U)$, we define
for $i=0,1,\ldots,k$,
\begin{equation*}
  {\left|\left|D^if\right|\right|}_{U} = \sup_{|\beta| =
    i}\sup_{U}\left|\frac{\partial^{i}}{\partial x_1^{\beta_1}\cdots\partial x_d^{\beta_d}}f\right|,
\end{equation*}
where the first supremum is over all multi-indices $\beta =
(\beta_1,\ldots,\beta_d)$, $\beta_j\in\mathbb{N}$, with
$|\beta|=\sum_{j=1}^d\beta_j$. 
Let $L>0$. Putting $U_L =\left\{x\in\mathbb{R}^d : L/2< |x| <
  2L\right\}$, we define
\begin{equation*}
  \mathcal{M}_L = \left\{\psi: U_L\rightarrow(L/10,5L),\ \psi\in
    C^4(U_L), {\left|\left|D^i\psi\right|\right|}_{U_L}\leq
    10\;\mbox{ for } i=1,\ldots,4\right\}. 
\end{equation*}
We will mostly
interpret functions $\psi\in\mathcal{M}_L$ as maps from
$U_L\cap\mathbb{Z}^d \subset\mathbb{R}^d$.
A typical function we have in mind is the constant function $\psi \equiv L$.

\subsubsection{Transition probabilities and exit distributions}
Given (not necessarily nearest neighbor) transition
probabilities $p = {(p(x,y))}_{x,y\in\mathbb{Z}^d}$, we write
$\Prw_{x,p}$ for the law of the canonical Markov chain ${(X_n)}_{n\geq 0}$
on $({(\mathbb{Z}^d)}^{\mathbb{N}},\pG)$, $\pG$
the $\sigma$-algebra generated by cylinder functions, with transition
probabilities $p$ and starting point $X_0=x$\, $\Prw_{x,p}$ -a.s. The
expectation with respect to $\Prw_{x,p}$ is denoted by $\Erw_{x,p}$.  We
will often consider the simple random walk kernel $p^{\mbox{\scriptsize RW}}(x,x\pm
e_i) = 1/(2d)$.

If $V\subset \mathbb{Z}^d$, we denote by $\tau_V = \inf\{n\geq 0 :
X_n\notin V\}$ the first exit time from $V$, with $\inf \emptyset =
\infty$, whereas $T_V = \tau_{V^c}$ is the first hitting time of
$V$. Given $x,z\in\mathbb{Z}^d$ and $p, V$ as above, we define
\begin{equation*}
\mbox{ex}_{V}(x,z;\,p) = \Prw_{x,p}\left(\tau_V=z\right).
\end{equation*}
Notice that for $x\in V^c$, $\mbox{ex}_{V}(x,z;p) = \delta_{x}(z)$.
For simple random walk, we write 
\begin{equation*}
\pi_V(x,z) = \mbox{ex}_{V}\left(x,z;p^{\mbox{\scriptsize RW}}\right).
\end{equation*}
Given $\omega\in\Omega$, we set
\begin{equation*}
\Pi_V(x,z) = \mbox{ex}_{V}(x,z;p_{\omega}).
\end{equation*}
Here, $\Pi_V$ should be understood as a {\it random} exit distribution, but
we suppress $\omega$ in the notation. 

\subsubsection{Coarse grained random walks}
In order to transfer information about both exit measures and sojourn times
from one scale to the next, we work with coarse graining schemes.

Fix once for all a probability density $\varphi\in
C^\infty(\mathbb{R}_+,\mathbb{R}_+)$ with compact support in $(1,2)$. Given
a nonempty subset $W\subset\mathbb{Z}^d$, $x\in W$ and $m_x
> 0$, the image measure of the rescaled density $(1/m_x)\varphi(t/m_x)\dt
t$ under the mapping $t\mapsto V_t(x)\cap W$ defines a probability distribution
on (finite) sets containing $x$.
If $\psi = (m_x)_{x\in W}$ is a field of positive numbers, we obtain in
this way a collection of probability distributions indexed by $x\in W$, a
{\it coarse graining scheme} on $W$. 

Now if $p=(p(x,y))_{x\in W,y\in\mathbb{Z}^d}$ is a collection of transition
probabilities on $W$, we define the coarse grained transitions belonging
to $(\psi,\,p)$ by
\begin{equation}
\label{eq:prel-cgpsi}
  p_{\psi}^{\mbox{\scriptsize CG}}(x,\cdot)  =
  \frac{1}{m_x}\int_{\mathbb{R}_+}
  \varphi\left(\frac{t}{m_x}\right)\mbox{ex}_{V_t(x)\cap W}(x,\cdot;p)\dt
  t,\quad x\in W.
\end{equation}
If $p =p^{\mbox{\scriptsize RW}}$, we write $\ph_{\psi}$ instead of
$p_{\psi}^{\mbox{\scriptsize CG}}$. Note that for every choice of $W$ and
$\psi$, $\ph_{\psi}$ defines a probability kernel. 

For the motion in the ball $V_L$, we use a particular field $\psi$,
which we describe in Section~\ref{smoothbad-cgs}. There, we will also
introduce a coarse grained RWRE transition kernel.

\subsubsection{Further notation and abbreviations}
For simplicity, we set $\Prw_x=\Prw_{x,p^{\mbox{\tiny RW}}}$,
$\Erw_x=\Erw_{x,p^{\mbox{\tiny RW}}}$. Given transition probabilities
$p_\omega$ coming from an environment $\omega$, we use the notation
$\Prw_{x,\omega}$, $\Erw_{x,\omega}$.
In order to avoid double indices, we usually write $\pi_L$ instead of
$\pi_V$, $\Pi_L$ for $\Pi_V$ and $\tau_L$ for $\tau_{V}$ if $V=V_L$ is the
ball of radius $L$ around zero.

Many of our quantities will be indexed by both $L$ and $r$, where $r$ is an
additional parameter. While we always keep the indices in the statements,
we normally drop both of them in the proofs. 
We will often use the abbreviations $\dist(y,B)$ for $\dist(\{y\},B)$, $T_x$ for $T_{\{x\}}$ and
$\pP\left(A;\, B\right)$ for $\pP\left(A\cap B\right)$.

\subsubsection{Some words about constants, $O$-notation and large $L$ behavior}
All our constants are positive. They only depend on the dimension $d\geq 3$
unless stated otherwise. In particular, they do {\it not} depend on $L$, on
$\omega$ or on any point $x\in\mathbb{Z}^d$, and they are also independent
of the parameter $r$ which will be introduced in
Section~\ref{smoothbad}. 

We use $C$ and $c$ for generic positive constants
whose values can change in different expressions, even in the same line. If
we use other constants like $K, C_1, c_1$, their values
are fixed throughout the proofs. Lower-case constants usually indicate
small (positive) values.

Given two functions $f,g$ defined on some subset of $\mathbb{R}$, we write
$f(t) = O\left(g(t)\right)$ if there exists a positive $C > 0$ and a real
number $t_0$ such that $|f(t)| \leq C |g(t)|\mbox{ for }t\geq t_0.$

If a statement holds for ``$L$ large (enough)'', this means that there exists
$L_0>0$ depending only on the dimension
such that the statement is true for all $L\geq L_0$. This applies
analogously to the expressions ``$\delta$
(or $\e$) small (enough)''.

The reader should always keep in mind that we are interested in asymptotics
when $L\rightarrow\infty$ and $\e$ is a (arbitrarily) small positive
constant. Even though some of our statements are valid only for large $L$
and $\e$ sufficiently small, we do not mention this every time.

\subsection{Main results on exit laws}
\label{mainresults-exitmeasures}
We still need some notation. For $x\in\mathbb{Z}^d$, $t>0$ and
$\psi:\partial V_t(x)\rightarrow (0,\infty)$ define
\begin{equation*}
  D_{t,\psi}(x) =  \left|\left|\left(\Pi_{V_t(x)}-\pi_{V_t(x)}\right)\ph_{\psi}(x,\cdot)\right|\right|_1 ,
\end{equation*}
\begin{equation*}
  D_{t}(x) = \left|\left|\left(\Pi_{V_t(x)}-\pi_{V_t(x)}\right)(x,\cdot)\right|\right|_1.
\end{equation*}
If $\psi \equiv m$ is constant, we write $D_{t,m}$ instead of
$D_{t,\psi}$. We usually drop $x$ from the notation if $x=0$. Further, let
\begin{equation*}
D_{t,\psi}^{\ast} =\sup_{x\in V_{t/5}}\left|\left|\left(\Pi_{V_t}-\pi_{V_t}\right)\ph_{\psi}(x,\cdot)\right|\right|_1,
\end{equation*} 
\begin{equation*}
 D_{t}^\ast = \sup_{x\in V_{t/5}}\left|\left|\left(\Pi_{V_t}-\pi_{V_t}\right)(x,\cdot)\right|\right|_1.
\end{equation*}
With $\delta > 0$, we set for $i=1,2,3$ 
\begin{equation*}
b_i(L,\psi,\delta)= \pP\left(\left\{(\log L)^{-9+9(i-1)/4} < D_{L,\psi}^\ast \leq
  (\log L)^{-9 +9i/4}\right\}\cap\left\{D_L^\ast\leq \delta\right\}\right),
\end{equation*}
and
\begin{equation*}
b_4(L,\psi,\delta)= \pP\left(\left\{D_{L,\psi}^\ast > (\log
    L)^{-3+3/4}\right\}\cup\left\{D_L^\ast> \delta\right\}\right).
\end{equation*}
The following technical condition will play a key role.

Let $\delta > 0$ and $L_1\geq 3$. We say that $\ctwo(\delta,L_1)$
holds if for all $3\leq L\leq L_1$, all $\psi\in\mathcal{M}_L$, $i=1,2,3,4$,
\begin{equation*}
b_i(L,\psi,\delta)\leq\frac{1}{4}\exp\left(-\left((3+i)/4\right)(\log
  L)^2\right).
\end{equation*}
Notice that if $\ctwo(\delta,L_1)$ is satisfied, then for any $3\leq L\leq L_1$ and any
$\psi\in\mathcal{M}_L$,
\begin{equation*}
\pP\left(\left\{D_{L,\psi}^\ast > (\log L)^{-9}\right\}\cup\left\{D_L^\ast >
\delta\right\}\right) \leq \exp\left(-(\log L)^2\right).
\end{equation*}
We can now formulate our results. Proposition~\ref{def:main-prop},
Theorem~\ref{def:main-theorem} and Corollary~\ref{def:main-transience} are
in a similar form already present in~\cite{BZ}. See also our discussion in the
introduction.

The main technical statement is
\begin{proposition}
  \label{def:main-prop}
  Assume $\czero$. There exists $\delta_0 > 0$ such that for any $\delta\in
  (0,\delta_0]$, there exists $\e_0 = \e_0(\delta) > 0$ with the
  following property: If $\e \leq \e_0$ and $\cone(\e)$ holds, then
  \begin{enumerate}
  \item There exists $L_0 = L_0(\delta)$ such that for $L_1\geq L_0$,
    \begin{equation*}
      \ctwo\left(\delta,L_1\right) \Rightarrow \ctwo\left(\delta,L_1(\log L_1)^2\right).  
    \end{equation*}
  \item There exist sequences $l_n$, $m_n\rightarrow \infty$ such that if
    $L_1\geq l_n$ and $L_1\leq L \leq L_1(\log L_1)^2$, $m\geq m_n$,
    \begin{equation*}
      \ctwo(\delta,L_1) \Rightarrow \left(\pP\left(D_{L,m}^\ast >
        1/n\right) \leq \exp\left(-(\log L)^2\right)\right).
    \end{equation*}
  \end{enumerate}
\end{proposition}
As a direct consequence, we get
\begin{theorem}[$d\geq 3$]
Assume $\czero$. 
\label{def:main-theorem}
    \begin{enumerate}
    \item There exists $\delta_0>0$ such that for any $\delta\in
      (0,\delta_0]$, there exists $\e_0 = \e_0(\delta) > 0$ with the
      following property: If $\e\leq \e_0$ and $\cone(\e)$ holds, then for
      all $L\geq 1$,
      \begin{equation*}
        \pP\left(D_L^\ast>
          \delta\right) \leq \exp\left(-(\log L)^2\right).
      \end{equation*}
    \item There exists $\e_0 > 0$ such that if $\cone(\e)$ is
      satisfied for some $\e\leq \e_0$, then for any $\eta > 0$, we can find
      $L_\eta$ and a smoothing radius $m_\eta$ such that
      for $m\geq m_\eta$, $L\geq L_\eta$,
      \begin{equation*}
        \pP\left(D_{L,m}^\ast
          > \eta\right) \leq \exp\left(-(\log L)^2\right).
      \end{equation*}
    \end{enumerate}
\end{theorem}
\begin{remark}
  (i) In particular, part (i) of Proposition~\ref{def:main-prop} tells us
  that if $\delta \leq \delta_0$, then $\ctwo(\delta,L)$ holds for all
  large $L$, provided $\cone(\e)$ is fulfilled for $\e$ small enough,
  depending only on $\delta$ (and the dimension). This follows immediately from the fact that
  given any $\delta > 0$ and any $L_1\geq 3$, we can always find
  $\e > 0$ such that $\cone(\e)$ implies $\ctwo(\delta,L_1)$.\\
  (ii) As an easy consequence of part (ii) of the theorem, if
  one increases the smoothing scale with $L$, i.e. if $m = m_L \uparrow
  \infty$ (arbitrary slowly) as $L\rightarrow \infty$, then
\begin{equation*}
D_{L,m_L}^\ast\rightarrow 0\quad\pP\mbox{-almost surely}.
\end{equation*}
(iii) One could define the smoothing kernel $\ph_{\psi}$
differently. However, our particular form is useful for the induction
procedure.
\end{remark}
Our methods allow us to compare the exit measures in a more local way. For
positive $t$ and $z\in\partial V_L$, let $W_t(z)=V_t(z)\cap \partial V_L$. Then
$W_t(z)$ contains on the order of $t^{d-1}$ points. The center 
$z\in\partial V_L$ will play no particular role, so we drop it from the notation.
If we choose our parameters according to
Theorem~\ref{def:main-theorem} (i), we have good control over
$\Pi_L(x,W_t)$ in terms of $\pi_L(x,W_t)$, provided $x$ has a distance of
order $L$ from the boundary and $t$ is sufficiently large. For the
statement of the following theorem, we pick $\delta\in (0,\delta_0]$ and
$L_0(\delta)$ according to Proposition~\ref{def:main-prop}, and choose the
perturbation $\e\leq \e_0$ small enough such that $\cone(\e)$ implies
$\ctwo(\delta,L_0)$ (and then $\ctwo(\delta,L)$ for all $L\geq L_0$,
according to the proposition).

\begin{theorem}
\label{local-thm-exitmeas}
Assume $\czero$. In the setting just described, if $\cone(\e)$ is
fulfilled, then for $L\geq L_0$, there exists an event $A_L\in\pF$ with
$\pP(A^c_L) \leq \exp(-(1/2)(\log L)^2)$ such that on $A_L$,
the following holds true. If $0<\eta <1$ and $x\in V_{\eta L}$, then
\begin{enumerate}
\item For $t\geq L/(\log L)^{15}$ and every set $W_t$ as above, there
  exists $C=C(\eta)$ with
\begin{equation*}
\Pi_L(x,W_t) \leq C\pi_L(x,W_t).
\end{equation*}
\item For $t\geq L/(\log L)^{6}$,
\begin{equation*}
  \Pi_L(x,W_t) = \pi_L(x,W_t)\left(1+O\left((\log L)^{-5/2}\right)\right).
\end{equation*}
Here, the constant in the $O$-notation depends only on $d$ and $\eta$.
\end{enumerate}
\end{theorem}
 From Proposition~\ref{def:main-prop}, we also obtain transience of the
 RWRE.
 \begin{corollary}[Transience]
   \label{def:main-transience}
   Assume $\czero$. There exist $\e_0$, $\rho> 0$ such that if $\cone(\e)$ is satisfied for
   some $\e\leq \e_0$, then for $\pP$-almost all $\omega\in\Omega$, there
   exists $m_0=m_0(\omega)\in\mathbb{N}$ with the following property: For integers $m\geq
   m_0$ and $k\geq 1$, 
\begin{equation}
\label{eq:main-transience-eq}
\sup_{x : |x| \geq
  \rho^{m+k}}\Prw_{x,\omega}\left(T_{V_{\rho^m}}<\infty\right)\leq \left(2/3\right)^k.
\end{equation}
In particular, the RWRE $(X_n)_{n\geq 0}$ is transient.
 \end{corollary}

\subsection{Main results on mean sojourn times}
\label{mainresults-meantime}
For the times, we propagate a condition similar to  
$\ctwo(\delta,L)$. In this regard, we first introduce a monotone increasing
function which will limit the normalized
mean sojourn time in the ball. Let $0<\eta<1$, and define $f_\eta:
\mathbb{R}_+\rightarrow \mathbb{R}_+$ by setting
\begin{equation*}
  f_\eta(L)= \frac{\eta}{3}\sum_{k=1}^{\lceil \log L\rceil}k^{-3/2}.
\end{equation*}
Note that
$\eta/3\leq f_\eta(L)< \eta$ and therefore $\lim_{\eta\downarrow 0}\lim_{L\rightarrow\infty}f_\eta(L) = 0$.

Recall that $\Erw_0$ is the expectation with respect to
simple random walk starting at the origin. We say that $\ctime(\eta,L_1)$ holds, if for all $3\leq L \leq L_1$,
\begin{equation*}
\pP\left(\Erw_{0,\omega}\left[\tau_L\right] \notin \left[1-f_\eta(L),\,
      1+f_\eta(L)\right]\cdot\Erw_0\left[\tau_L\right]\right) \leq
L^{-6d}.
\end{equation*}
Our main technical result is
\begin{proposition}
\label{def:times-main-prop}
  Assume $\czero$ and {\bf A2}, and let $0<\eta<1$. There exists
  $\e_0=\e_0(\eta)> 0$ with the following
  property: If $\e \leq \e_0$ and $\cone(\e)$ holds, then
\begin{enumerate}
\item There exists $L_0=L_0(\eta)>0$ such that for $L_1\geq L_0$,
\begin{equation*}
\ctime(\eta,L_1)\Rightarrow \ctime(\eta,L_1(\log L_1)^2).
\end{equation*}
\item
\begin{equation*}
\lim_{L\rightarrow\infty}L^d\,\pP\left(\sup_{x: |x| \leq L^{3}}\sup_{y\in
      V_L(x)}\Erw_{y,\omega}\left[\tau_{V_L(x)}\right] \notin
  [1-\eta,1+\eta]\cdot L^2\right) = 0.
\end{equation*}
\end{enumerate}
\end{proposition}
By Borel-Cantelli and the Markov property, we immediately have
\begin{corollary}[Quenched moments]
 \label{def:times-cormoments}
 In the framework of Proposition~\ref{def:times-main-prop}, for $k\in \mathbb{N}$ and
 $\pP$-almost all  $\omega\in\Omega$,
\begin{equation*}
 \lim_{L\rightarrow\infty}\left(\sup_{x: |x| \leq L^{3}}\sup_{y\in V_L(x)}\Erw_{y,\omega}\left[\tau^k_{V_L(x)}\right]/L^{2k} \right)\leq 2^kk!\,.
\end{equation*}
\end{corollary}
The bounds on the quenched moments for $k=2$ are useful to prove
\begin{theorem}
\label{def:times-thm2}
Assume $\czero$ and {\bf A2}. Given $0<\eta < 1$, one can find $\e_0=\e_0(\eta) >
0$ such that if $\cone(\e)$ is satisfied for some $\e\leq \e_0$, then the
following holds: There exist $D_1$, $D_2$ $\in [1-\eta, 1+\eta]$ such that for
$\pP$-almost all $\omega$,
\begin{eqnarray*}
  \liminf_{L\rightarrow\infty}\left(\sup_{x: |x| \leq L^{3}}\sup_{y\in
      V_L(x)}\Erw_{y,\omega}\left[\tau_{V_L(x)}\right]/L^2 \right)&=& D_1,\\
  \limsup_{L\rightarrow\infty}\left(\sup_{x: |x| \leq L^{3}}\sup_{y\in
      V_L(x)}\Erw_{y,\omega}\left[\tau_{V_L(x)}\right]/L^2 \right) &=&D_2.
\end{eqnarray*}
\end{theorem}
\begin{remark}
  (i) Given $\eta$ and $L_1$, we can always guarantee (by making $\e$ smaller if
  necessary) that $\cone(\e)$ implies
  $\ctime(\eta,L_1)$.\\
 (ii) The factor $L^{-6d}$ in the definition of condition
 $\ctime(\eta,L_1)$, the factor $L^d$ 
  and also the choice of $L^3$ inside the probability in the statement of
  Proposition~\ref{def:times-main-prop} (ii) are connected to
  Assumption {\bf A2}. If, for instance, one could prove that for some
  $\alpha > 1$ and large $L$,
\begin{equation*}
\pP\left(\Erw_{0,\omega}\left[\tau_L\right]>(\log L)^4L^2\right)\leq\exp\left(-(\log L)^{\alpha}\right), 
\end{equation*}
then Proposition~\ref{def:times-main-prop} (ii) would hold with $L^d$ replaced by $L^r$ for
every $r\in\mathbb{N}$.\\
(iii) In the last theorem, we strongly believe that $D_1=D_2$.
\end{remark} 

\subsection{Perturbation expansion}
Our approach of comparing the behavior of the RWRE in space and time with
that of simple random walk is based on a perturbation argument. Namely, the resolvent
equation allows us to express Green's functions of the RWRE in terms of
ordinary Green's functions. More generally, let $p =
\left(p(x,y)\right)_{x,y\in\mathbb{Z}^d}$ be a family of finite range
transition probabilities on $\mathbb{Z}^d$, and let $V\subset\mathbb{Z}^d$
be a finite set. The corresponding Green's kernel or Green's function for 
$V$ is defined by
\begin{equation*}
g_V(p)(x,y) = \sum_{k=0}^\infty \left(1_Vp\right)^k(x,y).
\end{equation*}
The connection with the exit measure is given by the fact that for $z\notin
V$, we have
\begin{equation*}
g_V(p)(\cdot,z) = \mbox{ex}_{V}(\cdot,z;p).
\end{equation*}
Now write $g$ for $g_V(p)$ and let $P$ be another transition
kernel with corresponding Green's function $G$ for $V$. With $\Delta=
1_V\left(P-p\right)$, we have by the resolvent equation 
\begin{equation}
\label{eq:prel-pbe1}
G -g = g\Delta G = G\Delta g.
\end{equation}
In order to get rid of $G$ on the right hand side, we
iterate~\eqref{eq:prel-pbe1} and obtain
\begin{equation}
\label{eq:prel-pbe2}
G -g = \sum_{k=1}^\infty \left(g\Delta\right)^kg, 
\end{equation}
provided the infinite series converges, which will always be the case in our
setting, due to $\cone(\e)$ and $V$ being finite.
A modification of~\eqref{eq:prel-pbe2} turns out to
be particularly useful. Note that by~\eqref{eq:prel-pbe2},
\begin{equation*}
G = g\sum_{k=0}^\infty\left(\Delta g\right)^k.
\end{equation*}
Replacing the rightmost $g$ by $g(x,\cdot)= \delta_x(\cdot) + 1_Vpg(x,\cdot)$
and reordering terms, we arrive at
\begin{equation}
\label{eq:prel-pbe3}
G = g\sum_{m=0}^\infty{\left(Rg\right)}^m\sum_{k=0}^\infty\Delta^k,
\end{equation}
where $R= \sum_{k=1}^\infty\Delta^kp$.

\subsection{A short reading guide}
\label{readingguide}
The key idea behind our results on exit laws from $V_L$ is to compare the
RWRE exit measure  with that of simple random walk by means of
the perturbation expansion
\begin{equation*}
\Pi_L-\pi_L = \Gh1_{V_L}(\Ph-\ph)\pi_L.
\end{equation*}
Here, $\Ph$ is a coarse grained RWRE transition kernel inside $V_L$,
$\ph$ is a coarse grained simple random walk kernel, and $\Gh =
\Gh(\Ph)$ is the Green's function associated to $\Ph$.
 
Our coarse grained transition kernels are given by exit distributions from
smaller balls inside $V_L$, and we obtain our results by transfering
inductively estimates on smaller scales to scale $L$. The coarse graining
schemes defined in Section~\ref{smoothbad} determine the radii of the
smaller balls. In the bulk of $V_L$, we choose the radius $s_L
=L/(\log L)^3$, but we refine the radii when approaching the boundary. Our
schemes are parametrized by a real number $r$, which determines the
distance to the boundary $\partial V_L$ at which the refinement stops.  We
choose $r$ equal to $r_L=L/(\log L)^{15}$ for the estimates involving a
global smoothing, and equals a large constant for the non- or locally
smoothed estimates.

Besides the coarse graining schemes, Section~\ref{smoothbad} introduces the
concept of ``good'' and ``bad'' points and so-called goodified Green's
functions. Roughly speaking, we call a point $x\in V_L$ {\it
  good} if the exit measure on such a smaller ball around $x$ is close to
the exit measure of simple random walk, in both a smoothed and non-smoothed
way. If inside $V_L$ all points are good, then the estimates on smaller
balls can be transferred to a (globally smoothed) estimate on the larger
ball $V_L$ (Lemma~\ref{def:lemma-goodpart}), using some averaging argument
and an exponential inequality.

But {\it bad} points can appear, and in fact we have to distinguish four
different levels of badness (Section~\ref{s1}).  When bad points are
present, it is convenient to ``goodify'' the environment, that is to
replace bad points by good ones. This important concept is first explained
in Section~\ref{smoothbad} and then further developed in
Section~\ref{super}. However, for the globally smoothed estimate, due to
the additional smoothing step we only have to deal with the case where all
bad points are enclosed in a comparably small region - two or more such
regions are too unlikely (Lemma~\ref{def:smoothbad-lemmamanybad}). Some
special care is required for the worst class of bad points in the interior
of the ball. For environments containing such points, we slightly modify
the coarse graining scheme inside $V_L$, as described in
Section~\ref{super-cgmod}. In Lemma~\ref{def:lemma-badpart}, we prove the
smoothed estimates on environments with bad points and show that the degree
of badness decreases by one from one scale to the next.

Concerning exit measures where no or only a local last smoothing step is
added (Section~\ref{nonsmv-exits}, Lemmata~\ref{def:nonsmv-lemma1}
and~\ref{def:nonsmv-lemma2}, respectively), bad points near the boundary of
$V_L$ are much more delicate to handle, since we have to take into account
several possibly bad regions. However, they do not occur too frequently
(Lemma~\ref{def:smoothbad-lemmabbad}) and can be controlled by capacity arguments.

All these estimates require precise bounds on coarse grained Green's
functions, which are developed in Section~\ref{super}. Basically, we show
that on environments with no bad points, the coarse grained RWRE Green's
function for the ball can be estimated from above by the analogous quantity
coming from simple random walk. 

Section~\ref{estimates} is devoted to technical bounds on hitting
probabilities of both simple random walk and Brownian motion, and to
difference estimates of smoothed exit measures. The reason for working
sometimes with Brownian motion instead of a random walk is of technical
nature - some estimates are easier to prove for the former, as for example
Lemma~\ref{def:est-phi} (iii). They can then be transferred to random walks
via coupling arguments provided in the appendix.  

The statements from
Sections~\ref{smv-exits} and~\ref{smv-exits} are finally used in
Section~\ref{proofmain} to prove the main results on exit measures,
including the proof of transience of the RWRE.

The object of interest in Section~\ref{times} is the mean sojourn time of
the RWRE in the ball $V_L$. Employing the Markov property, we represent
this quantity as a convolution of a coarse grained RWRE Green's
function $\Gh$ and mean sojourn times in smaller balls $\Lambda_L(y)$,
\begin{equation*}
  \Erw_{x,\omega}\left[\tau_L\right] = \sum_{y\in V_L}\Gh(x,y)\Lambda_L(y).
\end{equation*}
Again, multiscale analysis is used to transport time estimates on a smaller
to a bigger scale. It turns out that we need control over space and
time on the next two lower levels. This requires a stronger
notion of good and bad points concerning both spatial and temporal
behavior. In Section~\ref{proofmaintimes}, we 
prove our main results on mean sojourn times.

Finally, in the appendix we prove the main statements of
Section~\ref{estimates}, as well as a local central limit theorem for the coarse
grained simple random walk.

   \section{Coarse graining schemes and notion of badness}
\label{smoothbad}
The purpose of this section is to introduce coarse graining schemes in the
ball as well as the concept of ``good'' and ``bad'' points. Also, we prove
two estimates ensuring that we do not have to consider environments with
bad points that are widely spread out in the ball or densely packed in the
boundary region.
\subsection{Coarse graining schemes in the ball}
\label{smoothbad-cgs}
Once for all, define
\begin{equation*}
s_L = \frac{L}{(\log L)^3} \quad\mbox{and}\quad r_L= \frac{L}{(\log L)^{15}}.
\end{equation*}
We will use particular coarse graining schemes indexed by a parameter $r$,
which can either be a constant $\geq 100$, but much smaller than $r_L$, or,
in most of the cases, $r=r_L$. We fix a smooth function $h :
\mathbb{R}_+\rightarrow\mathbb{R}_+$ satisfying
\begin{equation*}
  h(x)= \left\{\begin{array}{l@{\quad \mbox{for\ }}l}
      x & x\leq \frac{1}{2}\\
      1 & x\geq 2\end{array}\right.,
\end{equation*}
such that $h$ is strictly monotone and concave on $(1/2,2)$, with first derivative bounded uniformly by $1$.
Define $h_{L,r} : \overline{C}_L \rightarrow\mathbb{R}_+$ by
\begin{equation}
\label{eq:smoothbad-hLr}
 h_{L,r}(x) = \frac{1}{20}\max\left\{s_L
  h\left(\frac{\dL(x)}{s_L}\right),\, r\right\}.
\end{equation}
Since we mostly work with $r=r_L$, we use the abbreviation $h_L = h_{L,r_L}$.
We write $\Ph_{L,r}$ for the coarse grained RWRE transition kernel
associated to $(\psi= \left(h_{L,r}(x)\right)_{x\in V_L},\,p_\omega)$,
\begin{equation*}
\Ph_{L,r}(x,\cdot) =
\frac{1}{h_{L,r}(x)}\int_{\mathbb{R}_+}
  \varphi\left(\frac{t}{h_{L,r}(x)}\right)\Pi_{V_t(x)\cap V_L}(x,\cdot)\dt
  t,
\end{equation*}
and $\ph_{L,r}$ for that coming from simple random walk, where $\Pi$ is
replaced by $\pi$.  For convenience, we set $\Ph_{L,r}(x,\cdot)
=\ph_{L,r}(x,\cdot) = \delta_x(\cdot)$ for $x\in\mathbb{Z}^d\backslash
V_L$. Notice that by the strong Markov property, the exit measures from the
ball $V_L$ remain unchanged under these transition kernels, i.e. 
\begin{equation*}
\mbox{ex}_{V_L}\left(x,\cdot;\Ph_{L,r}\right) = \Pi_L(x,\cdot)
\quad\mbox{and}\quad
\mbox{ex}_{V_L}\left(x,\cdot;\ph_{L,r}\right) = \pi_L(x,\cdot).
\end{equation*}
We denote by $\Gh_{L,r}$ the (coarse grained) RWRE Green's function coming
from $\Ph_{L,r}$, and by $\gh_{L,r}$ the Green's function from $\ph_{L,r}$,
everything in $V_L$.
\begin{remark}
  \label{def:prel-rem}
  (i) Later on, we will also work with slightly modified transition kernels
  $\Pho$ and $\pho$, which depend on the environment.  We elaborate on this in Section~\ref{super-cgmod}.\\
  (ii) Due to the lack of the last smoothing step outside $V_L$, we need to
  zoom in near the boundary in order to handle non-smoothed exit
  distributions in Section~\ref{nonsmv-exits}. The parameter $r$ allows
  us to adjust the step size in the boundary region. \\
  (ii) Note that for every choice of $r$,
\begin{equation*}
h_{L,r}(x) = \left\{\begin{array}{l@{\quad \mbox{for\ }}l}
\dL(x)/20 & x \in V_L \mbox{ with } r_L\leq\dist_L(x) \leq s_L/2\\
s_L/20 & x \in V_L \mbox{ with } \dist_L(x) \geq 2s_L
\end{array}\right..
\end{equation*}
\end{remark}
\begin{figure}
\begin{center}\parbox{5.5cm}{\includegraphics[width=5cm]{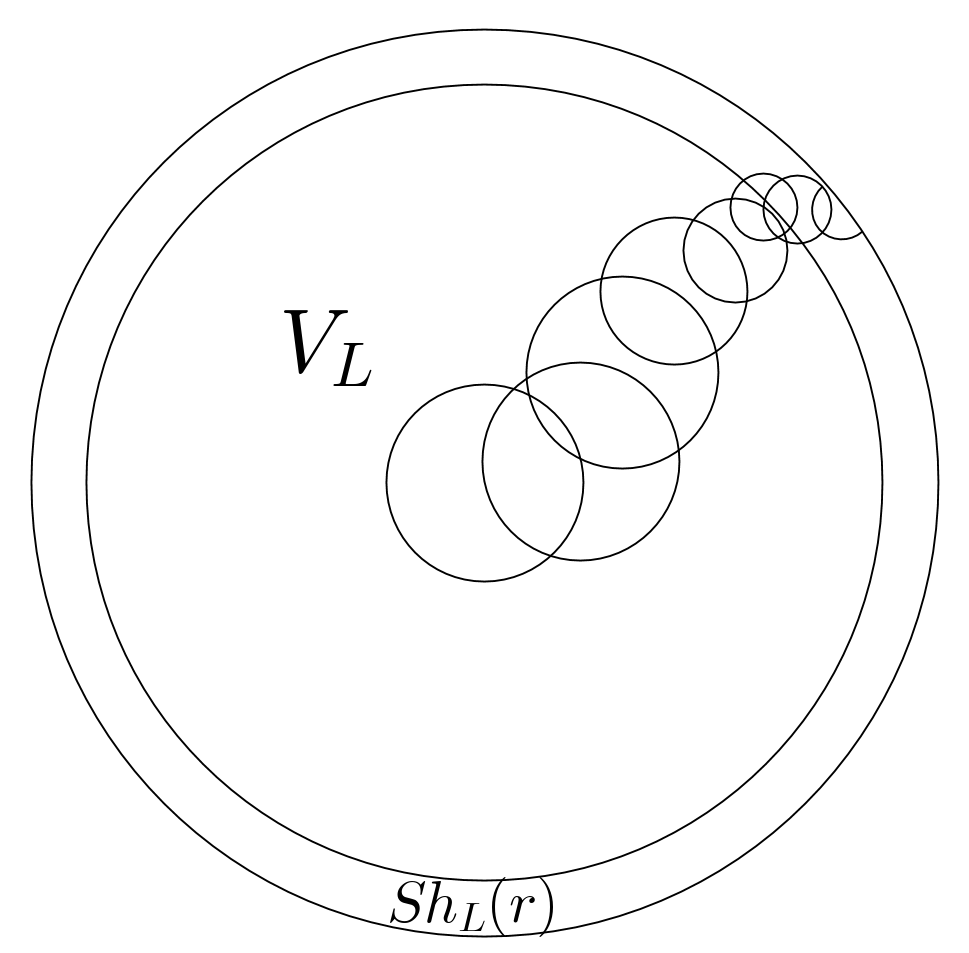}}
\parbox{9.5cm}{
\caption{The coarse graining scheme in $V_L$. In the bulk $\{x\in
  V_L:\,d_L(x) \geq 2s_L\}$, the exit
  distributions are taken from balls of radii between
  $(1/20)s_L$ and $(1/10)s_L$. When entering $\sh_L(2s_L)$, the coarse
  graining radii start to shrink, up to the boundary layer
  $\sh_L(r)$, where the exit distributions are taken from intersected balls
  $V_t(x)\cap V_L$, $t\in[(1/20)r,\,(1/10)r]$.}} 
\end{center}
\end{figure}
\subsection{Good and bad points}
\label{goodandbad}
We shall partition the grid points inside $V_L$ according to
their influence on the exit behavior. We say that a point
$x\in V_L$ is {\it good} (with respect to $L$, $\delta >0$ and $r$,
$100\leq r\leq r_L$) if  
\begin{itemize}
\item For all $t\in[h_{L,r}(x),2h_{L,r}(x)]$,
  $\left|\left|(\Pi_{V_t(x)}-\pi_{V_t(x)})(x,\cdot)\right|\right|_1 \leq \delta$.
\item If $\dL(x) > 2r$, then additionally 
\begin{equation*}
\left|\left|(\Ph_{L,r} - \ph_{L,r})\ph_{L,r}(x,\cdot)\right|\right|_1 \leq \left(\log
     h_{L,r}(x)\right)^{-9}.
\end{equation*}
\end{itemize}
A point $x\in V_L$ which is not good is called {\it bad}. 
We denote by $\badP_{L,r} = \badP_{L,r}(\omega)$ the set of all bad points
inside $V_L$ and write $\badP_L = \badP_{L,r_L}$ for short. Furthermore,
set $\badP_{L,r}^\partial = \badP_{L,r}\cap\sh_L(r_L)$
and $\badPo_{L,r} =
\badP_{L,r}\cup\badP_{L} = \badP_{L,r}^\partial\cup\badP_L$.
Of course, the set of bad points depends also on $\delta$, but
we do not indicate this. 
\begin{remark}
  (i) For the coarse graining scheme associated to $r=r_L$, we have by
  definition $\badPo_{L,r_L} = \badP_L$. When performing the 
  estimates in Section~\ref{nonsmv-exits}, we work with constant $r$. In
  this case, $\badPo_{L,r}$ can contain more points than $\badP_L$.\\
  (ii) Assume $L$ large. If $x\in V_L$ with $\dist_L(x) > 2r$, then the
  function $h_{L,r}(x+\cdot)$, defined in~\eqref{eq:smoothbad-hLr}, lies in
  $\mathcal{M}_t$ for each $t\in[h_{L,r}(x),2h_{L,r}(x)]$. Thus, for all
  $x\ \in V_L$, we can use $\ctwo(\delta,L_1)$ to control the event
  $\{x\in\badP_{L,r}\}$, provided $2h_{L,r}(x) \leq L_1$. We make use of
  this in Lemma~\ref{def:smoothbad-lemmamanybad}.
\end{remark}
\subsubsection{Goodified transition kernels}
It is difficult to obtain estimates for the RWRE in the presence of bad points.
For all environments, we therefore introduce ``goodified'' transition kernels
$\Phg_{L,r}$,
\begin{equation}
  \label{eq:smoothbad-goodifiedkernel}
  \Phg_{L,r}(x,\cdot) = \left\{\begin{array}{l@{\ \mbox{for\ }}l}
      \Ph_{L,r}(x,\cdot) & x \in V_L\backslash \badPo_{L,r}\\
      \ph_{L,r}(x,\cdot) & x \in \badPo_{L,r}\end{array}\right..
\end{equation}
Furthermore, we write $\Ghg_{L,r}$  for the corresponding (random) Green's function. 
\subsection{Bad regions in the case $r=r_L$}
\label{s1}
The following lemma shows that with high probability, all bad points
with respect to $r=r_L$ are contained in a ball of radius $4h_L(x)$. Let
\begin{equation*}
  \mathcal{D}_L = \left\{V_{4h_L(x)}(x) : x\in V_L\right\}.
\end{equation*}
We will look at the events $\onebad= \left\{\badP_L\subset D\mbox{ for
    some } D\in\mathcal{D}_L\right\}$ and
$\manybad={\left(\onebad\right)}^c$. It is also useful to define the set
of {\it good} environments, $\good = \{\badP_L =
\emptyset\}\subset\onebad$.
\begin{lemma}
  \label{def:smoothbad-lemmamanybad}
  For large $L_1$, $\ctwo(\delta,L_1)$ implies that for $L$ with
  $L_1\leq L \leq L_1(\log L_1)^2$,
  \begin{equation*}
    \pP\left(\manybad\right) \leq
    \exp\left(-\frac{19}{10}(\log L)^2\right). 
  \end{equation*}
\end{lemma}
\begin{prooof}
Set $\Delta = 1_{V_L}(\Ph_{L,r_L} - \ph_{L,r_L})$, $\ph =
\ph_{L,r_L}$. For all $x \in V_L$ with $\dL(x) > 2r_L$, using $\frac{1}{20} r_L \leq
  h_L(x) \leq s_L\leq L_1/2$,
  \begin{eqnarray*}
    \pP\left(x\in \badP_{L}\right) &=& \pP\left(\left\{\left|\left|\Delta\ph(x,\cdot)\right|\right|_1 >(\log
        h_L(x))^{-9}\right\}\cup\left\{\left|\left|\Delta(x,\cdot)\right|\right|_1 > \delta\right\}\right)\\ 
    &\leq& 
    \pP\left(\bigcup\limits_{t\in\left[h_L(x),2h_L(x)\right]}\left\{D_{t,h_L}(x)
        > (\log h_L(x))^{-9}\right\}\cup\left\{D_{t}(x) > \delta\right\}\right)\nonumber\\    
    &\leq& Cs_L^{d}\exp\left(-\left(\log(r_L/20)\right)^{2}\right),
  \end{eqnarray*}
  and a similar estimate holds in the case $\dL(x) \leq 2r_L$.  On the event
  $\manybad$, there exist $x,y \in \badP_L$ with $|x-y| > 2h_L(x)
  +2h_L(y)$. But for such $x,y$, the events $\{x\in \badP_L\}$ and $\{y\in
  \badP_L\}$ are independent, whence for $L$ large
\begin{equation*}
  \pP\left(\manybad\right) \leq C
  L^{2d}s_L^{2d}\left[\exp\left(-(\log(r_L/20))^2\right)\right]^2\leq
  \exp\left(-(19/10)(\log L)^2\right). 
\end{equation*}
\end{prooof}
\begin{figure}
\begin{center}\parbox{5.5cm}{\includegraphics[width=5cm]{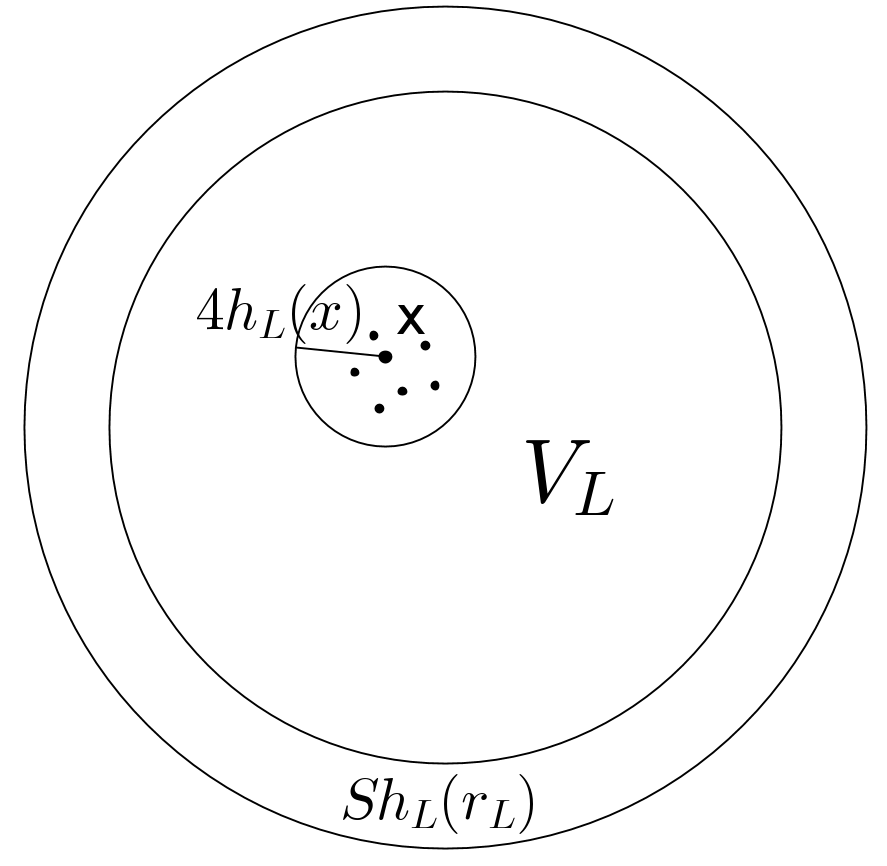}}
\parbox{9.5cm}{
\caption{On environments $\omega\in\onebad$, all bad points are 
  enclosed in a ball $V_{4h_L(x)}(x)$.}} 
\end{center}
\end{figure}
The estimate is good enough for our inductive procedure, so we only have to
deal with the case where all bad points are enclosed in a ball
$D\in\mathcal{D}_L$.  However, inside $D$ we need to look closer at the
degree of badness.  We say that $\omega\in\onebad$ is {\it bad on level}
$i$, $i=1,2,3$, if the following holds:
\begin{itemize}
\item For all $x\in V_L$, for all $t\in[h_L(x),2h_L(x)]$, $\left|\left|(\Pi_{V_t(x)} - \pi_{V_t(x)})(x,\cdot)\right|\right|_1 \leq \delta$.
\item For all $x\in V_L$ with $\dL(x) > 2r_L$, additionally 
\begin{equation*}
\left|\left|(\Ph_{L,r_L} - \ph_{L,r_L})\ph_{L,r_L}(x,\cdot)\right|\right|_1 \leq \left(\log
     h_L(x)\right)^{-9+9i/4}.
\end{equation*}
\item There exists $x\in \badP_L(\omega)$ with $\dL(x) > 2r_L$ such that
  \begin{equation*}
    \left|\left|(\Ph_{L,r_L} - \ph_{L,r_L})\ph_{L,r_L}(x,\cdot)\right|\right|_1 > \left(\log
      h_L(x)\right)^{-9+9(i-1)/4}.
  \end{equation*}
\end{itemize}
If $\omega\in\onebad$ is neither bad on level $i=1,2,3$ nor good, we call
$\omega$ {\it bad on level} $4$.  In this case, $\badP_L(\omega)$ contains ``really bad'' points.  We write $\onebadi\subset\onebad$ for
the subset of all those $\omega$ which are bad on level
$i=1,2,3,4$. Observe that
\begin{equation*}
  \onebad = \good\mathbin{\dot{\cup}}\mathbin{\mathop{\dot{\bigcup}}_{i=1}^4}\onebadi.   
\end{equation*}
On $\good$, $\Phg_{L,r_L}=\Ph_{L,r_L}$ and therefore $\Ghg_{L,r_L}=\Gh_{L,r_L}$.
\subsection{Bad regions when $r$ is a constant}
\label{s2}
When estimating the non-smoothed quantity $D_L^\ast$, we cannot stop the
refinement of the coarse graining in the boundary region $\sh_L(r_L)$. Instead, we
will choose $r$ as a (large) constant. However, now it is no longer true
that essentially all bad points are contained in one single region
$D\in\mathcal{D}_L$. For example, if $x\in V_L$ such that
$\dist_L(x)$ is of order $\log L$, we only have a bound of the form
\begin{equation*}
  \pP\left(x\in \badP_{L,r}\right) \leq \exp\left(-c{\left(\log
        \log L\right)}^2\right), 
\end{equation*}
which is clearly not enough to get an estimate as in
Lemma~\ref{def:smoothbad-lemmamanybad}.  We therefore choose a different strategy to handle bad points
within $\sh_L(r_L)$. We split the boundary region into layers of an
appropriate size and use independence to show that with high probability, bad
regions are rather sparse within those layers. Then the Green's function estimates of
Corollary~\ref{def:super-cor} will ensure that on such environments, there is a
high chance to never hit points in $\badP_{L,r}^\partial$ before leaving the ball.

To begin with the first part, fix $r$ with $r\geq r_0\geq
100$, where $r_0 = r_0(d)$ is a constant that will be chosen
below. Let $L$ be large enough such that $r < r_L$, and set $J_1 = J_1(L)
= \left\lfloor \frac{\log(r_L/r)}{\log 2}\right\rfloor +1$. We define 
layers $\Lambda_0 = \sh_L(2r)$ and $\Lambda_j = \sh_L(r2^j,r2^{j+1})$,
$1\leq j\leq J_1$. Then,
\begin{equation*}
  \sh_L(2r_L) \subset \bigcup_{0\leq j\leq J_1}\Lambda_j\subset \sh_L(4r_L).
\end{equation*}
Let $j\in\mathbb{N}$. For $k\in\mathbb{Z}$, consider the interval
$I_k^{(j)} = (kr2^j,(k+1)r2^j]\cap \mathbb{Z}$. We divide $\Lambda_j$
into subsets by setting $D_{\bf k}^{(j)} = \Lambda_j\cap
\left(I_{k_1}\times\ldots\times I_{k_d}\right)$, where ${\bf k} =
(k_1,\ldots,k_d) \in\mathbb{Z}^d$.  Denote by $\mathcal{Q}_{j,r}$ the set
of these subsets which are not empty.  Setting $N_{j,r} =
|\mathcal{Q}_{j,r}|$, it follows that
\begin{equation*}
  \frac{1}{C}{\left(\frac{L}{r2^j}\right)}^{d-1} \leq N_{j,r} \leq
  C{\left(\frac{L}{r2^j}\right)}^{d-1}. 
\end{equation*}
We say that a set $D\in\mathcal{Q}_{j,r}$ is {\it bad} if
$\badP^\partial_{L,r}\cap D \neq \emptyset$. As we want to make use of
independence, we partition $\mathcal{Q}_{j,r}$ into disjoint sets
$\mathcal{Q}_{j,r}^{(1)},\ldots, \mathcal{Q}_{j,r}^{(R)}$, such that for each
$1\leq m \leq R$ we have
\begin{itemize}
\item $\dist(D,D') > 4\max_{x\in\Lambda_j}h_{L,r}(x)$ for all $D\neq
  D'\in\mathcal{Q}_{j,r}^{(m)}$,
\item $N_{j,r}^{(m)} = \left |\mathcal{Q}_{j,r}^{(m)}\right| \geq
  \frac{N_{j,r}}{2R}$.
\end{itemize}
\begin{figure}
\begin{center}\parbox{5.5cm}{\includegraphics[width=5cm]{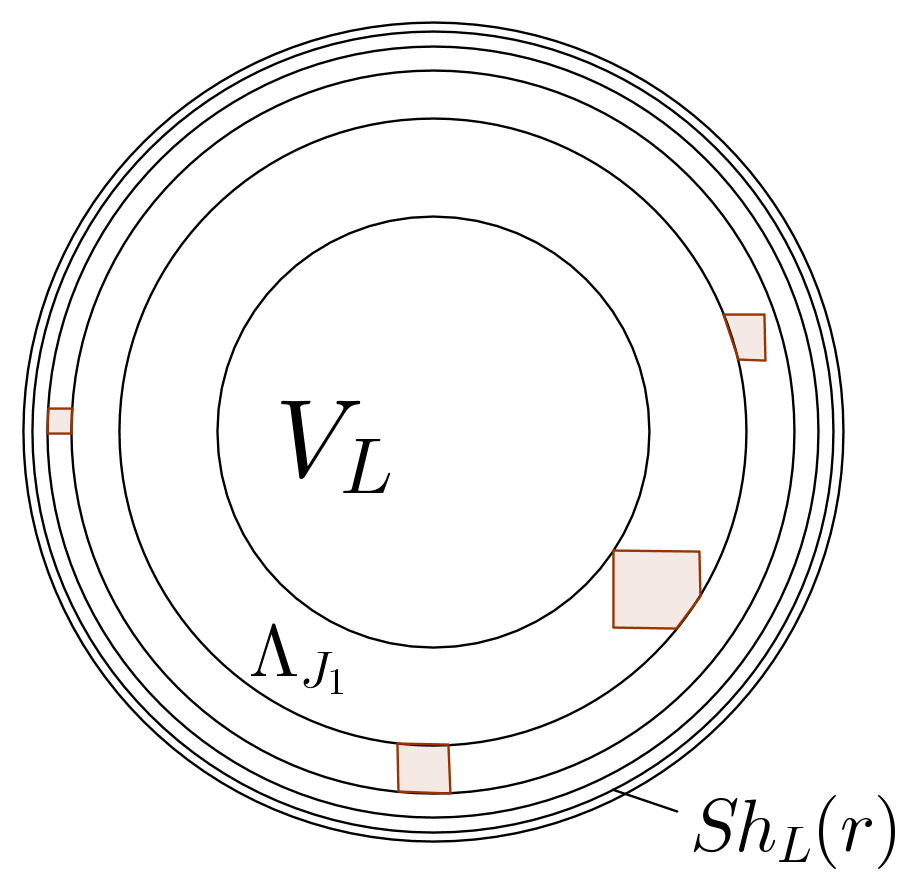}}
\parbox{8.5cm}{
\caption{The layers $\Lambda_j$, $0\leq j\leq J_1$, with $\Lambda_0=\sh_L(2r)$. Subsets $D_{\bf
    k}^{(j)}\subset\Lambda_j$ containing bad points are shaded.}} 
\end{center}
\end{figure}
Notice that $R\in\mathbb{N}$ can be chosen to depend on the dimension only.
Then the events $\{D$ is bad$\}$, $D\in\mathcal{Q}_{j,r}^{(m)}$, are
independent. Further, if $L_1\leq L \leq L_1(\log L_1)^2$, it follows that
under $\ctwo(\delta,L_1)$, 
\begin{equation*}
  \pP\left(D\mbox{ is bad}\right) \leq
  C(r2^j)^{2d}\exp\left(-\left(\log(r2^j/20)\right)^2\right)
  \leq \exp\left(-\left(\log r + j\right)^{5/3}\right)= p_{j,r}, 
\end{equation*}
for all $r \geq r_0$ and $j\in\mathbb{N}$, if $r_0$ is big enough. Let
$Y_{j,r}$ and $Y_{j,r}^{(m)}$ be the number of bad sets in
$\mathcal{Q}_{j,r}$ and $\mathcal{Q}_{j,r}^{(m)}$, respectively. For $r\geq
5$, we have $p_{j,r} \leq (\log r +j)^{-3/2}\leq 1/2$.  A standard large
deviation estimate for Bernoulli random variables yields
\begin{equation*}
  \pP\left(Y_{j,r}^{(m)} \geq (\log r + j)^{-3/2}N_{j,r}^{(m)}\right) \leq
  \exp\left(-N_{j,r}^{(m)}\ I\left((\log r + j)^{-3/2}\ \big|\ p_{j,r}\right)\right), 
\end{equation*}
with $I(x\ |\ p) = x\log(x/p) + (1-x)\log((1-x)/(1-p))$. By enlarging $r_0$
if necessary, we get $I\left((\log r + j)^{-3/2}\ |\ p_{j,r}\right) \geq 2
R(\log r+j)^{1/7}$ for $r\geq r_0$, whence
\begin{eqnarray*}
  \lefteqn{\pP\left(Y_{j,r} \geq (\log r + j)^{-3/2}N_{j,r}\right)}\\ 
  &\leq&
  R\max_{m=1,\ldots,R}\pP\left(Y_{j,r}^{(m)} \geq
    (\log r + j)^{-3/2}N_{j,r}^{(m)}\right)\leq R\exp\left(-(\log r +j)^{1/7}N_{j,r}\right)\\
  &\leq&R\exp\left(-\frac{1}{C}\ (\log r +j)^{1/7}{\left(\frac{L}{r2^j}\right)}^{d-1}\right)\leq
  \exp\left(-(\log r +j)^{1/7}(\log L)^{29}\right),
\end{eqnarray*}
for $r_0\leq r < r_L$, $0\leq j\leq J_1(L)$ and $L$ large enough. In
particular,
\begin{equation*}
  \sum_{0\leq j\leq
    J_1(L)}\pP\left(Y_{j,r}\geq (\log r + j)^{-3/2}N_{j,r}\right) \leq
  \exp\left(-(\log L)^{28}\right).
\end{equation*}
Therefore, setting
\begin{equation*}
  \bbad = \bigcup_{0\leq j\leq J_1(L)}\left\{Y_{j,r} \geq
    (\log r + j)^{-3/2}N_{j,r}\right\},
\end{equation*}
we have proved the following
\begin{lemma}
  \label{def:smoothbad-lemmabbad}
  There exists a constant $r_0>0$ such that if $r\geq r_0$ and
  $L_1$ is large enough, then $\ctwo(\delta,L_1)$  implies
  that for $L$ with $L_1\leq L \leq L_1(\log L_1)^2$,
\begin{equation*}
  \pP\left(\bbad\right) \leq \exp\left(-(\log L)^{28}\right).
\end{equation*}
\end{lemma}


   \section{Some important estimates}
\label{estimates}
In this section, we present various results on exit and hitting
proababilities for both simple random walk and Brownian motion.
\subsection{Hitting probabilities}
The first two lemmata concern simple random walk.
\begin{lemma}
\label{def:lemmalawler}
Let $0<\eta < 1$. 
\begin{enumerate}
\item There exists $C=C(\eta)>0$ such that  
for all $x\in V_{\eta L}$, 
  $y\in\partial V_L$,
\begin{equation*}
C^{-1}L^{-d+1} \leq \pi_L(x,y) \leq CL^{-d+1}.
\end{equation*}
\item
There exists $C=C(\eta)>0$ such that for all $x,x'\in V_{\eta L}$,
$y\in\partial V_L$,
\begin{equation*}
\left|\pi_L(x,y)-\pi_L(x',y)\right| \leq C|x-x'|L^{-d}.
\end{equation*}
\item Let $0 < l < L$ and $x \in V_L$ with $l < |x| < L$. Then
\begin{equation*}
\Prw_x(\tau_L < T_{V_l}) = \frac{l^{-d+2}-|x|^{-d+2} + O(l^{-d+1})}{l^{-d+2}-L^{-d+2}}.
\end{equation*}
\end{enumerate}
\end{lemma}
\begin{prooof}
  (i) $\pi_L(\cdot,y)$ is harmonic inside $V_L$. Applying a discrete
  Harnack inequality, as, for example, provided by Theorem 6.3.9 in the book of
  Lawler and Limic~\cite{LawLim}, we see that
  $C^{-1}\pi_L(0,y)\leq\pi_L(\cdot,y)\leq C\pi_L(0,y)$ on $V_{\eta L}$, for
  some $C=C(d,\eta)$. Part (i) then follows from Lemma 6.3.7 in the same book.\\
  (ii) By the triangle inequality,
  \begin{equation*}
    \left|\pi_L(x,y)-\pi_L(x',y)\right|\leq C|x-x'|\max_{u,v\in V_{\eta
        L}: |u-v|\leq 1}|\pi_L(u,y)-\pi_L(v,y)|.
  \end{equation*}
  For $u\in V_{\eta L}$, the function $\pi_L(u+\cdot,y)$ is harmonic inside
  $V_{(1-\eta)L}$. The claim now follows from~\cite{LawLim} Theorem 6.3.8, (6.19),
  together with (i).\\
 (iii) This is Proposition 1.5.10 of~\cite{LAW}. 
\end{prooof}
A good control over hitting probabilities is given by
\begin{lemma}
\label{def:hittingprob}
Let $a \geq 1$ and $x,y \in
\mathbb{Z}^d$ with $x\notin V_a(y)$. Then
\begin{enumerate}
\item
\begin{equation*}
\Prw_x\left(T_{V_a(y)} < \infty\right) = \left(\frac{a}{|x-y|}\right)^{d-2}\left(1+O(a^{-1})\right).
\end{equation*}
\item There exists $C>0$,
  independent of $a$, such that when $|x-y|>7a$,
\begin{equation*}
\Prw_x\left(T_{V_a(y)} < \tau_L\right) \leq C\frac{a^{d-2}\max\{a,\dL(y)\}\max\{1,\dL(x)\}}{|x-y|^d}.
\end{equation*}
\item There exists $C>0$ such that for all $x\in V_L$, $y\in\partial V_L$,
\begin{equation*}
C^{-1}\frac{\dL(x)}{|x-y|^{d}}\leq \pi_L(x,y) \leq C\frac{\max\{1,\dL(x)\}}{|x-y|^{d}}.
\end{equation*}
\end{enumerate}
\end{lemma}
This lemma will be proved in the appendix. 

We need analogous results for Brownian motion in $\mathbb{R}^d$. Denote by
$\pibm_L(y,dz)$ the exit measure of $d$-dimensional Brownian motion from
$C_L$, started at $y\in C_L$. By a small abuse of notation, we also write
$\pibm_L(y,z)$ for the (continuous version of the) density with respect to
surface measure on $C_L$, which is given by the Poisson kernel
\begin{equation}
\label{eq:est-poissonkernel}
\pibm_L(y,z)= \frac{1}{d\,\alpha(d)\,L}\frac{L^2-|y|^2}{|y-z|^d},
\end{equation}
where $\alpha(d)$ is the volume of the unit ball. From this explicit form,
we can directly read off the analogous statements of
Lemma~\ref{def:lemmalawler} (i), (ii) and Lemma~\ref{def:hittingprob}
(iii), with $V_L$ replaced by $C_L$. Let us now formulate and prove the
analog of parts (i) and (ii) from the last lemma.  Denote by $\Pbm_x$ the
law of standard $d$-dimensional Brownian motion, started at
$x\in\mathbb{R}^d$. For the following statement, $T_{C_a(y)}$ and
$\tau_{C_L}$ are defined in the obvious way in terms of Brownian motion.
\begin{lemma}
\label{def:hittingprob-bm}
Let $a>0$ and $x,y\in\mathbb{R}^d$ with $x\notin C_a(y)$. Then
\begin{enumerate}
\item 
\begin{equation*}
\Pbm_x\left(T_{C_a(y)} < \infty\right) = \left(\frac{a}{|x-y|}\right)^{d-2}.
\end{equation*}
\item Assume $C_{2a}(y)\subset C_L$. There exists $K>0$ such
  that
\begin{equation*}
\Pbm_x\left(T_{C_a(y)} < \tau_{C_L}\right)\leq K\frac{a^{d-2}\dist_L(y)\dist_L(x)}{|x-y|^d}.
\end{equation*}
\end{enumerate}
\end{lemma}
\begin{prooof2}{\bf of Lemma~\ref{def:hittingprob-bm}:}
(i) See for example the book of Durrett~\cite{DUR}, $(1.12)$. \\
(ii) Recall that the Green's function of Brownian motion for $C_L$ is given by
\begin{equation*}
\gbm(x,y) = A_d\left(\left(\frac{1}{|x-y|}\right)^{d-2} -
  \left(\frac{L}{\left|x\right|\left|x^\star-y\right|}\right)^{d-2}\right), 
\end{equation*}
where $A_d$ is an explicit constant, and for $x\neq 0$, $x^\star =
(L^2/|x|^2)x$ is the inversion of $x$ with respect to $C_L$.
Now, for $a > 0$ and $x\notin C_a(y)$, we have 
\begin{equation*}
\int\limits_{C_{a/2}(y)}\gbm(x,z)\dt z \geq
\Pbm_x\left(T_{C_a(y)}<\tau_{C_L}\right)\inf\limits_{v\in\partial
  C_a(y)}\int\limits_{C_{a/2}(y)}\gbm(v,z)\dt z.
\end{equation*}
By Proposition $1$ of~\cite{CHU}, the infimum on the right-hand side can
be bounded from below by $c\,a^2$. Using the second upper bound on $\gbm(x,z)$
from the same proposition, we get
\begin{equation*}
\Pbm_x\left(T_{C_a(y)}<\tau_{C_L}\right) \leq c^{-1}a^{-2}
\int\limits_{C_{a/2}(y)}\gbm(x,z)\dt z \leq K\,\frac{a^{d-2}\dist_L(y)\dist_L(x)}{|x-y|^d}.
\end{equation*}
\end{prooof2}
\begin{remark}
  Lemma~\ref{def:hittingprob} (ii) can be proved in the same way if $|x|,
  |y| \leq cL$ for some $c<1$, for example by using Proposition 8.4.1 of~\cite{LawLim},
  which is based on a coupling argument. Since we need an estimate
  including the case when $x$ or $y$ are near the boundary, we give a
  self-contained proof in the appendix.
\end{remark}
Probabilities of the above type will often be estimated by the following
\begin{lemma}
\label{def:hittingprob-technical}
Let $a > 0$, $l, m\geq 1$ and $x\in\mathbb{Z}^d$. Set $R_l = V_l\backslash
V_{l-1}$, $\alpha=\max\left\{\left||x|-l\right|,a\right\}$. Then for some constant $C = C(m) > 0$
\begin{equation*}
\sum_{y\in R_l}\frac{1}{\left(a+|x-y|\right)^m} \leq C\left\{\begin{array}{l@{\quad \mbox{\textup{for}\ }}l}
   l^{d-(m+1)} & 1 \leq m < d-1\\
  \max\{\log(l/\alpha), 1\} & m = d-1\\      
   \alpha^{d-(m+1)}& m \geq d\end{array}\right..
\end{equation*}
\end{lemma}
\begin{prooof}
If $\alpha > l$, then the left-hand side
is bounded by 
\begin{equation*}
Cl^{d-1}\alpha^{-m}\leq C\max\left\{\alpha^{d-(m+1)},l^{d-(m+1)}\right\}.
\end{equation*}
If $\alpha \leq
l$, we set $A_k = \{y\in R_l: |x-y| \in
[(k-1)\alpha,k\alpha)\}$. Then, for all $k\geq 1$,
\begin{equation*}
  \max_{y\in A_k}\frac{1}{\left(a+|x-y|\right)^m} \leq 2^mk^{-m}\alpha^{-m}.
\end{equation*}
Since for $k\alpha \leq l/10$ we have $|A_k| \leq C\alpha(k\alpha)^{d-2}$,
the claim then follows from
\begin{eqnarray*}
  \sum_{y\in R_l}\frac{1}{\left(a+|x-y|\right)^m} &\leq&
  C\left(\sum_{1\leq k\leq
      \lfloor l/(10\alpha)\rfloor}\frac{\alpha(k\alpha)^{d-2}}{(k\alpha)^m}\right) + Cl^{d-1}l^{-m}\\
  &\leq& C\alpha^{d-{(m+1)}}\sum_{1\leq k\leq
    \lfloor l/(10\alpha)\rfloor}k^{d-(m+2)} + Cl^{d-(m+1)}.
\end{eqnarray*}
\end{prooof}

\subsection{Smoothed exit measures}
We will compare exit laws of simple random walk with exit laws of Brownian motion.
Given a field of positive real numbers $\psi =\left(m_x\right)_{x\in
  \mathbb{R}^d}$, we define the smoothed exit law from $V_L$ of simple
random walk as
\begin{equation*}
  \phi_{L,\psi}(x,z) = \pi_L\ph_{\psi}(x,z) = \sum_{y\in\partial
    V_L}\pi_L(x,y)\frac{1}{m_y}\int_{\mathbb{R}_+}\varphi\left(\frac{t}{m_y}\right)
  \pi_{V_t(y)}(y,z)\dt t.
\end{equation*}
Denoting by $\pibm_{C_t(x)}(x,dz)$ the exit measure of $d$-dimensional Brownian
motion from $C_t(x)$, started at $x$, we let analogous to~\eqref{eq:prel-cgpsi},
\begin{equation*}
  \phbm_{\psi}(x,dz) = \frac{1}{m_x}\int_{\mathbb{R}_+}\varphi\left(\frac{t}{m_x}\right)
  \pibm_{C_t(x)}(x,dz)\dt t.
\end{equation*} 
Then define the smoothed Brownian exit measure from $C_L$ as
\begin{equation*}
  \phibm_{L,\psi}(x,dz) =
  \pibm_L\phbm_{\psi}(x,dz) = \int_{\partial
    C_L}\pibm_L(x,dy)
  \frac{1}{m_y}\int_{\mathbb{R}_+}\varphi\left(\frac{t}{m_y}\right)
   \pibm_{C_t(y)}(y,dz)\dt t.
\end{equation*}
By $\phibm_{L,\psi}(x,z)$ we denote the density of $\phibm_{L,\psi}(x,dz)$
with respect to $d$-dimensional Lebesgue measure.
\begin{lemma}
\label{def:est-phi}
Let $\psi \in \mathcal{M}_L$. There exists a constant $C > 0$
such that
\begin{enumerate}
\item
\begin{equation*}
  \sup_{x\in V_L}\sup_{z\in\mathbb{Z}^d}\left|\left(\phi_{L,\psi}-\phibm_{L,\psi}\right)(x,z)\right|\leq
  CL^{-(d+1/4)}.
\end{equation*}
\item
\begin{equation*}
  \sup_{z\in\mathbb{R}^d}{\left|\left|D^i\phibm_{L,\psi}(\cdot,z)\right|\right|}_{C_L}
  \leq CL^{-(d+i)},\quad i =0,1,2,3.
\end{equation*}
\item \begin{equation*}
\sup_{x,x'\in V_L\cup \partial V_L}\sup_ {z\in\mathbb{Z}^d}\left|\phi_{L,\psi}(x,z)-\phi_{L,\psi}(x',z)\right| \leq C\left(L^{-(d+1/4)}
  + |x-x'|L^{-(d+1)}\right).
\end{equation*}
\end{enumerate}
\end{lemma}
For the proof, we refer to the appendix.  The next proposition will be
applied at the end of the proof of Lemma~\ref{def:lemma-goodpart}. At this
point, the invariance condition $\czero$ comes into play.  We give a
general formulation in terms of a signed measure $\nu$. Let us introduce
the following notation.  For $x=(x_1,\ldots,x_d)\in \mathbb{Z}^d$, put
\begin{eqnarray*}
x^{(i)} &=& \left(x_1,\ldots,x_{i-1},-x_i,x_{i+1},\ldots,x_d\right),\\
x^{\leftrightarrow (i,j)} &=&
\left(x_1,\ldots,x_{i-1},x_j,x_{i+1},\ldots,x_{j-1},x_i,x_{j+1},\ldots,x_d\right),\,\textup{if
  }i<j.
\end{eqnarray*}
\begin{proposition}
\label{def:est-isometry}
Let $l>0$. Consider a measure $\nu$ on $V_l$ with total mass zero
satisfying
\begin{enumerate}
\item $\nu(x) = \nu(x^{(i)})$ for all $x$ and all $i=1,\ldots, d$.
\item $\nu(x)=\nu(x^{\leftrightarrow (i,j)})$ for all $x$ and all $i,j = 1, \ldots, d$, $i<j$. 
\end{enumerate}
Then there exists $C>0$ such that for $y'\in V_L$ with $V_l(y')
\subset V_L$ and all $z\in\mathbb{Z}^d$, $\psi\in\mathcal{M}_L$,
\begin{equation*}
\left | \sum_{y\in V_l(y')}\nu(y-y')\phi_{L,\psi}(y,z)\right| \leq
C||\nu||_1\left(L^{-(d+1/4)} + \left(\frac{l}{L}\right)^3L^{-d}\right).
\end{equation*}
\end{proposition}
\begin{prooof}
  Since the proof is the same for all $y'\in V_L$ with $V_l(y')
\subset V_L$, we can assume $y'=0$. By
  Lemma~\ref{def:est-phi} (i),
\begin{equation*}
\left|\sum_{y}\nu(y)\phi_{L,\psi}(y,z) -
  \sum_{y}\nu(y)\phibm_{L,\psi}(y,z)\right| \leq
C||\nu||_1L^{-(d+1/4)}.
\end{equation*}
Taylor's expansion gives
\begin{eqnarray}
\label{eq:est-isometry-1} 
\lefteqn{\sum_{y}\nu(y)\phibm_{L,\psi}(y,z)}\nonumber\\&=&
\sum_{y}\nu(y)\left[\phibm_{L,\psi}(y,z) -
  \phibm_{L,\psi}(0,z)\right]\nonumber\\
&=&\sum_{y}\nu(y)\nabla_x\phibm_{L,\psi}(0,z)\cdot y
+\ \frac{1}{2}\sum_{y}\nu(y)y\cdot H_x\phibm_{L,\psi}(0,z)y + R(\nu,0,z),\nonumber\\
\end{eqnarray}
where $\nabla_x\phibm_{L,\psi}$ is the gradient, $H_x\phibm_{L,\psi}$ the
Hessian of $\phibm_{L,\psi}$ with respect to the first variable, and
$R(\nu,0,z)$ is the remainder term, which can be bounded by
Lemma~\ref{def:est-phi} (ii), namely
\begin{equation*}
|R(\nu,0,z)| \leq C ||\nu||_1\left(\frac{l}{L}\right)^3L^{-d}.
\end{equation*}
Since $\nu$ satisfies property (i), the first summand
on the right side of~\eqref{eq:est-isometry-1} vanishes. Due to the same
reason, the second summand equals 
\begin{equation*}
\frac{1}{2}\sum_{i=1}^d\frac{\partial^2}{\partial x_i^2}\phibm_{L,\psi}(0,z)\sum_{y}\nu(y)(y_i)^2.
\end{equation*}
By property (ii), the sum over $y$ does not depend on $i$, so a multiple of the Laplacian of $\phibm_{L,\psi}$ remains. But
for each $v\in\partial C_L$, $\pibm_L(\cdot,v)$ is harmonic in $C_L$, thus
also the Laplacian vanishes. This proves the proposition. 
\end{prooof}

   \section{Green's functions for the ball}
\label{super}
One principal task of our approach aims at developing good estimates on Green's
functions for the ball of both coarse grained (goodified) RWRE as well as coarse grained
simple random walk. The main result is
Lemma~\ref{def:superlemma}. For the coarse grained simple random walk, the estimates on
hitting probabilities of the last section together with
Proposition~\ref{def:super-behaviorgreen} yield the right control.

On a certain class of environments, we need to modify
the transition kernels in order to ensure that bad
points are not visited too often by the coarse grained random walks. This
modification will be described in Section~\ref{super-cgmod}.
 
\subsection{A local central limit theorem}
\label{clt}
Let $m\geq 1$. Denote by $\ph_m$ the coarse grained transition
probabilities on $\mathbb{Z}^d$ belonging to the field $\psi =
{(m_x)}_{x\in\mathbb{Z}^d}$, where $m_x= m$ is chosen constant in $x$.
Notice that $\ph_m$ is centered, and the covariances satisfy
\begin{equation*}
\sum_{y\in\mathbb{Z}^d}(y_i-x_i)(y_j-x_j)\ph_m(x,y) = \gamma_m\delta_i(j),
\end{equation*}
where for large $m$ (recall the coarse graining scheme) $1/d<\gamma_m/m^2< 4/d$.
\begin{proposition}[Local central limit theorem]
\label{def:super-localclt}
Let $x,y\in\mathbb{Z}^d$. For $m \geq 1$ and all integers $n\geq 1$,
\begin{equation*}
  \ph_m^n(x,y) = \frac{1}{(2\pi
    \gamma_mn)^{d/2}}\exp\left(-\frac{|x-y|^2}{2\gamma_mn}\right) +
  O\left(m^{-d}n^{-(d+2)/2}\right).
\end{equation*}
\end{proposition}
For the corresponding Green's function $\gh_{m,{\mathbb{Z}^d}}(x,y) =
\sum_{n=0}^\infty\ph_m^n(x,y)$ we obtain
\begin{proposition}
\label{def:super-behaviorgreen}
Let $x,y\in\mathbb{Z}^d$. There exists $m_0 > 0$ such that if $m\geq m_0$, then 
\begin{enumerate}
\item For $|x-y| < 3m$,
\begin{equation*}
\gh_{m,{\mathbb{Z}^d}}(x,y) = \delta_{x}(y) + O(m^{-d}).
\end{equation*}
\item For $|x-y| \geq 3m$, there exists a constant $c(d) > 0$ such that
\begin{equation*}
\gh_{m,{\mathbb{Z}^d}}(x,y )= \frac{c(d)}{\gamma_m|x-y|^{d-2}} +
O\left(\frac{1}{|x-y|^{d}}\left(\log\frac{|x-y|}{m}\right)^d\right). 
\end{equation*}
\end{enumerate}
Here, the constants in the $O$-notation are independent of $m$ and $|x-y|$.
\end{proposition}
In our applications, $m$ will be a function of $L$. Although these results
look rather standard, we cannot directly refer to the literature because we
have to keep track of the $m$-dependency. We give a proof of both
statements in the appendix.

We will use the last proposition to estimate the Green's function for the
ball $V_L$, $\gh_{m}(x,y) =
\sum_{n=0}^\infty\left(1_{V_L}\ph_m\right)^n(x,y)$. Clearly, $\gh_{m}$
is bounded from above by $\gh_{m,{\mathbb{Z}^d}}$, and more precisely,
the strong Markov property shows
\begin{equation}
\label{eq:clt-greenonball}
  \gh_m(x,y) = \Erw_{x,\ph_m}\left[\sum_{k=0}^{\tau_L-1}1_{\{X_k = y\}}\right] =
  \gh_{m,\mathbb{Z}^d}(x,y) - \Erw_{x,\ph_m}\left[\gh_{m,\mathbb{Z}^d}\left(X_{\tau_L},y\right)\right]. 
\end{equation}

\subsection{Estimates on coarse grained Green's functions}
\label{super-Ghestimates}
As we will show, the perturbation expansion enables us to control the
goodified Green's function $\Ghg_{L,r}$ essentially in terms of
$\gh_{L,r}$. The boundary region $\sh_L(r)$ turns out to be problematic,
since even for good $x$, we cannot estimate the variational distance
between the transition kernels by $\delta$.  We therefore work in this (and
only in this) section with slightly modified transition kernels
$\Pt_{L,r}$, $\pt_{L,r}$, $\Ptg_{L,r}$ in the enlarged ball $V_{L+r}$,
taking the exit measure in $\sh_L(r)$ from uncut balls $V_t(x)\subset
V_{L+r}$, $t\in[h_{L,r}(x),2h_{L,r}(x)]$. More precisely, setting
$h_{L,r}(x) = (1/20)r$ for $x\notin C_L$, we let $\pt_{L,r}$ be the coarse
grained simple random walk kernel under $\tilde{\psi} = (h_{L,r}(x))_{x\in
  V_{L+r}}$, that is
\begin{equation*}
  \pt_{L,r}(x,\cdot) =
      \frac{1}{h_{L,r}(x)}\int_{\mathbb{R}_+}
      \varphi\left(\frac{t}{h_{L,r}(x)}\right)\pi_{V_t(x)\cap
        V_{L+r}}(x,\cdot)\dt t.
\end{equation*}
For the corresponding RWRE kernel, we forget about the environment on
$V_{L+r}\backslash V_L$ and set
\begin{equation*}
  \Pt_{L,r}(x,\cdot) = \left\{\begin{array}{l@{\quad \mbox{for\ }}l}
      \frac{1}{h_{L,r}(x)}\int_{\mathbb{R}_+}
      \varphi\left(\frac{t}{h_{L,r}(x)}\right)\Pi_{V_t(x)}(x,\cdot)\dt t &
      x\in V_L\\
      \pt_{L,r}(x,\cdot) & x\in V_{L+r}\backslash V_L
  \end{array}\right..
\end{equation*}
For all good $x\in V_L$ we now have $||(\Pt_{L,r}-\pt_{L,r})(x,\cdot)||_1 \leq \delta$,
while for $x\in V_{L+r}\backslash V_L$, the difference vanishes anyway.
The goodified version of $\Pt_{L,r}$ is then
obtained in an analogous way to~\eqref{eq:smoothbad-goodifiedkernel},
\begin{equation*}
\Ptg_{L,r}(x,\cdot) = \left\{\begin{array}{l@{\quad \mbox{for\ }}l} 
       \Pt_{L,r}(x,\cdot) & x\notin \badPo_{L,r}\\
       \pt_{L,r}(x,\cdot) & x\in \badPo_{L,r}\end{array}\right..
 \end{equation*}
We write $\Gt_{L,r},\gt_{L,r}$ and $\Gtg_{L,r}$ for the corresponding
Green's functions on $V_{L+r}$. Note that 
\begin{equation}
\label{eq:super-greensfcts-pointwisebdd}
\Gh_{L,r}\leq \Gt_{L,r},\quad \gh_{L,r} \leq
\gt_{L,r},\quad\Ghg_{L,r} \leq \Gtg_{L,r}\quad\mbox{pointwise on }V_{L+r}\times \left(V_{L+r}\backslash \partial V_L\right).
\end{equation}
Since we do not have exact expressions for $\gt_{L,r}$ or $\Gt_{L,r}$, we will construct a
(deterministic) kernel $\Gamma_{L,r}$ that bounds the Green's functions
from above. For $x\in V_{L+r}$, set
\begin{equation*}
  \dtL(x) = \max\left(\frac{\dist_{L+r}(x)}{2},3r\right),\quad a(x) =\min\left(\dtL(x),s_L\right).
\end{equation*}
Further, let
\begin{equation*}
  \Gamma^{(1)}_{L,r}(x,y) = \frac{\tilde{d}(x)\tilde{d}(y)}{a(y)^2(a(y) +
    |x-y|)^d},\quad
  \Gamma^{(2)}_{L,r}(x,y) = \frac{1}{a(y)^2(a(y) + |x-y|)^{d-2}}.
\end{equation*}
The kernel $\Gamma_{L,r}$ is defined as the pointwise minimum 
\begin{equation}
\label{eq:defgamma}
  \Gamma_{L,r} = \min\left\{\Gamma^{(1)}_{L,r},\Gamma^{(2)}_{L,r}\right\}.
\end{equation}
We cannot derive pointwise estimates on the Green's functions in terms of
$\Gamma_{L,r}$, but we can use this kernel to obtain upper bounds on
neighborhoods $U(x) = V_{a(x)}(x)\cap V_{L+r}$. Call a function $F:
V_{L+r}\times V_{L+r}\rightarrow\mathbb{R}_+$ a {\it positive
  kernel}. Given two positive kernels $F$ and $G$, we write $F\preceq G$ if
for all $x,y\in V_{L+r}$,
\begin{equation*}
  F(x,U(y)) \leq G(x,U(y)),
\end{equation*}
where $F(x,U)$ stands for $\sum_{y\in U\cap\mathbb{Z}^d}F(x,y)$. Further,
we write $F\asymp 1$, if there is a constant $C>0$ such that for all
$x,y\in V_{L+r}$, 
\begin{equation*}
\frac{1}{C}F(x,y)\leq F(\cdot,\cdot)\leq CF(x,y)\quad\mbox{on }U(x)\times U(y).
\end{equation*}
We adopt this notation to positive functions of one argument: For
$f:V_{L+r}\rightarrow\mathbb{R}_+$, $f\asymp 1$ means that for some $C>0$, $C^{-1}f(x)\leq
f(\cdot)\leq Cf(x)$ on any $U(x)\subset V_{L+r}$.
Finally, given $0 <\eta < 1$, we say that a positive kernel $A$ on $V_{L+r}$ is
$\eta$-{\it smoothing}, if for all $x\in V_{L+r}$, $A(x,U(x)) \leq \eta$, and
$A(x,y) = 0$ whenever $y\notin U(x)$.

Now we are in the position to formulate our main statement of this
section. Recall our convention concerning constants: They only depend on
the dimension unless stated otherwise.  
\begin{lemma}\ 
\label{def:superlemma}
\begin{enumerate}
\item There exists a constant $C_1>0$ such that
\begin{equation*}
      \gh_{L,r}\preceq C_1\Gamma_{L,r}\quad\mbox{and}\quad\gt_{L,r} \preceq C_1\Gamma_{L,r}.
    \end{equation*}
\item  There exists a constant $C>0$ such that for small $\delta>0$,  
\begin{equation*}
  \Ghg_{L,r} \preceq C\Gamma_{L,r}\quad\mbox{and}\quad\Gtg_{L,r} \preceq C\Gamma_{L,r}.
\end{equation*}  
\end{enumerate}
\end{lemma}
\begin{remark}
  Thanks to~\eqref{eq:super-greensfcts-pointwisebdd}, it suffices to prove
  the bounds for $\gt_{L,r}$ and $\Gtg_{L,r}$.  For later use, we keep
  track of the constant in part (i) of the lemma.
\end{remark}
We first prove part (i), which is a straightforward consequence of the estimates on hitting
probabilities in Section~\ref{estimates} and the next lemma. 
\begin{lemma}
  \label{def:super-greenbound}
  There exists $C>0$ such that for all $x\in V_{L+r}$ and $y \in V_{L}$ with $\dL(y) \geq 4s_L$,
    \begin{equation*}
      \gt_{L,r}(x,y)\leq C\left\{ \begin{array}{l@{,\quad}l}\frac{1}{s_L^2\max\{|x-y|,s_L\}^{d-2}}
          & y\neq x\\
          1 & y=x\end{array}\right..
    \end{equation*}
\end{lemma}
\begin{prooof}
If $x=y$, then the claim follows from transience of simple random
  walk. Now assume $x\neq y$, and always $\dL(y) \geq 4s_L$. Consider first
  the case $|x-y| \leq s_L$. Let $\gh_m$ be defined as in the beginning of
  Section~\ref{clt}. Recall our coarse graining scheme. With $m= s_L/20$
  we have
  \begin{equation*}
    \gt(x,y) \leq \gh_m(x,y) +
    \sup_{v\in\sh_L(2s_L)}\Prw_v\left(T_{V_{s_L }(y)} <
      \tau_{V_{L+r}}\right)\sup_{w: w\neq y, \atop |w-y| \leq s_L}\gt(w,y).
  \end{equation*}
  Since
  \begin{equation*}
    \sup_{v\in\sh_L(2s_L)}\Prw_v\left(T_{V_{s_L}(y)} < \tau_{V_{L+r}}\right) < 1
  \end{equation*}
  uniformly in $L$, it follows from Proposition~\ref{def:super-behaviorgreen}
  that
  \begin{equation*}
   \gt(x,y) \leq C\sup_{w: w\neq y, \atop |w-y| \leq s_L} \gh_m(w,y)\leq C\sup_{w: w\neq y, \atop |w-y| \leq s_L} \gh_{m,\mathbb{Z}^d}(w,y) \leq \frac{C}{s_L^d}.
  \end{equation*}
  If $|x-y| > s_L$ we use Lemma~\ref{def:hittingprob} (i) and the first
  case to get
  \begin{equation*}
    \gt(x,y) \leq \Prw_x\left(T_{V_{s_L}(y)} < \infty\right)\sup_{w: w\neq
        y, \atop |w-y| \leq s_L}\gt(w,y) \leq \frac{C}{s_L^2|x-y|^{d-2}}. 
  \end{equation*}  
\end{prooof}
\begin{prooof2}{\bf of Lemma~\ref{def:superlemma} (i):}
 It suffices to prove the bound for $\gt$. 
  First we show that there exists a constant $C
  > 0$ such that for all $y \in V_{L+r}$,
  \begin{equation}
    \label{eq:super-keyest-g1}
    \sup_{x\in V_{L+r}}\gt\left(x,U(y)\right) \leq C.
  \end{equation}  
  At first let $\dist_{L+r}(y) \leq 6r$. Then $U(y)\subset
  \sh_{L+r}(10r)$. We claim that 
  \begin{equation}
    \label{eq:super-keyest-g2}
    \sup_{x\in V_{L+r}}\gt\left(x,\sh_{L+r}(10r)\right)\leq C.
  \end{equation}
  for some $C> 0$. Indeed, if $z\in\sh_{L+r}(10r)$, then $\pt(z,\cdot)$ is
  an (averaging) exit distribution from balls $V_l(z)\cap V_{L+r}$, where
  $l\geq r/20$. Using Lemma~\ref{def:lemmalawler} (i), we find a
  constant $k_1 = k_1(d)$ such that starting at any $z\in\sh_{L+r}(10r)$,
  $V_{L+r}$ is left after $k_1$ steps with probability $ > 0$, uniformly in
  $z$. This together with the strong Markov property
  implies~\eqref{eq:super-keyest-g2}.  Next assume $6r<\dist_{L+r}(y)\leq
  6s_L$. Then $U(y) \subset
  S(y)=\sh_{L+r}\left(\frac{1}{2}\dist_{L+r}(y),2\dist_{L+r}(y)\right)$. We
  claim that
  \begin{equation}
    \label{eq:super-keyest-g3}
    \sup_{x\in V_{L+r}}\gt\left(x,S(y)\right)\leq C.
  \end{equation}
  For $z\in S(y)$, $\pt(z,\cdot)$ is an averaging exit distribution from
  balls $V_l(z)$, where $l\geq \dist_{L+r}(y)/240$. By 
  Lemma~\ref{def:lemmalawler} (i), we find some small $0 < c
  < 1$ and a constant $k_2(c,d)$
  such that after $k_2$ steps, the walk has probability $> 0$ to be in
  $\sh_{L+r}\left(\frac{1-c}{2}\dist_{L+r}(y)\right)$, uniformly in $z$ and $y$. But starting in
  $\sh_{L+r}\left(\frac{1-c}{2}\dist_{L+r}(y)\right)$, Lemma~\ref{def:lemmalawler} (iii) shows that with probability $> 0$, the
  ball $V_{L+r}$ is left before $S(y)$ is visited
  again. Therefore~\eqref{eq:super-keyest-g3} and
  hence~\eqref{eq:super-keyest-g1} hold in this case.
  At last, let $\dist_{L+r}(y) > 6s_L$. Then $\dL(w) \geq 4s_L$ for
  $w\in U(y)$. Estimating
  \begin{equation*}
  \gt(x,w) \leq 1 + \sup_{v: v\neq w}\gt(v,w),
  \end{equation*}
  we get with part (i) that
 \begin{equation*}
  \sup_{w\in U(y)}\gt(x,w) \leq 1 + \frac{C}{s_L^d}.
  \end{equation*}
  Summing over $w\in U(y)$, \eqref{eq:super-keyest-g1} follows. Finally, note
  that for any $x\in V_{L+r}$,
  \begin{equation*}
  \gt(x,U(y)) \leq \Prw_x\left(T_{U(y)} \leq \tau_{V_{L+r}}\right)\sup_{w\in U(y)}\gt(w,U(y)).
  \end{equation*}
  Now $\gt\preceq C\Gamma$ follows from \eqref{eq:super-keyest-g1} and
  the hitting estimates of
  Lemma~\ref{def:hittingprob}.
\end{prooof2}
Let us now explain our strategy for proving part (ii). By
version~\eqref{eq:prel-pbe3} of the perturbation expansion, we can express
$\Gtg_{L,r}$ in a series involving $\gt_{L,r}$ and differences of
exit measures. The Green's
function $\gt_{L,r}$ is already controlled by means of $\Gamma_{L,r}$.
Looking at~\eqref{eq:prel-pbe3}, we thus have to understand what happens if
$\Gamma_{L,r}$ is concatenated with certain smoothing kernels. This will be
the content of Proposition~\ref{def:super-concatenating}. 

We start with collecting some important properties of $\Gamma_{L,r}$, which will be
used throughout this text. Define for $j\in\mathbb{N}$
\begin{equation*}
  \mathcal{L}_j = \{y\in V_L : j\leq\dL(y) < j+1\},\quad \mathcal{E}_j=\{y\in V_{L+r} : \dtL(y)\leq 3jr\}.
\end{equation*}

\begin{lemma}[Properties of $\Gamma_{L,r}$]\     
  \label{def:super-gammalemma}
  \begin{enumerate}
  \item Both $\dtL$ and $a$ are Lipschitz with constant $1/2$. Moreover,
    for $x,y \in V_{L+r}$,  
    \begin{equation*}
      a(y) + |x-y| \leq a(x) + \frac{3}{2}|x-y|.
    \end{equation*}
  \item
    \begin{equation*}
      \Gamma_{L,r}\asymp 1.
    \end{equation*}
  \item For $0\leq j \leq 2s_L$, $x\in V_{L+r}$,
    \begin{equation*}
      \sum_{y\in\mathcal{L}_j}\left(\max\left\{1,\frac{\dtL(x)}{a(y)}\right\}\frac{1}{(a(y)+
          |x-y|)^d}\right) \leq C\frac{1}{j\vee r}. 
    \end{equation*}
\item For $1\leq j \leq \frac{1}{3r}s_L$,
\begin{equation*}
  \sup_{x\in V_{L+r}}\Gamma_{L,r}(x,\mathcal{E}_j) \leq C\log(j+1),
\end{equation*}
and for $0\leq \alpha < 3$,
\begin{equation*}
  \sup_{x\in V_{L+r}}\Gamma_{L,r}\left(x,\sh_L\left(s_L, L/(\log
      L)^\alpha\right)\right) \leq C(\log\log L)(\log L)^{6-2\alpha}. 
\end{equation*}
\item For $x\in V_{L+r}$, in the case of constant $r$,
\begin{equation*}
  \Gamma_{L,r}(x,V_L) \leq C\max\left\{\frac{\dtL(x)}{L}(\log L)^6,\,
    \left(\frac{\dtL(x)}{r}
 \wedge \log L\right)\right\}.
\end{equation*}
In the case $r= r_L$,
\begin{equation*}
  \Gamma_{L,r_L}(x,V_L) \leq C\max\left\{\frac{\dtL(x)}{L}(\log L)^6,\,
    \left(\frac{\dtL(x)}{r_L}\wedge \log\log L\right)\right\}.
\end{equation*}
\end{enumerate}
\end{lemma}

\begin{prooof}
  (i) The second statement is a direct consequence of the Lipschitz property,
  which in turn follows immediately from the definitions of $\dtL$ and $a$. \\
  (ii) As for $y'\in U(y)$, $\frac{1}{2}a(y)\leq a(y') \leq
  \frac{3}{2}a(y)$ and similarly with $a$ replaced by $\dtL$, it suffices
  to show that for $x'\in U(x)$, $y'\in U(y)$,
  \begin{equation}
    \label{eq:super-gammaest-ineq}
    \frac{1}{C}\left(a(y)+|x-y|\right) \leq a(y')+|x'-y'| \leq C\left(a(y)+|x-y|\right).
\end{equation}
First consider the case $|x-y| \geq 4\max\{a(x), a(y)\}$. Then
\begin{equation*}
  a(y) +|x-y| \leq 2a(y') + 2\left(|x-y| - a(x) - a(y)\right) \leq 2\left(a(y') + |x'-y'|\right).
\end{equation*}
If $|x-y| \leq 4 a(y) $ then
\begin{equation*}
  a(y) +|x-y| \leq 5a(y) \leq 5a(y) + |x'-y'| \leq 10\left(a(y') + |x'-y'|\right),
\end{equation*}
while for $|x-y| \leq 4 a(x)$, using part (i) in the first inequality,
\begin{equation*}
  a(y) +|x-y| \leq a(x)+\frac{3}{2}|x-y| \leq 7a(x) \leq  14\left(a(y') + |x'-y'|\right).
\end{equation*}
This proves the first inequality in~\eqref{eq:super-gammaest-ineq}. The
second one follows from
\begin{equation*}
  a(y') + |x'-y'| \leq \frac{5}{2}a(y) + a(x) + |x-y| \leq \frac{7}{2}\left(a(y) + |x-y|\right).
\end{equation*}
(iii) If $j\leq 2s_L$ and $y\in\mathcal{L}_j$, then $a(y)$ is of order
$j\vee r$. By Lemma~\ref{def:hittingprob-technical} we have
\begin{equation*}
  \sum_{y\in\mathcal{L}_j}\frac{1}{(j\vee r + |x-y|)^d}\leq
  C\min\left\{\frac{1}{j\vee r},\frac{1}{\left|\dLk(x)-(j+r)\right|}\right\}. 
\end{equation*}
It remains to show that 
\begin{equation}
\label{eq:super-gammalemma-layerbound2}
\max\left\{1,\frac{\dtL(x)}{j\vee r}\right\}\min\left\{\frac{1}{j\vee r},
    \frac{1}{\left|\dLk(x)-(j+r)\right|}\right\} \leq C\frac{1}{j\vee r}.
\end{equation}
If $\dtL(x)\leq (j\vee 3r)$, this is clear. If
$\dtL(x) > (j\vee 3r)$, ~\eqref{eq:super-gammalemma-layerbound2}
follows from
$|\dLk(x)-(j+r)| \geq \dtL(x)/2$.\\
(iv) If $\dtL(y) \leq 3jr$, then $\dL(y) \leq 6jr$. Estimating $\Gamma$ by
$\Gamma^{(1)}$, we get
\begin{equation*}
  \Gamma(x,\mathcal{E}_j) \leq C\sum_{i=0}^{6jr}
  \sum_{y\in\mathcal{L}_i}\frac{\dtL(x)}{a(y)}\frac{1}{(a(y) + |x-y|)^d}.
\end{equation*}
Now the first assertion of (iv) follows from (iii). The second is proved 
similarly, so we omit the details.\\
(v) Set $B= \{y\in V_L: \dtL(y) \leq s_L\vee 2\dtL(x)\}$. For $y\in
V_L\backslash B$, it holds that $a(y) = s_L$ and $|x-y| \geq
\dtL(y)$. Therefore,
\begin{equation*}
\Gamma\left(x, V_L\backslash B\right)\leq\Gamma^{(1)}\left(x, V_L\backslash
  B\right)\leq  \frac{\dtL(x)}{s_L^2}\sum_{y\in V_{2L}}\frac{1}{(s_L +
  |y|)^{d-1}}\leq C\frac{\dtL(x)}{L}(\log L)^6. 
\end{equation*}
Furthermore,
\begin{equation*}
\Gamma\left(x, B\right)\leq \sum_{i=0}^{2s_L}
 \sum_{y\in\mathcal{L}_i}\frac{\dtL(x)}{a(y)}\frac{1}{(a(y) + |x-y|)^d} +
 \frac{1}{s_L^2}\sum_{y\in V_L:\atop s_L\leq \dtL(y)\leq 2\dtL(x)}\frac{1}{(s_L+|x-y|)^{d-2}}.
\end{equation*}
Lemma~\ref{def:hittingprob-technical} bounds the second term by
$C(\dtL(x)/L)(\log L)^6$. For the first term, we use twice part (iii) and
once Lemma~\ref{def:hittingprob-technical} to get
\begin{equation*}
\sum_{i=0}^{2s_L}\sum_{y\in\mathcal{L}_i}\frac{\dtL(x)}{a(y)}\frac{1}{(a(y) + |x-y|)^d}
 \leq C\sum_{i=0}^{5r}\frac{1}{i\vee r} + C\min\left\{
 \dtL(x)\sum_{i=5r}^{2s_L}\frac{1}{i^2}, \sum_{i=5r}^{2s_L}\frac{1}{i}\right\}.
\end{equation*}
This proves (v). 
\end{prooof}
\begin{proposition}[Concatenating]\
  \label{def:super-concatenating}
  Let $F, G$ be positive kernels with $F\preceq G$.
  \begin{enumerate}
  \item If $A$ is $\eta$-smoothing and $G\asymp 1$, then for some constant $C = C(d,G) > 0$,
    \begin{equation*}
      FA \preceq C\eta G.
    \end{equation*}
  \item If $\Phi$ is a positive function on $V_{L+r}$ with $\Phi\asymp
    1$, then for some $C = C(d,\Phi) > 0$,
    \begin{equation*}
      F\Phi \leq CG\Phi.
    \end{equation*}
  \end{enumerate}
\end{proposition}
\begin{prooof}
  (i) As $a$ is Lipschitz with constant $1/2$, we can choose $K=K(d)$
  points $y_k$ out of the set $M=\{y'\in V_{L+r} : U(y')\cap U(y) \neq
  \emptyset\}$ such that $M$ is covered by the union of the $U(y_k), k =
  1,\ldots,K.$ Since $A(y',U(y))\neq 0$ implies $y'\in M$, we then have
  \begin{eqnarray*}
    FA(x,U(y)) &=& \sum_{y' \in M}F(x,y')
    \sum_{y''\in U(y)}A(y',y'')
    \leq \eta\sum_{k=1}^KF(x,U(y_k))\\
    &\leq& \eta\sum_{k=1}^KG(x,U(y_k)).
  \end{eqnarray*}
  Using $G \asymp 1$, we get $G(x,U(y_k))\leq C|U(y_k)|G(x,y)$. Clearly
  $|U(y_k)| \leq C |U(y)|$, so that
  \begin{equation*}
    FA(x,U(y)) \leq CK\eta|U(y)|G(x,y).
  \end{equation*}
  A second application of $G\asymp 1$ yields the claim.\\
  (ii) We can find a constant $K = K(d)$ and a covering of $V_{L+r}$ by
  neighborhoods $U(y_k)$, $y_k \in V_{L+r}$, such that every
  $y\in V_{L+r}$ is contained in at most $K$ many of the sets $U(y_k)$. Using
  $\Phi\asymp 1$, it follows that for $x\in V_{L+r}$,
  \begin{eqnarray*}
    F\Phi(x) = \sum_{y\in V_{L+r}}F(x,y)\Phi(y) &\leq&
    C\sum_{k=1}^{\infty}F(x,U(y_k))\Phi(y_k)\leq
    C\sum_{k=1}^{\infty}G(x,U(y_k))\Phi(y_k)\\
    &\leq& C\sum_{k=1}^{\infty}\sum_{y\in U(y_k)}G(x,y)\Phi(y) \leq CK\sum_{y\in
      V_{L+r}}G(x,y)\Phi(y). 
  \end{eqnarray*}
\end{prooof}
In terms of our specific kernel $\Gamma_{L,r}$, we obtain
\begin{proposition}\ 
\label{def:super-keyest}
Let $A$ be $\eta$-smoothing, and let $F$ be a positive kernel satisfying
$F\preceq \Gamma_{L,r}$.
\begin{enumerate}
\item There exists a constant $C_2 > 0$ not depending on $F$ and $A$ such that
  \begin{equation*}
    FA \preceq C_2\eta\Gamma_{L,r}.
  \end{equation*}
\item If additionally $A(x,y) = 0$ for $x\notin V_L$ and $A(x,U(x)) \leq
  \left(\log (a(x)/20)\right)^{-9}$ for $x \in V_L\backslash \mathcal{E}_1$, then
  there exists a constant $C_3 > 0$ not depending on $F$ and $A$ such that for all $x,z\in V_{L+r}$,
  \begin{equation*}
    FA\Gamma_{L,r}(x,z) \leq C_3\eta^{1/2}\Gamma_{L,r}(x,z).
  \end{equation*}
\end{enumerate}
\end{proposition}
\begin{prooof}
  (i) This is Proposition~\ref{def:super-concatenating} (i) with $G =
  \Gamma$.\\ 
  (ii) We set $B = V_L\backslash \mathcal{E}_1$ and split into
  \begin{equation}
    \label{eq:super-keyest-splitting1}
    FA\Gamma = F1_{\mathcal{E}_1}A\Gamma + F1_BA\Gamma.
  \end{equation}
  Let $x,z\in V_{L+r}$ be fixed, and consider first
  $F1_{\mathcal{E}_1}A\Gamma(x,z)$. Using $\Gamma\asymp 1$, $A\Gamma(y,z) \leq
  C\eta\Gamma(y,z)$. As $\Gamma(\cdot,z) \asymp 1$ and $F1_{\mathcal{E}_1} \preceq
  \Gamma1_{\mathcal{E}_2}$, we get by Proposition~\ref{def:super-concatenating} ii)
  \begin{equation*}
    F1_{\mathcal{E}_1}A\Gamma(x,z) \leq C\eta \Gamma1_{\mathcal{E}_2}\Gamma(x,z).
  \end{equation*}
  Setting $\mathcal{E}_2^1 = \{y\in \mathcal{E}_2 : |y-z| \geq |x-z|/2\}$, $\mathcal{E}_2^2=
  \mathcal{E}_2\backslash \mathcal{E}_2^1$, we split further into
  \begin{equation*}
    \Gamma1_{\mathcal{E}_2}\Gamma = \Gamma1_{\mathcal{E}_2^1}\Gamma +
    \Gamma1_{\mathcal{E}_2^2}\Gamma.
  \end{equation*}
  If $y\in \mathcal{E}_2^1$, then $\Gamma(y,z) \leq C\Gamma(x,z)$. By
  Lemma~\ref{def:super-gammalemma} (iv), $\Gamma(x, \mathcal{E}_2) \leq
  C$. Together we obtain
  \begin{equation*}
    \Gamma1_{\mathcal{E}_2^1}\Gamma(x,z) \leq C\Gamma(x,z).
  \end{equation*}
  If $y \in \mathcal{E}_2^2$, then $\Gamma(x,y) \leq C\frac{a(z)^2}{r^2}\Gamma(x,z)$
  and $\Gamma^{(1)}(y,z) \leq C\frac{r^2}{a(z)^2}\Gamma^{(1)}(z,y)$, whence
  \begin{equation*}
    \Gamma1_{\mathcal{E}_2\backslash
      \mathcal{E}_2^2}\Gamma(x,z) \leq C\Gamma(x,z)\Gamma^{(1)}(z,\mathcal{E}_2) \leq C\Gamma(x,z).
  \end{equation*}
  We therefore have shown that
  \begin{equation*}
    F1_{\mathcal{E}_1}A\Gamma(x,z) \leq C\eta\Gamma(x,z).
  \end{equation*} 
  To handle the second summand of~\eqref{eq:super-keyest-splitting1}, set
  $\sigma(y) =
  \min\left\{\eta,\left(\log a(y)\right)^{-9}\right\}$, $y\in
  V_{L+r}$.  Clearly, $1_BA\Gamma(y,z) \leq C\sigma(y)\Gamma(y,z)$ and $F1_B
  \preceq \Gamma1_{V_L}$. Furthermore, $\sigma(\cdot)\Gamma(\cdot,z)\asymp
  1$, so that by Proposition~\ref{def:super-concatenating} ii)
  \begin{equation*}
    F1_BA\Gamma(x,z) \leq C\Gamma1_{V_L}\sigma\Gamma(x,z).
  \end{equation*} 
  Consider $D^1 = \{y\in V_L : |y-z| \geq |x-z|/2\}$, $D^2 =
  V_L\backslash D^1$ and split into
  \begin{equation*}
    \Gamma1_{V_L}\sigma\Gamma = \Gamma1_{D^1}\sigma\Gamma + \Gamma1_{D^2}\sigma\Gamma.
  \end{equation*} 
  If $y\in D^1$, then $\Gamma(y,z) \leq
  C\max\left\{1,\frac{\dtL(y)}{\dtL(x)}\right\}\Gamma(x,z)$, implying
  $\Gamma1_{D^1}\sigma\Gamma(x,z) \leq C\eta^{1/2}\Gamma(x,z)$ if we prove
  \begin{equation}
    \label{eq:super-keyest-toprove}
    \sum_{y\in
      V_L}\max\left\{1,\frac{\dtL(y)}{\dtL(x)}\right\}\Gamma(x,y)\sigma(y) \leq C\eta^{1/2}.
  \end{equation}
  To this end, we treat the summation over $S^1=\{y \in V_L: \dL(y) \leq
  2s_L\}$ and $S^2= V_L\backslash S_1$ separately. If $y\in S^2$, then
  $a(y) = s_L$. Estimating $\Gamma$ by $\Gamma^{(1)}$ and $\dtL(y)$, $\dtL(x)$
  simply by $L$, we get
  \begin{equation}
    \label{eq:super-keyest-s2}
    \sum_{y\in
      S^2}\max\left\{1,\frac{\dtL(y)}{\dtL(x)}\right\}\Gamma(x,y)\sigma(y) \leq
    \frac{C}{(\log L)^3}\sum_{y\in V_{2L}}\frac{1}{\left(s_L +
        |y|\right)^d}\leq \frac{C\log\log L}{(\log L)^3}.
  \end{equation}
  If $y\in S^1$, we estimate $\Gamma$ again by $\Gamma^{(1)}$ and split the
  summation into the layers $\mathcal{L}_j$, $j = 0,\ldots,2s_L$. On
  $\mathcal{L}_j$, $\sigma(y) \leq C\min\left\{\eta,(\log
    (j+ 1))^{-9}\right\}$. Thus, by Lemma~\ref{def:super-gammalemma} (iii),
  \begin{eqnarray*}
    \lefteqn{\sum_{y\in S^1}\max\left\{1,\frac{\dtL(y)}{\dtL(x)}\right\}
    \Gamma(x,y)\sigma(y)}\nonumber\\ 
    &\leq& C\sum_{j=0}^{2s_L}\sum_{y\in\mathcal{L}_j}\max\left\{1,\frac{\dtL(x)}{a(y)}\right\}\frac{\min\left\{\eta,(\log
        (j+1))^{-9}\right\}}{(a(y)+|x-y|)^d}\\ 
    &\leq& C\sum_{j=0}^{2s_L}\frac{\min\left\{\eta,(\log(
        j+1))^{-9}\right\}}{j\vee r}\leq C\eta^{1/2}.
  \end{eqnarray*}
  Together with~\eqref{eq:super-keyest-s2}, we have
  proved~\eqref{eq:super-keyest-toprove}. It remains to bound the
  term $\Gamma1_{D^2}\sigma\Gamma(x,z)$.  But if $y\in D^2$, then
  \begin{equation*}
    a(y) + |x-y| \geq a(y) + \frac{1}{2}|x-z| \geq a(z) - \frac{1}{2}|y-z| + \frac{1}{2}|x-z| \geq
    \frac{1}{4}\left(a(z) + |x-z|\right),
  \end{equation*}
  whence $\Gamma(x,y) \leq
  C\frac{a(z)^2}{a(y)^2}\max\left\{1,\frac{\dtL(y)}{\dtL(z)}\right\}\Gamma(x,z)$.
  Using Lemma~\ref{def:super-gammalemma} (i), we have
\begin{equation*} 
  \frac{a(z)^2}{a(y)^2}\Gamma(y,z) \leq C\Gamma(z,y),
  \end{equation*} 
so that
  $\Gamma1_{D^2}\sigma\Gamma(x,z) \leq C\eta^{1/2}\Gamma(x,z)$ follows
  again from~\eqref{eq:super-keyest-toprove}.
 \end{prooof}
Now we have collected all ingredients to finally prove part (ii) of our main Lemma~\ref{def:superlemma}.
\begin{prooof2}{\bf of Lemma~\ref{def:superlemma} (ii):}
  As already remarked, we only have to prove the statement involving $\Gtg$. The
  perturbation expansion~\eqref{eq:prel-pbe3} yields
  \begin{equation*}
    \Gtg = \gt\sum_{m=0}^\infty(R\gt)^m\sum_{k=0}^\infty\Delta^k,
  \end{equation*}
  where $\Delta = 1_{V_{L+r}}(\Ptg-\pt)$, $R =
  \sum_{k=1}^\infty\Delta^k\pt$. With the constants $C_1$ of
  Lemma~\ref{def:superlemma} (i) and $C_2,C_3$ of
  Proposition~\ref{def:super-keyest} we choose
\begin{equation*}
\delta \leq
  \frac{1}{16}\left(\frac{1}{C_2\vee C_1^2C_3^2}\right).
\end{equation*}
 From Lemma~\ref{def:superlemma} (i) and
 Proposition~\ref{def:super-keyest} (i) with $A = |\Delta|$, 
 $\eta = \delta$ we then deduce that $\gt|\Delta| \preceq (C_1/2)\Gamma$,
 and, by iterating,
\begin{equation*}
  \sum_{k=1}^\infty\gt|\Delta|^{k-1}\preceq 2C_1 \Gamma.
\end{equation*}
Furthermore, by part (ii) of Proposition~\ref{def:super-keyest} with $A=|\Delta\pt|$ and Lemma~\ref{def:superlemma} (i),
\begin{equation*}
  \sum_{k=1}^\infty\gt|\Delta|^{k-1}\left|\Delta\pt\right|\gt \preceq (C_1/2)\Gamma.
\end{equation*} 
Repeating this procedure shows that for $m\in\mathbb{N}$,
\begin{equation*}
  \gt(|R|\gt)^m \preceq C_12^{-m}\Gamma.
\end{equation*}
Finally, by a further application of Proposition~\ref{def:super-keyest}
(i),
\begin{equation*}
  \gt\sum_{m=0}^\infty(|R|\gt)^m\sum_{k=0}^\infty|\Delta|^k \preceq 4C_1\Gamma.
\end{equation*}
This proves the lemma.
\end{prooof2}
\subsection{Difference estimates}
\label{super-difference}
The results from the preceding section enable us to prove some difference
estimates on the coarse grained Green's functions, which will be used in
the part on mean sojourn times. The reader who is only interested in the
exit measures may skip this section.
\begin{lemma}
\label{def:super-greendifference}
There exists a constant $C >0$ such that
\begin{enumerate}
\item
\begin{equation*}
\sup_{x,x'\in V_L: |x-x'|\leq s_L}\sum_{y\in V_L}\left|\gh_{L,r}(x,y)-\gh_{L,r}(x',y)\right|
\leq C(\log\log L)(\log L)^3.
\end{equation*}
\item For $\delta > 0$ small,
\begin{equation*}
\sup_{x,x'\in V_L: |x-x'|\leq s_L}\sum_{y\in V_L}\left|\Ghg_{L,r_L}(x,y)-\Ghg_{L,r_L}(x',y)\right|
\leq C(\log\log L)(\log L)^3.
\end{equation*}
\end{enumerate}
\end{lemma}
\begin{prooof}
(i) 
Set $m = s_L/20$. Recall the definitions of $\ph_m$ and $\gh_m$ from
Section~\ref{clt}. We write
\begin{eqnarray}
\label{eq:super-greendifference-eq0}
  \lefteqn{\sum_{y\in V_L}\left|\gh(x,y)-\gh(x',y)\right|}\nonumber\\
  &\leq& \sum_{y\in
      V_L}\left|\left(\gh-\gh_m\right)(x,y)\right| + \sum_{y\in
      V_L}\left|\gh_m(x,y)-\gh_m(x',y)\right| + \sum_{y\in
      V_L}\left|\left(\gh_m-\gh\right)(x',y)\right|.\nonumber\\
\end{eqnarray}
If $x\in V_L\backslash \sh_L(2s_L)$, we have $\ph(x,\cdot)= \ph_m(x,\cdot)$. Clearly,
$\sup_{x\in V_L}\gh_m(x,\sh_L(2s_L)) \leq C$. Thus, with $\Delta=
1_{V_L}\left(\ph_m-\ph\right)$, expansion~\eqref{eq:prel-pbe1} and
Lemma~\ref{def:super-gammalemma} yield (remember $\gh\preceq C\Gamma$) 
\begin{eqnarray*} \sum_{y\in V_L}\left|(\gh_m-\gh)(x,y)\right|&=&
  \sum_{y\in V_L}|\gh_m\Delta\gh(x,y)|\\ 
&\leq& 2\,\gh_m(x,\sh_L(2s_L))\sup_{v\in
    \sh_L(3s_L)}\gh(v,V_L) \leq C(\log L)^3.
\end{eqnarray*}
It remains to handle the middle term
of~\eqref{eq:super-greendifference-eq0}. By~\eqref{eq:clt-greenonball}, 
\begin{eqnarray*}
  \lefteqn{\gh_m(x,y)-\gh_m(x',y)}\\
  & =& \gh_{m,\mathbb{Z}^d}(x,y)-
  \gh_{m,\mathbb{Z}^d}(x',y) +
  \Erw_{x',\ph_m}\left[\gh_{m,\mathbb{Z}^d}(X_{\tau_L},y)\right] -
  \Erw_{x,\ph_m}\left[\gh_{m,\mathbb{Z}^d}(X_{\tau_L},y)\right].  
\end{eqnarray*}
Using Proposition~\ref{def:super-behaviorgreen}, it follows that for $|x-x'|
\leq s_L$,
\begin{equation*}
  \sum_{y\in
      V_L}\left|\gh_{m,\mathbb{Z}^d}(x,y)-\gh_{m,\mathbb{Z}^d}(x',y)\right| \leq
  C(\log L)^3. 
\end{equation*}
At last, we claim that
\begin{equation}
\label{eq:super-greendifference-eq1}
  \sum_{y\in V_L}\left|\Erw_{x',\ph_m}\left[\gh_{m,\mathbb{Z}^d}(X_{\tau_L},y)\right] -
    \Erw_{x,\ph_m}\left[\gh_{m,\mathbb{Z}^d}(X_{\tau_L},y)\right]\right|
  \leq C(\log\log L){(\log L)}^3.  
\end{equation}
Since $|x-x'| \leq m$, we can define on the same probability space, whose
probability measure we denote by $\mathbb{Q}$, a random walk $(Y_n)_{n\geq
  0}$ starting at $x$ and a random walk $(\tilde{Y}_n)_{n\geq 0}$
starting at $x'$, both moving according to $\ph_m$ on $\mathbb{Z}^d$,
such that for all times $n$, $|Y_n - \tilde{Y}_n| \leq s_L$. However, with
$\tau= \inf\{n \geq 0 : Y_n \notin V_L\}$, $\tilde{\tau}$ the same for
$\tilde{Y}_n$, we cannot deduce that
$|Y_{\tau}-\tilde{Y}_{\tilde{\tau}}| \leq s_L$, since it is
possible that one of the walks, say $Y_n$, exits $V_L$ and then moves far
away from the exit point, while staying close to both $V_L$ and the walk
$\tilde{Y}_n$, which might still be inside $V_L$. In order to show that
such an event has a small probability, we argue in a similar way to~\cite{LawLim}, Proposition 7.7.1. Define
\begin{equation*}
\sigma(s_L) = \inf\left\{n\geq 0 : Y_n\in \sh_L(s_L)\right\},
\end{equation*}
and analogously $\tilde{\sigma}(s_L)$. Let $\vartheta =
\sigma(s_L)\wedge\tilde{\sigma}(s_L)$.  Since
$|Y_{\vartheta}-\tilde{Y}_{\vartheta}| \leq s_L$, 
\begin{equation*}
\sigma(2s_L) \vee \tilde{\sigma}(2s_L) \leq \vartheta.
\end{equation*}
For $k\geq 1$, we introduce the events
\begin{eqnarray*}
B_k&=&\left\{\left|Y_i-Y_{\sigma(2s_L)}\right| > ks_L\mbox{ for
      all } i=\sigma(2s_L),\ldots,\tau\right\},\\
\tilde{B}_k&=&\left\{\left|\tilde{Y}_i-\tilde{Y}_{\tilde{\sigma}(2s_L)}\right| > ks_L\mbox{ for
      all } i=\tilde{\sigma}(2s_L),\ldots,\tilde{\tau}\right\}.
\end{eqnarray*}
By the strong Markov property and the gambler's ruin estimate
of~\cite{LawLim}, p. 223 (7.26),
\begin{equation*}
\mathbb{Q}\left(B_k\cup \tilde{B}_k\right) \leq C_1/k
\end{equation*}
for some $C_1>0$ independent of $k$. Applying the triangle inequality to
\begin{equation*}
  Y_{\tau} - \tilde{Y}_{\tilde{\tau}} =
  \left(Y_{\tau}-Y_{\vartheta}\right) +
  \left(Y_{\vartheta}-\tilde{Y}_{\vartheta}\right) + \left(\tilde{Y}_{\vartheta}-\tilde{Y}_{\tilde{\tau}}\right),
\end{equation*}
we deduce, for $k \geq 3$,  
\begin{equation*}
\mathbb{Q}\left(\left|Y_\tau - \tilde{Y}_{\tilde{\tau}}\right|\geq
  ks_L\right) \leq 2C_1/(k-1).
\end{equation*}
Since $|Y_\tau-\tilde{Y}_{\tilde{\tau}}| \leq 2(L +s_L) \leq 3L $, it follows that
\begin{equation*}
\mathbb{E}_{\mathbb{Q}}\left[\left|Y_\tau-\tilde{Y}_{\tilde{\tau}}\right|\right] \leq
\sum_{k=1}^{3L}\mathbb{Q}\left(\left|Y_\tau-\tilde{Y}_{\tilde{\tau}}\right|
  \geq k\right)\leq C(\log\log L)s_L.
\end{equation*}
Also, for $v,w$ outside and $y$ inside $V_L$,
\begin{equation*}
  \left|\frac{1}{|v-y|^{d-2}}-\frac{1}{|w-y|^{d-2}}\right|
  \leq 
  C\frac{|v-w|}{(L+1-|y|)^{d-1}}.  
\end{equation*}
By Proposition~\ref{def:super-behaviorgreen},~\eqref{eq:super-greendifference-eq1} now follows from summing over
$y\in V_L$.\\
(ii) Let $x, x'\in V_L$ with $|x-x'|\leq s_L$ and set $\Delta =
1_{V_L}(\Phg-\ph)$. With $B = V_L\backslash\sh_L(2r_L)$,  
\begin{equation*} 
\Ghg = \gh 1_{B}\Delta\Ghg +\gh 1_{B^c}\Delta\Ghg +\gh.
\end{equation*}
Replacing successively $\Ghg$ in the first summand on the right-hand
side,
\begin{equation*}
\Ghg= \sum_{k=0}^\infty{\left(\gh1_B\Delta\right)}^k\gh +
    \sum_{k=0}^\infty{\left(\gh1_B\Delta\right)}^k\gh1_{B^c}\Delta\Ghg
    = F+F1_{B^c}\Delta\Ghg,  
\end{equation*}
where we have set $F=\sum_{k=0}^\infty{\left(\gh1_B\Delta\right)}^k\gh$.
With $R = \sum_{k=1}^{\infty}(1_B\Delta)^k\ph$,
expansion~\eqref{eq:prel-pbe3} gives
\begin{equation}
\label{eq:super-diffF1}
F = \gh\sum_{m=0}^\infty
(R\gh)^m\sum_{k=0}^\infty\left(1_B\Delta\right)^k =
\gh\sum_{k=0}^\infty\left(1_B\Delta\right)^k + \gh RF.
\end{equation}
Following the proof of Lemma~\ref{def:superlemma} (ii), one deduces $|F|\preceq C\Gamma$.
By Lemma~\ref{def:super-gammalemma} (iv) and (v), we see that for large $L$, uniformly in $x\in V_L$,  
\begin{equation*}
|F1_{B^c}\Delta\Ghg(x,V_L)| \leq
C\Gamma(x,\sh_L(2r_L))\sup_{v\in\sh_L(3r_L)}\Gamma(v,V_L) \leq C\log\log L.
\end{equation*}
Therefore,
\begin{equation*}
  \sum_{y\in V_L}\left|\Ghg(x,y)-\Ghg(x',y)\right| \leq C\log\log L +
  \sum_{y\in V_L}\left|F(x,y)-F(x',y)\right|.
\end{equation*}
Using~\eqref{eq:super-diffF1} and twice part (i),
\begin{eqnarray}
\label{eq:super-diffF2}
\lefteqn{\sum_{y\in V_L}\left|F(x,y)-F(x',y)\right|}\nonumber\\ 
&\leq & \sum_{y\in
  V_L}\left|\gh\sum_{k=0}^\infty\left(1_B\Delta\right)^k(x,y)
  -\gh\sum_{k=0}^\infty\left(1_B\Delta\right)^k(x',y)\right| + \sum_{y\in V_L}\left|\gh RF(x,y) - \gh RF(x',y)\right|.\nonumber\\
\end{eqnarray}
The first expression on the right is estimated by
\begin{equation*}
\sum_{y\in V_L}\left|\sum_{w\in
    V_L}\left(\gh(x,w)-\gh(x',w)\right)\sum_{k=0}^\infty\left(1_B\Delta\right)^k(w,y)\right|
\leq C(\log\log L)(\log L)^3,
\end{equation*}
where we have used part (i) and the fact that $||1_B\Delta(w,\cdot)||_1\leq \delta$.
The second factor of~\eqref{eq:super-diffF2} is again bounded by (i) and the fact
that for $u \in V_L$,
\begin{eqnarray*}
\sum_{y\in V_L}|RF(u,y)| &=& \sum_{y\in
  V_L}\left|\sum_{k=1}^\infty\left(1_B\Delta\right)^k\ph F(u,y)\right|\\
& \leq& \sum_{k=0}^\infty ||1_B\Delta(u,\cdot)||_1^k\sup_{v\in
  B}||1_B\Delta\ph(v,\cdot)||_1\sup_{w\in V_L}\sum_{y\in V_L}|F(w,y)|\\ 
&\leq& C(\log L)^{-9+6} = C (\log L)^{-3}.
\end{eqnarray*}
Altogether, this proves part (ii).
\end{prooof}

\subsection{Modified transitions on environments bad on level $4$}
\label{super-cgmod}
We shall now describe an environment-depending second version of the coarse
graining scheme, which leads to modified transition kernels $\Pho_{L,r}$,
$\Phgo_{L,r}$, $\pho_{L,r}$ on ``really bad'' environments.

Assume $\omega\in\onebad$ is bad on level $4$, with $\badP_L(\omega)\subset
V_{L/2}$. Then there exists $D=V_{4h_L(z)}(z)\in\mathcal{D}_L$ with
$\badP_L(\omega) \subset D$, $z\in V_{L/2}$. On $D$, $c\,r_L\leq
h_{L,r}(\cdot) \leq C\,r_L$. By Lemma~\ref{def:superlemma} and the
definition of $\Gamma_{L,r}$, it follows easily that we can find a constant
$K_1\geq 2$, depending only on $d$, such that whenever $|x-y| \geq
K_1h_{L,r}(y)$ for some $y\in \badP_L$, we have

\begin{equation}
\label{eq:super-kernelmod}
  \Ghg_{L,r}(x,\badP_L) \leq C\Gamma_{L,r}(x,D)\leq \frac{1}{10}.
\end{equation}
On such $\omega$, we let $t(x) = K_1h_{L,r}(x)$ and define on $V_L$, 
\begin{equation*}
\Pho_{L,r}(x,\cdot) = \left\{\begin{array}{l@{\quad}l}
\mbox{ex}_{V_{t(x)}(x)}\left(x,\cdot;\Ph_{L,r}\right)& \mbox{for }x\in \badP_L\\
\Ph_{L,r}(x,\cdot)& \mbox{otherwise}\end{array}\right..      
\end{equation*}

By replacing $\Ph$ by $\ph$ on the right side, we define
$\pho_{L,r}(x,\cdot)$ in an analogous way. Note that $\pho_{L,r}$
depends on the environment. We work again with a goodified version
of $\Pho_{L,r}$,
\begin{equation*}
  \Phgo_{L,r}(x,\cdot) = \left\{\begin{array}{l@{\quad}l}
\mbox{ex}_{V_{t(x)}(x)}\left(x,\cdot;\Phg_{L,r}\right)& \mbox{for }x\in \badP_L\\
\Phg_{L,r}(x,\cdot)& \mbox{otherwise}\end{array}\right..      
\end{equation*}  

For all other environments falling not into the above class, we change nothing and put $\Pho_{L,r}= \Ph_{L,r}$,
$\Phgo_{L,r}=\Phg_{L,r}$, $\pho_{L,r}=\ph_{L,r}$. This defines
$\Pho_{L,r}$, $\Phgo_{L,r}$ and $\pho_{L,r}$ on
all environments. 
We write $\Gho_{L,r}$, $\Ghgo_{L,r}$, $\gho_{L,r}$ for the Green's
functions  corresponding to $\Pho_{L,r}$, $\Phgo_{L,r}$ and $\pho_{L,r}$. 

\begin{figure}
\begin{center}\parbox{5.5cm}{\includegraphics[width=5cm]{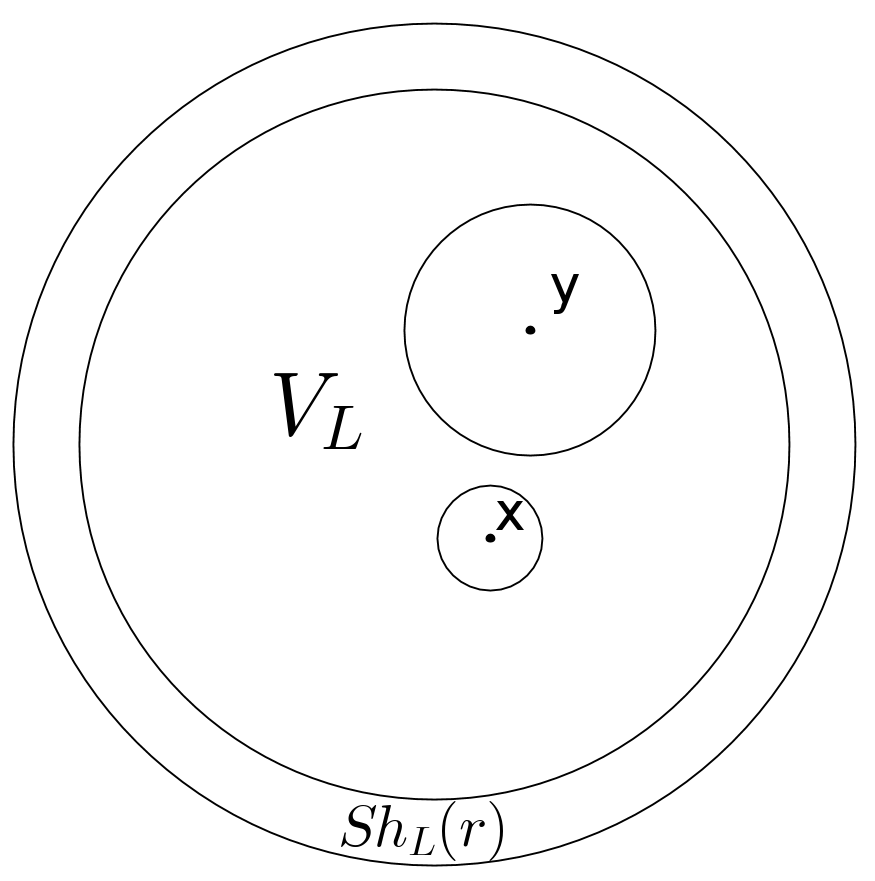}}
\parbox{9.5cm}{
\caption{$\omega\in\onebad$ bad on level $4$, with
  $\badP_L\subset V_{L/2}$. The point $x$ is ``good'', so the
  coarse graining radii do not change at $x$. The point $y$ is ``bad''. Therefore, 
  at $y$, the exit distribution is taken from the larger set $V_{t(y)}(y)$, where $t(y)=K_1h_{L,r}(y)$.}}
\end{center}
\end{figure}

\subsubsection{Some properties of the new transition kernels}
The following observations can be read off the definition and will be
tacitly used below.
\begin{itemize}
\item On environments which are good or bad on level
  at most $3$, the new kernels agree with the old ones, and so do their
  Green's functions, i.e. $\Gh_{L,r} = \Gho_{L,r}$ and $\Ghg_{L,r}  =
  \Ghgo_{L,r}$. On $\good$ with the choice $r=r_L$, we have equality of all
  four Green's functions. 
\item If $\omega$ is not bad on level $4$ with $\badP_L\subset V_{L/2}$, then
\begin{equation*}
1_{V_L}(\Pho_{L,r}-\Phgo_{L,r})  = 1_{V_L}(\Ph_{L,r}-\Phg_{L,r})=
1_{\badPo_{L,r}}(\Ph-\ph).
\end{equation*}
This will be used in Sections~\ref{smv-onebad-exits} and~\ref{nonsmv-exits}.
\item In contrast to $\ph_{L,r}$, the kernel $\pho_{L,r}$ depends on the
  environment, too. However, $\Pho_{L,r}$, $\Phgo_{L,r}$ and $\pho_{L,r}$
  do not change the exit measure from $V_L$, i.e. for example,
 \begin{equation*}
 \mbox{ex}_{V_L}\left(x,\cdot;\Phgo_{L,r}\right) =  \mbox{ex}_{V_L}\left(x,\cdot;\Phg_{L,r}\right).
  \end{equation*}
\item The old transition kernels are finer in the sense that the (new) Green's
  functions $\Gho$, $\Ghgo$, $\gho$ are pointwise bounded
  from above by $\Gh$, $\Ghg$ and $\gh$,
  respectively. In particular, we obtain with the same constants as in Lemma~\ref{def:superlemma},
\end{itemize}
\begin{lemma}\ 
\label{def:superlemma2}
\begin{enumerate}
\item 
\begin{equation*}
      \gho_{L,r}\preceq C_1\Gamma_{L,r}.
    \end{equation*}
\item  For $\delta>0$ small,  
\begin{equation*}
  \Ghgo_{L,r} \preceq C\Gamma_{L,r}.
\end{equation*}  
\end{enumerate}
\end{lemma}

%
For the new goodified Green's function, we have
\begin{corollary}
There exists a constant $C > 0$ such that for $\delta > 0$ small, 
  \label{def:super-cor}
  \begin{enumerate}
  \item On $\onebad$, if $\badP_L\cap\sh_L(r_L)=\emptyset$ or for general
    $\badP_L$ in the case $r=r_L$, 
    \begin{equation*}
    \sup_{x\in V_L}\Ghgo_{L,r}(x,\badP_L)\leq C.
    \end{equation*}
    On $\onebad$, if $\badP_L\not\subset V_{L/4}$, then, with
    $t=\dist(\badP_L,\partial V_L)$,
    \begin{equation*} 
    \sup_{x\in
      V_{L/5}}\Ghgo_{L,r}(x,\badP_L) \leq C\left(\frac{s_L\wedge(t\vee
        r_L)}{L}\right)^{d-2}.
    \end{equation*}
  \item On ${(\bbad)}^c$, $\sup_{x\in
      V_{2L/3}}\Ghgo_{L,r}\left(x,\badP_{L,r}^\partial\right) \leq C(\log
    r)^{-1/2}$.
  \item For $\omega \in \onebad$ bad on level at most $3$ with
    $\badP_L\cap\sh_L(r_L)=\emptyset$, or for $\omega$ bad on level $4$
    with $\badP_L\subset V_{L/2}$, putting $\Delta= 1_{V_L}(\Pho_{L,r}-\Phgo_{L,r})$,
    \begin{equation*}
      \sup_{x\in V_L}\sum_{k=0}^\infty
      {\left|\left|\left(\Ghgo_{L,r}1_{\badP_L}\Delta\right)^k(x,\cdot)\right|\right|}_1\leq C.
    \end{equation*}
   \end{enumerate}
\end{corollary}

\begin{prooof}
  (i) The set $\badP_L$ is contained in a neighborhood $D\in\mathcal{D}_L$. As
  $\Ghgo \preceq C\Gamma$, we have
  \begin{equation}
  \label{eq:super-cor-1}
    \Ghgo\left(x,\badP_L\right) \leq C\Gamma^{(2)}(x,D).
  \end{equation}
  From this, the first statement of (i) follows. Now let $x$ be inside
  $V_{L/5}$, and $\badP_L\not\subset V_{L/4}$. If the midpoint $z$ of $D$
  can be chosen to lie inside $V_L\backslash\sh(r_L)$, $a(\cdot)/h_L(z)$
  and $h_L(z)/a(\cdot)$ are bounded on $D$. Then, the second statement of (i) is
  again a consequence of~\eqref{eq:super-cor-1}. If $z\in \sh(r_L)$, we have
  \begin{eqnarray*}
    \Ghgo\left(x,\badP_L\right) \leq C\Gamma^{(1)}(x,D) &\leq& 
    C\sum_{j=0}^{2r_L}\sum_{y\in\mathcal{L}_j\cap D}\frac{L}{a(y)L^{d}}\\
  &\leq& CL^{-d+1}\sum_{j=0}^{2r_L}\frac{r_L^{d-1}}{j\vee r} \leq C(\log
  L)\left(\frac{r_L}{L}\right)^{d-1}.
  \end{eqnarray*}
  (ii) Recall the notation of Section~\ref{s2}. In order to bound
  $\sup_{x\in V_{2L/3}}\Ghgo(x,\badP^\partial_{L,r})$, we look at the
  different bad sets $D_{j,r} \in\mathcal{Q}_{j,r}$ of layer $\Lambda_j$,
  $0\leq j\leq J_1$. Estimating $\Ghgo$ by $\Gamma^{(1)}$, we have
  \begin{equation*}
    \Ghgo\left(x,D_{j,r}\right) \leq C (r2^j)^{d-1}L^{-d+1}.
  \end{equation*}
  On ${(\bbad)}^c$, the number of bad sets in layer $\Lambda_j$ is bounded by
  \begin{equation*}
  C(\log r +j)^{-3/2}{(L/(r2^j))}^{d-1}.
  \end{equation*}
  Therefore, 
  \begin{equation*}
    \Ghgo\left(x,\badP^\partial_{L,r}\cap\Lambda_j\right) \leq C(\log r + j)^{-3/2}.
  \end{equation*}
  Summing over $0\leq j\leq J_1$, this shows
  \begin{equation*}  
    \Ghgo\left(x,\badP^\partial_{L,r}\right) \leq C(\log r)^{-1/2}. 
 \end{equation*}
 (iii) Assume $\omega\in\good$ or $\omega$ is bad on level $i=1,2,3$. Then
 $1_{\badP_L}\Delta = 1_{\badP_L}(\Ph-\ph)$. Further, if $\badP_L\cap\sh_L(r_L) =
 \emptyset$, we have $||\Ghgo1_{\badP_L}\Delta(x,\cdot)||_1 \leq
 C\delta$. By choosing $\delta$ small enough, the claim follows. If
 $\omega$ is bad on level $4$ and $\badP_L\subset V_{L/2}$, we do not gain
 a factor $\delta$ from $||1_{\badP_L}\Delta(y,\cdot)||_1$. However, thanks
 to our modified transition kernels, using~\eqref{eq:super-kernelmod},
 $||1_{\badP_L}\Delta\Ghgo1_{\badP_L}(y,\cdot)||_1 \leq 1/5$ (recall that
 $\Ghgo\leq \Ghg$ pointwise), so that (ii) follows in this case, too.
\end{prooof}
\begin{remark}
  \label{def:super-remark}
  All $\delta_0 > 0$ and $L_0$ appearing in the next sections are
  understood to be chosen in such a way that if we take $\delta \in
  (0,\delta_0]$ and $L\geq L_0$, then the conclusions of
  Lemmata~\ref{def:superlemma},~\ref{def:super-greendifference},~\ref{def:superlemma2} and Corollary~\ref{def:super-cor} are
  valid.
\end{remark}

   \section{Globally smoothed exits}
\label{smv-exits}
The aim here is to establish the estimates for the smoothed difference
$D_{L,\psi}^\ast$ which are required to propagate condition $\ctwo\left(\delta,L\right)$.
For the entire section, we choose $r=r_L$.  
We start with an auxiliary statement which will be of constant use.
\begin{lemma}
\label{def:est-deltaphi}
Let $\psi\in\mathcal{M}_L$ and set $\Delta =
1_{V_L}(\Phg_{L,r_L}-\ph_{L,r_L})$. Then, for some $C>0$,
\begin{equation*}
\sup_{x\in V_L}\sup_{z\in\mathbb{Z}^d}\left|\Delta\phi_{L,\psi}(x,z)\right| \leq C(\log L)^{-12}L^{-d}.
\end{equation*}
\end{lemma}
\begin{prooof}
Using $\Delta\phi = \Delta\ph\phi$ and the fact that $\Delta\ph(x,\cdot)$
sums up to zero, 
\begin{eqnarray*}
\left|\Delta\phi(x,z)\right|
&=& \left|\sum_{y\in
  V_L\cup\partial
  V_L}\Delta\ph(x,y)\left(\phi(y,z)-\phi(x,z)\right)\right|\\
&\leq& ||\Delta\ph(x,\cdot)||_1\sup_{y: |\Delta\ph(x,y)| > 0}|\phi(y,z)-\phi(x,z)|.
\end{eqnarray*}
For $x\in V_L\backslash\sh_L(2r_L)$, we have by definition 
$||\Delta\ph(x,\cdot)||_1 \leq C(\log L)^{-9}$. Further, notice that
$|\Delta\ph(x,y)| > 0$ implies $|y-x|\leq s_L$. Bounding
$|\phi(y,z)-\phi(x,z)|$ by Lemma~\ref{def:est-phi} (iii), the statement
follows for those $x$. If $x\in\sh_L(2r_L)$, we simply bound
$||\Delta\ph(x,\cdot)||_1$ by $2$. Now we can restrict the supremum to
those $y\in V_L$ with $|x-y| \leq 3r_L$, so the claim follows again from
Lemma~\ref{def:est-phi} (iii).
\end{prooof}
\subsection{Estimates on ``goodified'' environments}
The following Lemma~\ref{def:lemma-goodpart} compares the ``goodified'' smoothed exit
distribution with that of simple random walk.  In particular, it
provides an estimate for $D_{L,\psi}^\ast$ on $\good$.
Here we will work with the transition kernels 
$\Ph_{L,r_L}$, $\Phg_{L,r_L}$ and $\ph_{L,r_L}$. For the
goodified exit measure from $V_L$ we write
\begin{equation*}
\Pg_L=
 \mbox{ex}_{V_L}\left(x,\cdot;\Phg_{L,r}\right). 
  \end{equation*}
\begin{lemma}
\label{def:lemma-goodpart}
Assume $\czero$. There exist $\delta_0 > 0$ and $L_0>0$ such
that if $\delta\in(0,\delta_0]$ and $L\geq L_0$, then for $\psi\in\mathcal{M}_L$, 
\begin{equation*}
  \pP\left(\sup_{x\in V_L}||(\Pi_L^g - \pi_L)\ph_{\psi}(x,\cdot)||_1 \geq (\log
    L)^{-(9 + 1/6)}\right) \leq \exp\left(-(\log L)^{7/3}\right).
\end{equation*}
\end{lemma}
\begin{prooof}
Clearly, the claim follows if we show
\begin{equation}
\label{eq:smvgoodenv-toshow}
\sup_{x\in V_L}\sup_{z\in\mathbb{Z}^d}\pP\left(\left|(\Pg - \pi)\ph_{\psi}(x,z)\right| \geq (\log
L)^{-(9+1/5)}L^{-d}\right) \leq \exp\left(-(\log L)^{5/2}\right).
\end{equation}
Using the abbreviations $\phi = \pi\ph_{\psi}$, $\Delta =
1_{V_L}(\Phg-\ph)$, we start with the perturbation expansion
\begin{equation*}
(\Pg-\pi)\hat{\pi}_{\psi} = \Ghg\Delta\phi.
\end{equation*}
Set $S = \sh_L(2L/(\log L)^2)$ and write
\begin{equation}
\label{eq:smvgoodenv-splitting}
\Ghg\Delta\phi =
\Ghg1_{S}\Delta\phi + \Ghg1_{S^c}\Delta\phi.
\end{equation}
Using $\Ghg \preceq C\Gamma$, Lemma~\ref{def:super-gammalemma} (iv) (with
$r= r_L$) and Lemma~\ref{def:est-deltaphi} yield the estimate
\begin{equation*}
|\Ghg1_{S}\Delta\phi(x,z)| \leq\sup_{x\in V_L}\Ghg(x,S)\sup_{y\in
  V_L}|\Delta\phi(y,z)|\leq (\log L)^{-19/2}L^{-d}
\end{equation*}
for $L$ large. It remains to bound $|\Ghg1_{{S}^c}\Delta\phi(x,z)|$. With
$B = V_L\backslash\sh_L(2r_L)$,
\begin{equation*}
\Ghg = \gh 1_{B}\Delta\Ghg +\gh 1_{B^c}\Delta\Ghg +\gh.
\end{equation*}
By replacing successively $\Ghg$ in the first summand on the right-hand
side,
\begin{eqnarray}
\label{eq:smvgoodenv-splitting2}
\Ghg1_{{S}^c}\Delta\phi &=& \left(\sum_{k=0}^\infty{\left(\gh1_B\Delta\right)}^k\gh +
    \sum_{k=0}^\infty{\left(\gh1_B\Delta\right)}^k\gh1_{B^c}\Delta\Ghg\right)1_{{S}^c}\Delta\phi\nonumber\\
    &=& F1_{{S}^c}\Delta\phi+F1_{B^c}\Delta\Ghg1_{{S}^c}\Delta\phi,  
\end{eqnarray}
where $F=\sum_{k=0}^\infty{\left(\gh1_B\Delta\right)}^k\gh$.
With $R = \sum_{k=1}^{\infty}(1_B\Delta)^k\ph$,
expansion~\eqref{eq:prel-pbe3} shows
\begin{equation*}
F = \gh\sum_{m=0}^\infty
(R\gh)^m\sum_{k=0}^\infty\left(1_B\Delta\right)^k.
\end{equation*}
From the proof of Lemma~\ref{def:superlemma} (ii) we learn that $|F|\preceq C\Gamma$.
By Lemma~\ref{def:super-gammalemma} (iv),~(v) and again
Lemma~\ref{def:est-deltaphi}, we see that for large $L$, uniformly in $x\in
V_L$ and $z\in\mathbb{Z}^d$,  
\begin{eqnarray*}
|F1_{B^c}\Delta\Ghg1_{S^c}\Delta\phi(x,z)| &\leq&
C\Gamma(x,\sh_L(2r_L))\sup_{v\in\sh_L(3r_L)}\Gamma(v,S^c\cap V_L)\sup_{w\in
  V_L}|\Delta\phi(w,z)|\\ 
&\leq& (\log L)^{-11}L^{-d}.
\end{eqnarray*}
Thus, the second summand
of~\eqref{eq:smvgoodenv-splitting2} is harmless. However, with the first
summand one has to be more careful.  With $\xi =
\gh\sum_{k=0}^\infty(1_B\Delta)^k1_{S^c}\Delta\phi$, we have
\begin{equation*}
F1_{S^c}\Delta\phi= \xi + \gh\sum_{m=0}^\infty(R\gh)^mR\xi = \xi + F1_{B}\Delta\ph\xi.
\end{equation*}
Clearly, $|F1_B\Delta\ph(x,y)|\leq C(\log L)^{-3}$, so it remains
to estimate $\xi(y,z)$, uniformly in $y$ and $z$. Set $N=N(L)=\lceil
\log\log L\rceil$. For small $\delta$, the summands of $\xi$ with $k\geq N$
are readily bounded by   
\begin{eqnarray*}
\sup_{y\in V_L}\sup_{z\in \mathbb{Z}^d}\sum_{k=N}^\infty|\gh(1_B\Delta)^k1_{S^c}\Delta\phi(y,z)| &\leq&
C(\log L)^6\sum_{k=N}^\infty\delta^k(\log L)^{-12}L^{-d}\\ &\leq& (\log L)^{-10}L^{-d}.
\end{eqnarray*}
Now we look at the summands with $k<N$. Since the coarse grained walk
cannot bridge a gap of length $L/(\log L)^2$ in less than $N$ steps, we can
drop the kernel $1_B$. Defining $S' = \sh_L\left(3L/(\log L)^2\right)$, we
thus have
\begin{equation*}
\gh(1_B\Delta)^k1_{S^c}\Delta\phi = \gh1_{S'}\Delta^k1_{S^c}\Delta\phi +
\gh1_{S'^c}\Delta^k1_{S^c}\Delta\phi.
\end{equation*}
The first summand is bounded in the same way as
$\Ghg1_S\Delta\phi$ from~\eqref{eq:smvgoodenv-splitting}. Further, we can
drop the kernel $1_{S^c}$ in the second
summand. Therefore,~\eqref{eq:smvgoodenv-toshow} follows if we show
\begin{equation*}
\sup_{x\in
  V_L}\sup_{z\in \mathbb{Z}^d}\pP\left(\left|\sum_{k=1}^N\gh1_{S'^c}\Delta^k\phi(x,z)\right|
  \geq \frac{1}{2}(\log L)^{-(9+1/5)}L^{-d}\right)\leq \exp\left(-(\log L)^{5/2}\right).  
\end{equation*} 
For $j\in\mathbb{Z}$, consider the interval $I_j
= (jNs_L,(j+1)Ns_L]\subset \mathbb{Z}$. We divide $ S'^c\cap V_L$ into subsets
$W_{\bj} =  (S'^c\cap V_L)\cap \left(I_{j_1}\times\ldots\times I_{j_d}\right)$,
where ${\bj} = (j_1,\ldots,j_d) \in\mathbb{Z}^d$.  
Let $J$ be the set of those $\bj$ for which $W_{\bj} \neq
\emptyset$. Then we can find a constant $K$ depending only on the dimension
and a disjoint partition of $J$ into sets $J_1,\ldots,J_K$, such that for any $1\leq
l\leq K$,
\begin{equation}
\label{eq:smvgoodenv-distance}
\bj,\bj' \in J_l,\ \bj \neq \bj' \Longrightarrow
\dist(W_{\bj},W_{\bj'}) > Ns_L.
\end{equation}
For $x\in V_L$, $z\in\mathbb{Z}^d$, we set 
\begin{equation*}
\xi_{\bj} = \xi_{\bj}(x,z) = \sum_{y\in W_{\bj}}\sum_{k=1}^N\gh(x,y)\Delta^k\phi(y,z),
\end{equation*}
and further $t = t(d,L) = (1/2)(\log L)^{-(9+1/5)}L^{-d}$.
Assume that we can prove
\begin{equation}
\label{eq:smvgoodenv-expbound}
\left|\sum_{\bj\in J}\pE\left[\xi_{\bj}\right]\right| \leq \frac{t}{2}.
\end{equation}
Then
\begin{equation*}
  \pP\left(\left|\sum_{\bj\in J}\xi_{\bj}\right| \geq t\right)
  \leq
  \pP\left(\left|\sum_{\bj\in
        J}\xi_{\bj}-\pE\left[\xi_{\bj}\right]\right| \geq \frac{t}{2}\right) \leq K \max_{1\leq l\leq K}\pP\left(\left|\sum_{\bj\in
        J_l}\xi_{\bj}-\pE\left[\xi_{\bj}\right]\right| \geq \frac{t}{2K}\right).
\end{equation*}
Due to~\eqref{eq:smvgoodenv-distance},  the random variables
$\xi_{\bj}-\pE\left[\xi_{\bj}\right]$, $\bj\in J_l$,
are independent and centered. Hoeffding's inequality yields, with
${||\xi_{\bj}||}_{\infty} = \sup_{\omega\in\Omega}|\xi_{\bj}(\omega)|$, 
\begin{equation}
\label{eq:smvgoodenv-hoeffding}
 \pP\left(\left|\sum_{\bj\in
      J_l}\xi_{\bj}-\pE\left[\xi_{\bj}\right]\right| \geq
  \frac{t}{2K}\right)
\leq 2\exp\left(-c\frac{L^{-2d}(\log L)^{-(18+2/5)}}{\sum_{\bj \in
      J_l}{||\xi_{\bj}||}^2_{\infty}}\right) 
\end{equation}
for some constant $c>0$. In order to control the $\sup$-norm of the
$\xi_{\bj}$, we use the estimates
\begin{equation*} \gh(x,W_{\bj}) \leq C\Gamma^{(2)}(x,W_{\bj}) \leq
    \frac{CN^ds_L^d}{s_L^2(s_L + \dist(x,W_{\bj}))^{d-2}} =
    CN^d\left(1+\frac{\dist(x,W_{\bj})}{s_L}\right)^{2-d},
\end{equation*}
and, by Lemma~\ref{def:est-deltaphi} for $y\in W_{\bj}$,
$\left|\Delta^k\phi(y,z)\right| \leq C\delta^{k-1}k(\log L)^{-12}L^{-d}$.
Altogether we arrive at
\begin{equation*}
  ||\xi_{\bj}||_{\infty} \leq
  C\left(1+\frac{\dist(x,W_{\bj})}{s_L}\right)^{2-d}N^d(\log L)^{-12}L^{-d},
\end{equation*}
uniformly in $z$. If we put the last display into~\eqref{eq:smvgoodenv-hoeffding}, we get,
using $d\geq 3$ in the last line,
\begin{eqnarray*}
  \pP\left(\left|\sum_{\bj\in
        J_l}\xi_{\bj}-\pE\left[\xi_{\bj}\right]\right| \geq
    \frac{t}{2K}\right)&\leq& 
   2\exp\left(-c\frac{(\log L)^{6-2/5}}{N^4\sum_{r=1}^{C(\log
        L)^3/N}r^{-d+3}}\right)\\
   &\leq& 2\exp\left(-c\frac{(\log L)^{3-2/5}}{N^3}\right).
\end{eqnarray*}
It follows that for $L$ large enough, uniformly in $x$ and $z$,
\begin{equation*}
\pP\left(\left|\sum_{\bj\in J}\xi_{\bj}\right|
  \geq \frac{1}{2}(\log L)^{-(9+1/5)}L^{-d}\right)  \leq \exp\left(-(\log L)^{5/2}\right).
\end{equation*}
It remains to prove~\eqref{eq:smvgoodenv-expbound}. We have
\begin{equation*}
\left|\sum_{\bj\in J}\pE\left[\xi_{\bj}\right]\right| 
\leq  \sum_{y\in S'^c}\gh(x,y)\left|\sum_{y'\in V_L}\pE\left[\sum_{k=1}^N\Delta^k\ph(y,y')\right]\phi(y',z)\right|.
\end{equation*}
Now~\eqref{eq:smvgoodenv-expbound} follows from the estimates $\gh(x,S'^c)
\leq C(\log L)^6$ and
\begin{equation*}
\sup_{y\in S'^c}\left|\sum_{y'\in
    V_L}\pE\left[\sum_{k=1}^N\Delta^k\ph(y,y')\right]\phi(y',z)\right| \leq C(\log
L)^{-18} L^{-d},
\end{equation*}
which in turn follows from Proposition~\ref{def:est-isometry} applied to
$\nu(\cdot) = \pE\left[\sum_{k=1}^N\Delta^k\ph(y,y+\cdot)\right]$.
\end{prooof}


 \begin{remark}
   The reader should notice that for $y\in S'^c$, the signed
   measure $\nu$ fulfills the requirements (i) and (ii) of
   Proposition~\ref{def:est-isometry}. Indeed, after $N=\lceil\log\log
   L\rceil$ steps away from $y$, the coarse grained walks are still in the
   interior part of $V_L$, where the coarse graining radius did not start
   to shrink.  Due to $\czero$, we thus deduce that (i) and (ii)
   hold true for the signed measure
   $\pE[\sum_{k=1}^N(1_{V_L}(\Ph-\ph))^k\ph(y,y+\cdot)]$. Replacing
   $\Ph$ by $\Phg$ does not destroy the symmetries of this
   measure, so that Proposition~\ref{def:est-isometry} can be applied to $\nu$.
 \end{remark}

\subsection{Estimates in the presence of bad points}
\label{smv-onebad-exits}
In the following lemma, we estimate $D_{L,\psi}^\ast$ on environments with
bad points.  We work with the modified kernels $\Pho$, $\Phgo$, $\pho$ from
Section~\ref{super-cgmod}. Recall that the exit measures under these
kernels do not change, e.g. $\Pg_L=
\mbox{ex}_{V_L}(x,\cdot;\Phgo_{L,r})$.  Again, we make the
choice $r=r_L$ for the coarse graining scheme.
\begin{lemma}
\label{def:lemma-badpart}
In the setting of Lemma~\ref{def:lemma-goodpart}, for $i=1,2,3,4$,
  \begin{equation*}
  \begin{split}
    \pP\left(\sup_{x\in V_{L/5}}||(\Pi_L-\pi_L)\ph_{\psi}(x,\cdot)||_1 > (\log L)^{-9+
        9(i-1)/4};\, \onebadi\right)\\ 
        \leq\exp\left(-(\log L)^{7/3}\right).
\end{split}
\end{equation*}
\end{lemma}
\begin{prooof}
  By the triangle inequality, 
  \begin{equation}
  \label{eq:smvbadpart-1}
    ||(\Pi - \pi)\hat{\pi}_{\psi}(x,\cdot)||_1\leq
    ||(\Pi -\Pg)\hat{\pi}_{\psi}(x,\cdot)||_1 +
    ||(\Pg - \pi)\hat{\pi}_{\psi}(x,\cdot)||_1.
 \end{equation}
 The second summand on the right is estimated by
 Lemma~\ref{def:lemma-goodpart}. For the first term we have, with
 $\Delta = 1_{V_L}(\Pho-\Phgo)$,
\begin{equation*}
(\Pi-\Pg)\ph_\psi = \Ghgo1_{\badP_L}\Delta\Pi\ph_\psi.
\end{equation*}
Note that since we are on $\onebad$, the set $\badP_L$ is contained in a small region. First assume that $\badP_L\subset \sh_L(L/(\log
L)^{10})$. Then $\sup_{x\in V_{L/5}}\Ghgo(x,\badP_L) \leq C(\log L)^{-10}$
by Corollary~\ref{def:super-cor}, which bounds the first summand
of~\eqref{eq:smvbadpart-1}. Next assume $\omega$ bad on level
$4$ and $\badP_L\not\subset V_{L/2}$. Then $\sup_{x\in
  V_{L/5}}\Ghgo(x,\badP_L) \leq C(\log L)^{-3}$ by the same corollary,
which is good enough for this case.
  
It remains to consider the cases $\omega$ bad on level at most $3$ with
$\badP_L\not\subset\sh_L(L/(\log L)^{10})$, or $\omega$ bad on level $4$ with
$\badP_L\subset V_{L/2}$. We put $\phi = \pi\ph_{\psi}$ and expand 
  \begin{eqnarray*}
   (\Pi-\Pg)\ph_\psi = \left(\Ghgo1_{\badP_L}\Delta\Pi\right)\ph_\psi &=&
   \sum_{k=1}^\infty\left(\Ghgo1_{\badP_L}\Delta\right)^k\Pg\ph_\psi\\ 
   &=& \sum_{k=1}^\infty\left(\Ghgo1_{\badP_L}\Delta\right)^k\phi +
   \sum_{k=1}^\infty\left(\Ghgo1_{\badP_L}\Delta\right)^k(\Pg-\pi)\ph_{\psi}\\
   &=& F_1 + F_2. 
  \end{eqnarray*}
  By Corollary~\ref{def:super-cor},
  \begin{eqnarray*}
    ||F_1(x,\cdot)||_1 &\leq&
    \sum_{k=0}^\infty||(\Ghgo1_{\badP_L}\Delta)^k(x,\cdot)||_1\sup_{v\in
      V_L}\Ghgo(v,\badP_L)\sup_{w\in
      \badP_L}||\Delta\phi(w,\cdot)||_1\\
    & \leq& C\sup_{w\in
      \badP_L}||\Delta\phi(w,\cdot)||_1. 
  \end{eqnarray*}
  Proceeding as in Lemma~\ref{def:est-deltaphi},
  \begin{equation*}
    ||\Delta\phi(w,\cdot)||_1\leq ||\Delta\ph(w,\cdot)||_1\sup_{w':\,|\Delta\ph(w,w')| >
      0}||\phi(w',\cdot)-\phi(w,\cdot)||_1.
  \end{equation*}
  If $\omega$ is not bad on level $4$, we have on $\badP_L$ the equality 
  $\Delta=\Ph-\ph$.  Since
  $\badP_L\cap\sh_L(2r_L)=\emptyset$, this gives $\sup_{w\in
    \badP_L}||\Delta\ph(w,\cdot)||_1 \leq C(\log L)^{-9+9i/4}$ for every
  $i=1,2,3,4$.  For the second factor in the last display, Lemma~\ref{def:est-phi} (iii) yields
  the bound $C(\log L)^{-3}$. We arrive at $||F_1(x,\cdot)||_1\leq C(\log
  L)^{-12+9i/4}$. For $F_2$, we obtain once more with
  Corollary~\ref{def:super-cor}, 
\begin{equation*}
    ||F_2(x,\cdot)||_1 \leq C\sup_{y\in V_L}
||(\Pg-\pi)\ph_{\psi}(y,\cdot)||_1.
  \end{equation*}
This term is again estimated by Lemma~\ref{def:lemma-goodpart}, and
the lemma is proved.
\end{prooof} 


   \section{Non-smoothed and locally smoothed exits}
\label{nonsmv-exits}
Here, we aim at bounding the total variation distance of the exit measures
without additional smoothing (Lemma~\ref{def:nonsmv-lemma1}), as well as in
the case where a kernel of constant smoothing radius $s$ is added
(Lemma~\ref{def:nonsmv-lemma2}). We use the transition kernels $\Pho,\Phgo$
and $\pho$.

Throughout this section, we work with constant parameter $r$. We always
assume $L$ large enough such that $r< r_L$. The right choice of $r$ depends
on the deviations $\delta$ and $\eta$ we are shooting for and will become
clear from the proofs.  In either case, we choose $r\geq r_0$, where $r_0$
is the constant from Section~\ref{s2}. The value of $r$ will then also
influence the choice of the perturbation $\e_0$ in
Lemma~\ref{def:nonsmv-lemma1} and the smoothing radius $l$ in
Lemma~\ref{def:nonsmv-lemma2}, respectively.

We recall the partition of bad points into the sets $\badP_L$,
$\badP_{L,r}$, $\badP_{L,r}^\partial$, $\badPo_{L,r}$ and the
classification of environments into $\good$, $\onebad$ and $\bbad$ from
Section~\ref{smoothbad}.  

The bounds for $\manybad$
(Lemma~\ref{def:smoothbad-lemmamanybad}) and for $\bbad$
(Lemma~\ref{def:smoothbad-lemmabbad}) ensure that we may restrict ourselves to environments $\omega \in \onebad\cap
{(\bbad)}^c$. For such environments, we introduce two disjoint random sets
$Q_{L,r}^1(\omega)$, $Q_{L,r}^2(\omega) \subset V_L$ as follows:
\begin{itemize}
\item If $\badP_L(\omega)\subset V_{L/2}$, set $Q_{L,r}^1(\omega)=
  \badP_L(\omega)$ and $Q_{L,r}^2(\omega)= \badP_{L,r}^\partial(\omega)$.
\item If $\badP_L(\omega)\not\subset V_{L/2}$, set $Q_{L,r}^1(\omega)=
  \emptyset$ and $Q_{L,r}^2(\omega)= \badPo_{L,r}(\omega)$.
\end{itemize}
Of course, on $\good$, we have $Q_{L,r}^1(\omega) = \emptyset$ and 
$Q_{L,r}^2(\omega) = \badP_{L,r}^\partial(\omega)$.

\begin{lemma}
\label{def:nonsmv-lemma1}
There exists $\delta_0 > 0$ such that if $\delta\in (0,\delta_0]$, there
exist $\e_0 = \e_0(\delta) > 0$ and $L_0 = L_0(\delta)>0$ with the following
property: If $\e \leq \e_0$ and $L_1\geq L_0$, then
$\cone(\e)$, $\ctwo(\delta,L_1)$
imply that for $L_1\leq L\leq L_1(\log L_1)^2$,
\begin{equation*}
  \pP\left(\sup_{x\in V_{L/5}}||\Pi_L-\pi_L)(x,\cdot)||_1 >
    \delta\right) \leq \exp\left(-\frac{9}{5}(\log L)^2\right). 
\end{equation*}
\end{lemma}
\begin{prooof}
  We choose $\delta_0 > 0$ according to Remark~\ref{def:super-remark} and
  take $\delta \in (0, \delta_0]$. The right choice of $\e_0$ and $L_0$
  will be clear from the course of the proof. From
  Lemmata~\ref{def:smoothbad-lemmamanybad}
  and~\ref{def:smoothbad-lemmabbad} we learn that if we take $L_1$ large
  enough and $L$ with $L_1\leq L \leq L_1(\log L_1)^2$, then under
  $\ctwo(\delta,L_1)$
  \begin{equation*}
    \pP\left(\manybad\cup\bbad\right) \leq 
    \exp\left(-\frac{9}{5}(\log L)^2\right). 
  \end{equation*}
  Therefore, the claim follows if we show that on $\onebad\cap
  {(\bbad)}^c$, we have for all sufficiently small $\e$ and all large $L$,
  $x\in V_{L/5}$,
  \begin{equation*}
   ||(\Pi - \pi)(x,\cdot)||_1 \leq \delta.
  \end{equation*}
  Let $\omega\in\onebad\cap {(\bbad)}^c$. We use 
  the partition of $\badPo_{L,r}$ into the sets $Q^1$, $Q^2$ described
  above. With $\Delta = 1_{V_L}(\Pho-\Phgo)$, we have inside $V_L$
  \begin{equation*}
    \Pi = \Ghgo1_{Q^1}\Delta\Pi + \Ghgo1_{Q^2}\Delta\Pi + \Pg.
  \end{equation*}
  By replacing successively $\Pi$ in the first summand on the right-hand
  side, we arrive at
  \begin{equation*}
    \Pi = \sum_{k=0}^\infty{\left(\Ghgo1_{Q^1}\Delta\right)}^k\Pg +
    \sum_{k=0}^\infty{\left(\Ghgo1_{Q^1}\Delta\right)}^k\Ghgo1_{Q^2}\Delta\Pi.  
  \end{equation*}
  Since with $\Delta'= 1_{V_L}(\Phgo-\pho)$, $\Pg = \pi +
  \Ghgo\Delta'\pi$, we obtain
  \begin{eqnarray}
    \label{eq:nonsmv-lemma1-splitting}
    \lefteqn{\Pi -\pi}\nonumber\\ 
    &=& \sum_{k=1}^\infty{\left(\Ghgo1_{Q^1}\Delta\right)}^k\pi +
    \sum_{k=0}^\infty{\left(\Ghgo1_{Q^1}\Delta\right)}^k\Ghgo 1_{Q^2}\Delta\Pi + 
    \sum_{k=0}^\infty{\left(\Ghgo1_{Q^1}\Delta\right)}^k\Ghgo\Delta'\pi\nonumber\\
    &=&  F_1 + F_2 + F_3. 
  \end{eqnarray}
  We will now prove that each of the three parts $F_1,F_2,F_3$ is bounded
  by $\delta /3$. If $Q^1 \neq \emptyset$, then $Q^1 = \badP_L \subset V_{L/2}$ and $Q^2 =
  \badP_{L,r}^\partial$. Using Corollary~\ref{def:super-cor} in
  the second and Lemma~\ref{def:lemmalawler} (ii) in the third inequality,
  \begin{eqnarray}
    \label{eq:nonsmv-lemma1-a1}
    ||F_1(x,\cdot)||_1 &\leq&
    \sum_{k=0}^\infty||(\Ghgo1_{\badP_L}\Delta)^k(x,\cdot)||_1\sup_{y\in
      V_L}\Ghgo(y,\badP_L)\sup_{z\in\badP_L}||\Delta\pi(z,\cdot)||_1\nonumber\\
    &\leq& C \sup_{z\in V_{L/2}}||\Delta\pi(z,\cdot)||_1 \leq C(\log
    L)^{-3}\leq C (\log L_0)^{-3}
    \leq  \delta/3
  \end{eqnarray}
  for $L_0 = L_0(\delta)$ large enough, $L\geq L_0$.
  Regarding $F_2$, we have in the case $Q^1 \neq \emptyset$ 
  by Corollary~\ref{def:super-cor} (ii)
  \begin{equation*}
    ||F_2(x,\cdot)||_1 \leq C\sup_{y\in
      V_{2L/3}}\Ghgo(y,\badP_{L,r}^\partial) \leq C(\log r)^{-1/2}.
  \end{equation*}
  On the other hand, if $Q^1 = \emptyset$, then $\badP_L$ is outside
  $V_{L/3}$, so that by Corollary~\ref{def:super-cor} (i), (ii)
  \begin{equation*}
    ||F_2(x,\cdot)||_1 \leq 2\Ghgo(x,\badP^\partial_{L,r}\cup\badP_L)\leq
    C\left((\log L)^{-3} + (\log r)^{-1/2}\right). 
  \end{equation*}
  Altogether, for all $L \geq L_0$, by choosing $r = r(\delta)$ and $L_0 =
  L_0(\delta,r)$ large enough,
  \begin{equation}
  \label{eq:smv-lemma1-a2-2}
    ||F_2(x,\cdot)||_1 \leq C\left((\log L_0)^{-3} + (\log r)^{-1/2}\right)
    \leq \delta/3.
  \end{equation}
  It remains to handle $F_3$. Once again with Corollary~\ref{def:super-cor}
  (iii) for some $C_3 > 0$,
  \begin{equation*}
    ||F_3(x,\cdot)||_1 \leq C_3\sup_{y\in V_{2L/3}}\left|\left|\Ghgo\Delta'\pi(y,\cdot)\right|\right|_1. 
  \end{equation*}
 We have by definition of $\Delta'$,
\begin{equation}
\label{eq:nonsmv-lemma1-a3}
\Ghgo\Delta'\pi=\Ghgo1_{V_L\backslash
  \badPo_{L,r}}\Delta'\pi + \Ghgo1_{\badP_L}\Delta'\pi,
\end{equation}
and $\Delta'$ vanishes on $\badP_L$ except for the case $\omega$ bad on
level $4$ with $\badP_L\subset V_{L/2}$. In this case, we use
Corollary~\ref{def:super-cor} (i) and Lemma~\ref{def:lemmalawler} (ii) to obtain
\begin{equation}
\label{eq:nonsmv-lemma1-a3-0}
||\Ghgo1_{\badP_L}\Delta'\pi(y,\cdot)||_1\leq C(\log L)^{-3}\leq C_3^{-1}\delta/12
\end{equation}
for $L_0$ large enough, $L \geq L_0$. Concerning the first term
of~\eqref{eq:nonsmv-lemma1-a3}, we note that on $V_L\backslash
\badPo_{L,r}$, $\Delta'\pi= (\Ph-\ph)\ph\pi$. Therefore, if $z\in
V_L\backslash\left(\badPo_{L,r}\cup\sh_L(2r_L)\right)$, we obtain
$||\Delta'\pi(z,\cdot)||_1 \leq C(\log L)^{-9}$. Since $\Ghgo(y,V_L) \leq
C(\log L)^6$, it follows that
  \begin{equation}
    \label{eq:nonsmv-lemma1-a3-1}
    \sup_{y\in V_{2L/3}}\left|\left|\Ghgo
    1_{V_L\backslash\left(\badPo_{L,r}\cup\sh_L(2r_L)\right)}\Delta'\pi(y,\cdot)\right|\right|_1
    \leq C(\log L)^{-3} \leq C_3^{-1}\delta/12
  \end{equation}
  for $L$ large. Recall the definition of the layers
  $\Lambda_j$ from Section~\ref{s2}. For $z\in
  \Lambda_j\backslash\badPo_{L,r}$, $1\leq j\leq J_1$, we have
  $||\Delta'\pi(z,\cdot)||_1 \leq C(\log r + j)^{-9}$. By
  Lemma~\ref{def:super-gammalemma} (iii), $\Ghgo(y,\Lambda_j) \leq C$ for
  some constant $C$, independent of $r$ and $j$. Therefore,
  \begin{equation}
  \label{eq:nonsmv-lemma1-a3-2}
  \sup_{y\in V_{2L/3}}\left|\left|\Ghgo
  1_{\bigcup_{j=1}^{J_1}\Lambda_j\backslash\badPo_{L,r}}\Delta'\pi(y,\cdot)\right|\right|_1
  \leq C (\log r)^{-8} \leq C_3^{-1}\delta/12,
  \end{equation}
  if $r$ is chosen large enough. Finally, for the first layer $\Lambda_0$, there
  is a constant $C_0$ satisfying
  \begin{equation*}
    \sup_{y\in V_{2L/3}}\left|\left|\Ghgo
    1_{\Lambda_0\backslash\badPo_{L,r}}\Delta'\pi(y,\cdot)\right|\right|_1
    \leq C_0\sup_{z\in\Lambda_0}||\Delta'(z,\cdot)||_1. 
  \end{equation*}
  Now we take $\e_0 = \e_0(\delta,r)$ small enough such that for $\e\leq
  \e_0$, $\sup_{z\in\Lambda_0}||\Delta'(z,\cdot)||_1\leq
  C_0^{-1}C_3^{-1}\delta/12$. We have shown that $||F_3(x,\cdot)||_1 \leq \delta/3$,
  and the lemma is proven.
\end{prooof}
\begin{remark}
  As the proof shows, we do not have to assume $\ctwo(\delta,L_1)$ for the
  desired deviation $\delta$. We could instead assume $\ctwo(\delta',L_1)$ for some
  $0<\delta'\leq\delta_0$. However, $L_1$ has to be larger than $L_0$, which
  depends on $\delta$.  This observation will be useful in the next lemma.
\end{remark}
\begin{lemma}
  \label{def:nonsmv-lemma2}
  There exists $\delta_0 > 0$ with the following property: For each $\eta >
  0$, there exist a smoothing radius $l_0=l_0(\eta)$ and $L_0 =
  L_0(\eta)$ such that if $L_1\geq L_0$, $l\geq l_0$ and
  $\ctwo(\delta,L_1)$ holds for some $\delta\in(0,\delta_0]$, then for
  $L_1\leq L\leq L_1{(\log L_1)}^2$ and $\psi\equiv l$,
\begin{equation*}
  \pP\left(\sup_{x\in V_{L/5}}||(\Pi_L-\pi_L)\ph_{\psi}(x,\cdot)||_1 >
    \eta\right) \leq \exp\left(-\frac{9}{5}(\log L)^2\right). 
\end{equation*}

\end{lemma}
\begin{prooof}
  The proof is based on a modification of the computations in
  the foregoing lemma. Let $\delta_0$ be as in
  Lemma~\ref{def:nonsmv-lemma1}. We fix an arbitrary $0 <
  \delta\leq\delta_0$ and assume $\ctwo(\delta,L_1)$ for some $L_1 \geq
  L_0$, where $L_0 = L_0(\eta)$ will be chosen later. In the following,
  ``good'' and ``bad'' is always to be understood with respect to
  $\delta$. Again, for $L_1\leq L\leq L_1(\log L_1)^2$,
 \begin{equation*}
   \pP\left(\manybad\cup\bbad\right) \leq 
   \exp\left(-\frac{9}{5}(\log L)^2\right). 
  \end{equation*}
  For $\omega\in\onebad\cap {(\bbad)}^c$, we use the splitting~\eqref{eq:nonsmv-lemma1-splitting} of $\Pi-\pi$
  into the parts $F_1, F_2, F_3$. For the summands $F_1$ and $F_2$,
  we do not need the additional smoothing by $\ph_{\psi}$, since
  by~\eqref{eq:nonsmv-lemma1-a1} 
  \begin{equation*}
  ||F_1(x,\cdot)||_1 \leq C(\log L)^{-3}\leq \eta/3,
 \end{equation*}
  and by~\eqref{eq:smv-lemma1-a2-2}
  \begin{equation*}
  ||F_2(x,\cdot)||_1 \leq C\left((\log L)^{-3} + (\log r)^{-1/2}\right) \leq
  \eta/3,
 \end{equation*} 
 if $L\geq L_0$ and $r$, $L_0$ are chosen large enough, depending on
 $d$ and $\eta$.
 We turn to
 $F_3$. With~\eqref{eq:nonsmv-lemma1-a3-0},~\eqref{eq:nonsmv-lemma1-a3-1}
 and~\eqref{eq:nonsmv-lemma1-a3-2} we have (recall that $\Delta'=
 1_{V_L}(\Phgo-\pho)$)
 \begin{eqnarray}
    \label{eq:nonsmv-lemma2-a3}  
    ||F_3\ph_s(x,\cdot)||_1&\leq&   C\left(\sup_{y\in V_{2L/3}}\left|\left|\Ghgo 1_{V_L\backslash\Lambda_0}\Delta'\pi(y,\cdot)\right|\right|_1 + 
      \sup_{z\in \Lambda_0}||\Delta'\pi\ph_{\psi}(z,\cdot)||_1\right)\nonumber\\
    &\leq& C\left((\log L)^{-3} +(\log r)^{-8} + \sup_{z\in
        \Lambda_0}||\Delta'\pi\ph_{\psi}(z,\cdot)||_1\right)\nonumber\\ 
    &\leq& \eta/6 + C_1\sup_{z\in
      \Lambda_0}||\Delta'\pi\ph_{\psi}(z,\cdot)||_1,
  \end{eqnarray}
  if $L\geq L_0$ and $r$, $L_0$ are sufficiently large. Regarding the
  second summand of~\eqref{eq:nonsmv-lemma2-a3}, set $m = 3r$ and define for
  $K\in\mathbb{N}$
  \begin{equation*}
    \vartheta_K(z) = \min\left\{n\in\mathbb{N} : |X_n^z -z| > Km\right\}\in[0,\infty],
  \end{equation*}
  where 
  $X_n^z$ denotes simple random walk
  with start in $z$. By the invariance principle for simple random walk,
  we can choose $K$ so large such that
  \begin{equation*}
    \max_{z\in V_L:\, \dL(z) \leq m}\Prw_z\left(\vartheta_K(z) \leq \tau_L \right) \leq \frac{\eta}{24C_1}
  \end{equation*}
  uniformly in $L\geq L_0$, where $C_1$ is the constant
  from~\eqref{eq:nonsmv-lemma2-a3}. If
  $z\in\Lambda_0$, $z'\in V_L\cup\partial V_L$ with
  $\Delta'(z,z') \neq 0$, we have $\dL(z') \leq m$ and $|z-z'| \leq
  m$. Thus, using Lemma~\ref{def:app-kernelest} (iii) of the appendix with $\psi\equiv l$,
  \begin{eqnarray*}
    \lefteqn{C_1\sup_{z\in \Lambda_0}||\Delta'\pi\ph_{\psi}(z,\cdot)||_1}\\ 
    &\leq& 
    C_1\sup_{z\in \Lambda_0}\left|\left|\sum_{z'\in V_L\cup\partial
      V_L:\atop \Delta'(z,z')\neq 0}\Delta'(z,z')\right.\right.\\
   &&\,\left.\left.
     \left(\sum_{w\in\partial V_L:\atop |z'-w| > Km}\pi(z',w) + \sum_{w\in\partial
        V_L:\atop |z'-w| \leq Km}\pi(z',w)\right)
    \left(\ph_{\psi}(w,\cdot)-\ph_{\psi}(z,\cdot)\right)\right|\right|_1
    \\
&\leq& \frac{\eta}{6}+C(K+1)m\frac{\log l}{l} \leq \eta/3,
  \end{eqnarray*}
  if we choose $l = l(\eta,r)$ large enough. This proves the lemma.
\end{prooof}




   \section{Proofs of the main results on exit laws}
\label{proofmain}
\begin{prooof2}{\bf of Proposition~\ref{def:main-prop}:}
  We take $\delta_0$ small enough and, for $\delta \leq \delta_0$, we
  choose $L_0 = L_0(\delta)$ large enough according to
  Remark~\ref{def:super-remark} and the statements of
  Sections~\ref{smv-exits},~\ref{nonsmv-exits}.\\
  (ii) is a consequence of Lemma~\ref{def:nonsmv-lemma2}, so we have to
  prove (i). Let $L_1 \geq L_0$, and assume that $\ctwo(\delta,L_1)$
  holds. Then, for $i = 1,2,3$ and $L_1\leq L\leq L_1(\log L_1)^2$,
  $\psi\in\mathcal{M}_L$, using Lemma~\ref{def:smoothbad-lemmamanybad},
  \begin{eqnarray*}
    b_i(L,\psi,\delta) &\leq& 
    \pP\left(D_{L,\psi}^\ast > (\log L)^{-9 +
        9(i-1)/4}\right)\\ 
    &\leq& \pP\left(\manybad\right) + \pP\left(D_{L,\psi}^\ast > (\log L)^{-9 +
        9(i-1)/4};\, \onebad\right)\\
    &\leq& \exp\left(-\frac{19}{10}(\log L)^2\right) + \pP\left(D_{L,\psi}^\ast >
      (\log L)^{-9 + 9(i-1)/4};\, \onebad\right).  
  \end{eqnarray*}
  For the last summand, we have by
  Lemmata~\ref{def:lemma-goodpart},~\ref{def:lemma-badpart}, under $\ctwo(\delta,L_1)$,
  \begin{eqnarray*}
    \lefteqn{\pP\left(D_{L,\psi}^\ast > (\log L)^{-9 + 9(i-1)/4};\, \onebad\right)}\\ 
    &\leq&  
    \pP\left(D_{L,\psi}^\ast > (\log L)^{-9};\, \good\right) + 
    \sum_{j=1}^4
    \pP\left(D_{L,\psi}^\ast > (\log L)^{-9 +
        9(i-1)/4};\, \onebadj\right)\\
    &\leq& 
    \exp\left(-(\log L)^{7/3}\right) + 
    \sum_{j=1}^i\pP\left(D_{L,\psi}^\ast > (\log L)^{-9 +
        9(i-1)/4};\, \onebadj\right)\\
    &&\, +\sum_{j=i+1}^4\pP\left(\onebadj\right)\\
      &\leq& 4\ \exp\left(-(\log L)^{7/3}\right) +
    CL^ds_L^d\exp\left(-\left((3+i+1)/4\right)\left(\log(r_L/20)\right)^2\right).
  \end{eqnarray*}
  Therefore, by enlarging $L$ if necessary,
  \begin{equation*}
    \pP\left(D_{L,\psi}^\ast > (\log L)^{-9 +
          9(i-1)/4};\, \onebad\right)\leq
    \frac{1}{8}\exp\left(-\left((3+i)/4\right)\left(\log L\right)^2\right),
  \end{equation*}
  and
  \begin{equation*}
    b_i(L,\psi,\delta) \leq
    \frac{1}{4}\exp\left(-\left((3+i)/4\right)\left(\log L\right)^2\right). 
  \end{equation*}
  For the case $i=4$, notice that
  \begin{equation*}
    b_4(L,\psi,\delta) \leq \pP\left(D_{L,\psi}^\ast > (\log L)^{-9/4}\right) +
    \pP\left(D_L^\ast > \delta\right). 
  \end{equation*}
  The first summand can be estimated as the corresponding terms in the case
  $i=1,2,3$, while for the last term we use Lemma~\ref{def:nonsmv-lemma1}.  
\end{prooof2}
As Theorem~\ref{def:main-theorem} now follows immediately, we turn to the
proof of the local estimates. Here, the results from
Section~\ref{super} play again a key role.

\begin{prooof2}{\bf of Theorem~\ref{local-thm-exitmeas}:}
  As usual, we mostly drop $L$ as index, so always $\pi =\pi_L$, $\Ph=
  \Ph_L$ and so on.  For the whole proof, we let $r=r_L$. Choose $\delta_0$ and $L_0$ as in
  Proposition~\ref{def:main-prop}. Recall the definition of $\good$ from
  Section~\ref{smoothbad}. By Proposition~\ref{def:main-prop}, we find
  $\delta$, $\e > 0$ and $L_0>0$ such that under $\cone(\e)$ and
  $\czero$, condition $\ctwo(\delta,L)$ holds true for all $L\geq L_0$. We
  put $A_L = \good$ and note that similar to 
  Lemma~\ref{def:smoothbad-lemmamanybad}, if $L\geq L_0$,
  \begin{equation*}
  \pP(A_L^c) \leq \exp\left(-(1/2)(\log L)^2\right).
  \end{equation*}
  For the rest of the proof, take $\omega \in A_L$. On such environments,
  $\Gh$ equals $\Ghg$ by our choice $r=r_L$.
  Now let us prove part (i). Observe that $W_t$ can be covered by $K|W_t|r^{-(d-1)}$ many neighborhoods
$V_{3r}(y)$, $y\in \sh_L(r)$, as defined in
Section~\ref{super-Ghestimates}, where $K$ depends on the dimension only. In
particular, $\Gamma(x,W_t) \leq C(t/L)^{d-1}$. Applying
Lemma~\ref{def:superlemma} (ii), we deduce that
\begin{equation*}
\Pi_L(x,W_t)=\Ghg(x,W_t)\leq C(t/L)^{d-1}.
\end{equation*}    
From 
Lemma~\ref{def:lemmalawler} (i) we know that if $x\in V_{\eta L}$, then for some
constant $c=c(d,\eta)$,
\begin{equation*}
\pi(x,z) \geq cL^{-(d-1)}.
\end{equation*}
Together with the preceding equation, this shows (i).\\
\begin{figure}
\begin{center}\parbox{5.5cm}{\includegraphics[width=5cm]{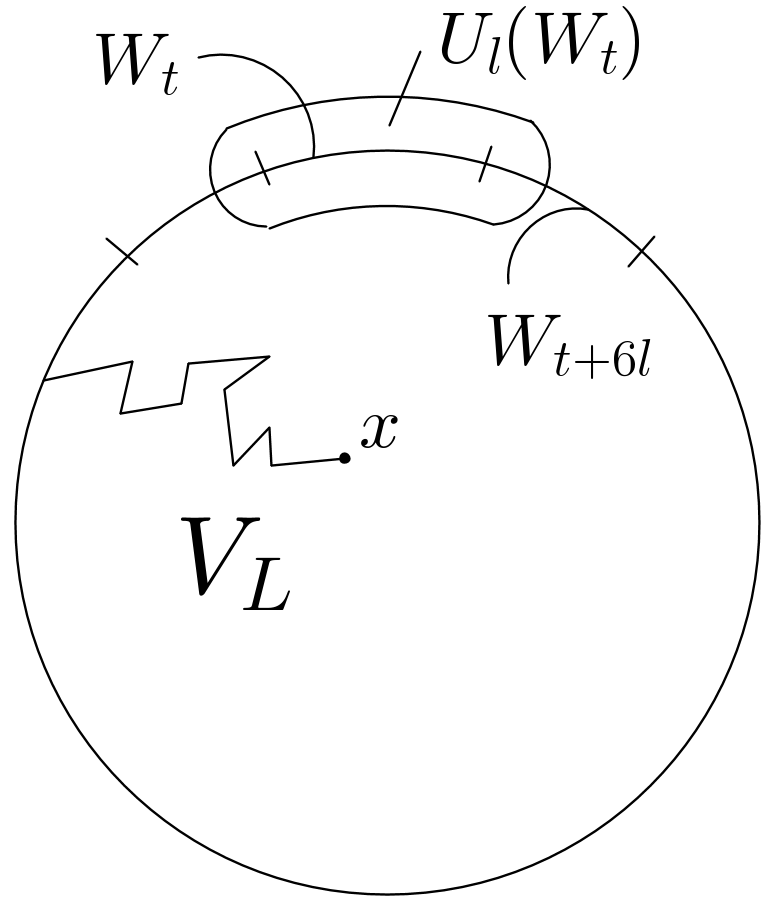}}
\parbox{9.5cm}{
  \caption{On the proof of Theorem~\ref{local-thm-exitmeas} (ii). There,
    $t\geq L/(\log L)^6 > l=L/(\log L)^{17/2}$. If the walk exits $V_L$ through
    $\partial V_L\backslash W_{t+6l}$, it cannot enter $U_l(W_t)$ in one
    step with $\ph_l$.}}
\end{center}
\end{figure}
(ii) Set $l=(\log L)^{13/2} r$ and
consider the smoothing kernel $\ph_\psi$ with $\psi\equiv
l\in\mathcal{M}_l$. Let 
\begin{equation*}
U_l(W_t) = \{y\in\mathbb{Z}^d : \dist(y,W_t) \leq 2l\}.
\end{equation*}
 We claim that
\begin{eqnarray}
\label{eq:local-eq1}
\Pi(x,W_t) -\pi(x,W_{t+6l}) &\leq& \left(\Pi-\pi\right)\ph_\psi(x,U_l(W_t)),\\
\label{eq:local-eq2}
\pi(x,W_{t-6l}) -\Pi(x,W_t) &\leq& \left(\pi-\Pi\right)\ph_\psi(x,U_l(W_{t-6l})).
\end{eqnarray}
Concerning the first inequality,
\begin{equation*}
\Pi\ph_\psi(x,U_l(W_t)) \geq \sum_{y\in W_t}\Pi(x,y)\ph_\psi(y,U_l(W_t)) = \Pi(x,W_t),
\end{equation*}
since $\ph_\psi(y,U_l(W_t))=1$ for $y \in W_t$. Also,
\begin{equation*}
\pi\ph_\psi(x,U_l(W_t)) = \sum_{y\in W_{t+6l}}\pi(x,y)\ph_\psi(y,U_l(W_t)) \leq \pi(x,W_{t+6l}),
\end{equation*}
since $\ph_\psi(y,U_l(W_t))=0$ for $y\in\partial V_L\backslash
W_{t+6l}$. This proves~\eqref{eq:local-eq1}, while~\eqref{eq:local-eq2} is
entirely similar. In the remainder of this proof, we often write $|F|(x,y)$ for $|F(x,y)|$. If we show 
\begin{equation}
\label{eq:local-eq3}
\left|\left(\pi-\Pi\right)\ph_\psi\right|(x,U_l(W_t)) \leq O\left((\log L)^{-5/2}\right)\pi(x,W_t),
\end{equation}
then by~\eqref{eq:local-eq1},
\begin{eqnarray*}
\Pi(x,W_t) &\leq& \pi(x,W_{t+6l}) + O((\log
L)^{-5/2})\pi(x,W_t)\\
&=& \pi(x,W_t) + \pi(x,W_{t+6l}\backslash W_t) +O((\log
L)^{-5/2})\pi(x,W_t)\\
&=& \pi(x,W_t)\left(1+O\left(\max\{l/t,(\log L)^{-5/2}\}\right)\right)\\
&=& \pi(x,W_t)\left(1+O((\log L)^{-5/2})\right).
\end{eqnarray*}
On the other hand, by~\eqref{eq:local-eq2} and still assuming~\eqref{eq:local-eq3},
\begin{equation*}
\Pi(x,W_t) \geq \pi(x,W_{t-6l})-O((\log
L)^{-5/2})\pi(x,W_t)=\pi(x,W_{t})\left(1-O((\log
L)^{-5/2})\right),
\end{equation*}
so that indeed
\begin{equation*}
\Pi(x,W_t) = \pi(x,W_t)\left(1+O\left((\log L)^{-5/2}\right)\right),
\end{equation*}
provided we prove~\eqref{eq:local-eq3}. In that direction, set $B =
V_L\backslash\sh_L(5r)$ and write, with $\Delta=1_{V_L}(\Phg-\ph)$, 
\begin{equation}
\label{eq:local-eq4}
(\pi-\Pi)\ph_\psi = \Ghg\Delta\pi\ph_\psi = \Ghg 1_B\Delta\pi\ph_\psi +
\Ghg 1_{\sh_L(5r)}\Delta\pi\ph_\psi.
\end{equation}
Looking at the first summand we have
\begin{equation*}
|\Ghg1_B\Delta\pi\ph_\psi|(x,U_l(W_t)) \leq (\Ghg1_B|\Delta\ph|\pi)(x,W_{t+6l}).
\end{equation*}
Following the proof of Proposition~\ref{def:super-keyest} (ii), we deduce
\begin{equation*}
\Ghg1_B|\Delta\ph|\Gamma(x,z) \leq C(\log L)^{-5/2}\Gamma(x,z).
\end{equation*}
Together with $\pi \preceq C\Gamma$ and $\pi(x,z)\geq c(d,\eta)L^{-(d-1)}$ this yields the bound  
\begin{equation*}
\Ghg1_B|\Delta\ph|\pi(x,W_{t+6l}) \leq C(\log L)^{-5/2}\Gamma(x,W_t)\leq C(\log L)^{-5/2}\pi(x,W_t).
\end{equation*}
To obtain~\eqref{eq:local-eq3}, it remains to handle the second summand
of~\eqref{eq:local-eq4}, i.e. we have to bound
\begin{equation*}
|\Ghg 1_{\sh_L(5r)}\Delta\pi\ph_\psi|(x,U_l(W_t)).
\end{equation*}
We abbreviate $S= \sh_L(5r)$ and split into
\begin{eqnarray*}
\lefteqn{\Ghg 1_S\Delta\pi\ph_\psi(x,U_l(W_t))}\\ &=& \sum_{y\in
  S}\Ghg(x,y)\sum_{z\in\partial
    V_L}\Delta\pi(y,z)(\ph_\psi(z,U_l(W_t))-\ph_\psi(y,U_l(W_t)))\\
&=&\sum_{y\in S}\Ghg(x,y)\sum_{z\in
  W_{t+6l}}\Delta\pi(y,z)(\ph_\psi(z,U_l(W_t))-\ph_\psi(y,U_l(W_t)))\\
&&\ -\sum_{y\in S}\Ghg(x,y)\sum_{z\in \partial V_L\backslash
  W_{t+6l}}\Delta\pi(y,z)\ph_\psi(y,U_l(W_t)).
\end{eqnarray*}
First note that since $\ph_\psi(y,z')=0$ if $|y-z'|> 2l$,
\begin{eqnarray*}
\lefteqn{\left|\sum_{y\in S}\Ghg(x,y)\sum_{z\in \partial V_L\backslash W_{t+6l}}
\Delta\pi(y,z)\ph_\psi(y,U_l(W_t))\right|}\\ 
&\leq& (\Ghg1_{U_{2l}(W_t)\cap S}|\Delta\pi|)(x,\partial V_L\backslash W_{t+6l}). 
\end{eqnarray*}
For $y\in U_{2l}(W_t)\cap S$, we apply Lemma~\ref{def:hittingprob} (iii) together
with Lemma~\ref{def:hittingprob-technical} and obtain
\begin{equation*}
|\Delta\pi|(y,\partial V_L\backslash W_{t+6l}) \leq \sup_{y':\,
  \dist(y',U_{2l}(W_t)\cap S)\leq r}\pi(y',\partial V_L\backslash W_{t+6l})\leq C\frac{r}{l}\leq
C(\log L)^{-13/2}.
\end{equation*}
Since $\Ghg\preceq \Gamma$ and $\pi(x,z)\geq cL^{-d-1}$,
$\Ghg(x,U_{2l}(W_t)\cap S)\leq C\pi(x,W_t)$, and thus
\begin{equation*}
(\Ghg1_{U_{2l}(W_t)\cap S}|\Delta\pi|)(x,\partial V_L\backslash W_{t+6l})
\leq C(\log L)^{-13/2}\pi(x,W_t).
\end{equation*}
It remains to bound
\begin{equation*}
\left|\sum_{y\in S}\Ghg(x,y)\sum_{z\in
  W_{t+6l}}\Delta\pi(y,z)(\ph_\psi(z,U_l(W_t))-\ph_\psi(y,U_l(W_t)))\right|.
\end{equation*}
Set $D^1(y) = \{z\in W_{t+6l}: |z-y| \leq l(\log l)^{-4}\}$. If
$D^1(y)\neq\emptyset$, then $\dist(y,W_t)\leq 7l$. Using Lemma~\ref{def:app-kernelest}
(iii) for the difference of the smoothing steps and the usual estimate for $\Ghg$,
\begin{eqnarray*}
\lefteqn{\sum_{y\in S}\Ghg(x,y)\sum_{z\in
  D^1(y)}|\Delta\pi(y,z)||\ph_\psi(z,U_l(W_t))-\ph_\psi(y,U_l(W_t))|
}\\&\leq& C\frac{t^{d-1}}{L^{d-1}}(\log l)^{-3}\leq C(\log L)^{-5/2}\pi(x,W_t).
\end{eqnarray*}
The region $W_{t+6l}\backslash D^1(y)$ we split into $B_0(y) = \{z\in
W_{t+6l}: |z-y| \in (l(\log l)^{-4}, t]\}$, and 
\begin{equation*}
B_i(y) = \{z\in W_{t+6l}: |z-y| \in (it, (i+1)t]\},\quad i=1,2,\ldots,
  \lfloor 2L/t\rfloor.
\end{equation*}
Furthermore, let
\begin{equation*}
S_i = \{y\in S : B_i(y) \neq \emptyset\},\quad i=0,1,\ldots, \lfloor 2L/t\rfloor.
\end{equation*}
Then
\begin{equation*}
\sum_{y\in S}\Ghg(x,y)\sum_{z\in W_{t+6l}\backslash
  D^1(y)}|\Delta\pi(y,z)|\leq
 C\sum_{i=0}^{\lfloor 2L/t\rfloor}\Ghg(x,S_i)\sup_{y\in S_i}|\Delta\pi|(y,B_i(y)).
\end{equation*}
If $i\geq 1$ and $y\in S_i$, then by Lemma~\ref{def:hittingprob} (iii) 
\begin{equation*}
|\Delta\pi|(y,B_i(y)) \leq \sup_{y':|y'-y|\leq r}\pi(y',B_i(y)) \leq
C\frac{rt^{d-1}}{(it)^d}\leq C\frac{r}{i^dt},
\end{equation*}
while in the case $i=0$, using the same lemma and additionally Lemma~\ref{def:hittingprob-technical},
\begin{eqnarray*}
  \sup_{y':|y'-y|\leq r}\pi(y',B_i(y))&\leq& C\,r\sum_{z\in\partial V_L}\frac{1}{((1/2)l\,(\log
    l)^{-4} + |y-z|)^d}\\
  & \leq& C\frac{r(\log l)^4}{l}\leq C (\log L)^{-5/2}.
\end{eqnarray*}
For the Green's function, we use the estimates
\begin{equation*}
\Ghg(x,S_0) \leq C\frac{t^{d-1}}{L^{d-1}},\quad \Ghg(x,\cup_{i \geq (1/10)L/t}S_i) \leq C,
\end{equation*}
while for $i=1,2,\ldots,\lfloor (1/10)L/t\rfloor$, it holds that $|S_i| \leq Cr(it)^{d-2}t$,
whence
\begin{equation*}
\Ghg(x,S_i) \leq C \frac{i^{d-2}t^{d-1}}{L^{d-1}}. 
\end{equation*}
Altogether, we obtain
\begin{eqnarray*}
  \lefteqn{\sum_{i=0}^{\lfloor 2L/t\rfloor}\Ghg(x,S_i)\sup_{y\in
      S_i}|\Delta\pi|(y,B_i(y))}\\ 
  &\leq& C\left((\log L)^{-5/2}\frac{t^{d-1}}{L^{d-1}} +
    \left(\frac{r}{t}\frac{t^{d-1}}{L^{d-1}}\sum_{i=1}^{\lfloor(1/10)L/t\rfloor}\frac{1}{i^2}\right)
    +\frac{t^{d-1}}{L^{d-1}}\frac{r}{L}\right)\\
  &\leq& C(\log L)^{-5/2}\frac{t^{d-1}}{L^{d-1}}. 
\end{eqnarray*}
This finishes the proof of part (ii). 
\end{prooof2}

Let us finally show how to obtain transience of the RWRE.\\
\begin{prooof2}{\bf of Corollary~\ref{def:main-transience}:}
  Fix numbers $\rho \geq 3$, $\alpha \in (0,(4\rho)^{-1})$ to be specified
  below. With these parameters and $n\geq 1$, we set
\begin{equation*}
q_{n,\alpha,\rho}= \ph_\psi,
\end{equation*}
where $\psi={\left(m_x\right)}_{x\in\mathbb{Z}^d}$ is chosen constant in
$x$, namely $m_x= \alpha\rho^n$. Define
\begin{equation*}
A_n= \bigcap_{|x|\leq
  \rho^{n^{3/2}}}\bigcap_{t\in [\alpha\rho^n,2\alpha\rho^n]}\left\{D_{t,\psi}(x)
    \leq (\log t)^{-9}\right\}.
\end{equation*}
By Proposition~\ref{def:main-prop} (i), there exists $\e_0 > 0$ such that
given $\e\in
(0,\e_0]$, $\cone(\e)$ implies that for $n$ large enough, we have 
\begin{equation*}
  \pP\left(A_n^c\right) \leq C\alpha^d\rho^{(d+1)n^{3/2}}\exp\left(-\left(\log\left(\alpha\rho^n\right)\right)^2\right). 
\end{equation*}
Therefore, for any choice of $\alpha, \rho$ it holds that
\begin{equation*}
\sum_{n=1}^\infty\pP\left(A_n^c\right) < \infty,
\end{equation*}
whence by Borel-Cantelli 
\begin{equation}
\label{eq:main-transience-bc}
\pP\left(\liminf_{n\rightarrow \infty} A_n\right)= 1. 
\end{equation}
We denote the coarse grained RWRE transition kernel by
\begin{equation*}
  Q_{n,\alpha,\rho}(x,\cdot) =
  \frac{1}{\alpha\rho^n}\int_{\mathbb{R}_+}\varphi\left(\frac{t}{\alpha\rho^n}\right)\Pi_{V_t(x)}(x,\cdot)\mbox{d}t. 
\end{equation*}
If $n$ is large enough and $|x| \leq \rho^{n^{3/2}}$, we have on $A_n$
\begin{equation*}
  \left|\left|\left(Q_{n,\alpha,\rho}-q_{n,\alpha,\rho}\right)q_{n,\alpha,\rho}(x,\cdot)\right|\right|_1
  \leq \left(\log(\alpha\rho^n)\right)^{-9} \leq C(\alpha,\rho)n^{-9}.
\end{equation*}
Now assume $|x| \leq \rho^n+1$. For $N$ fixed, $n$ large and $\omega\in A_n$, it follows that for $1\leq M \leq N$
\begin{equation}
\label{eq:main-transience-1}
\left|\left|\left(\left(Q_{n,\alpha,\rho}\right)^M-\left(q_{n,\alpha,\rho}\right)^M\right)q_{n,\alpha,\rho}(x,\cdot)\right|\right|_1 
\leq C(\alpha,\rho)Mn^{-9}.
\end{equation}
For fixed $\omega$, let $\left(\xi_k\right)_{k\geq 0}$ be the Markov chain
running with transition kernel $Q_{n,\alpha,\rho}$. Clearly,
$\left(\xi_k\right)_{k\geq 0}$ can be obtained by observing the basic RWRE
$\left(X_k\right)_{k\geq 0}$ at randomized stopping times. Then
\begin{eqnarray*}
  \lefteqn{\Prw_{x,\omega}\left(\xi_{N-1} \in
      V_{\rho^{n+1}+2\alpha\rho^n}\right)}\\ 
  &\leq& \left(Q_{n,\alpha,\rho}\right)^{N-1}q_{n,\alpha,\rho}\left(x, V_{\rho^{n+1}+4\alpha\rho^n}\right)\\
  &\leq &
  \left|\left|\left(\left(Q_{n,\alpha,\rho}\right)^{N-1}-\left(q_{n,\alpha,\rho}\right)^{N-1}\right)q_{n,\alpha,\rho}(x,\cdot)\right|\right|_1   + \left(q_{n,\alpha,\rho}\right)^N(x,V_{2\rho^{n+1}}).
\end{eqnarray*}
Using Proposition~\ref{def:super-localclt}, we can find 
$N=N(\alpha,\rho)\in\mathbb{N}$, depending not on $n$, such that for any
$x$ with $|x| \leq \rho^n+1$, it holds that
${(q_{n,\alpha,\rho})}^N(x,V_{2\rho^{n+1}}) \leq
1/10$. With~\eqref{eq:main-transience-1}, we conclude that for such $x$,
$n\geq n_0(\alpha,\rho,N)$ large enough and $\omega\in A_n$,
\begin{equation}
\label{eq:main-transience-2}
\Prw_{x,\omega}\left(\xi_{N-1} \in V_{\rho^{n+1}+2\alpha\rho^n}\right) \leq
C(\alpha,\rho)Nn^{-9} + 1/10 \leq 1/5.
\end{equation}
On the other hand, if $x$ is outside $V_{\rho^{n-1}+2\alpha\rho^n}$,
\begin{eqnarray*}
  \lefteqn{\Prw_{x,\omega}\left(\xi_{M} \in V_{\rho^{n-1}+2\alpha\rho^n}\mbox{ for
        some }0\leq M\leq N-1\right)}\\
  &\leq& \sum_{M=1}^{N-1} \left(Q_{n,\alpha,\rho}\right)^Mq_{n,\alpha,\rho}\left(x, V_{\rho^{n-1}+4\alpha\rho^n}\right)\\
  &\leq&
  \sum_{M=1}^{N-1}\left|\left|\left(\left(Q_{n,\alpha,\rho}\right)^M-\left(q_{n,\alpha,\rho}\right)^M\right)q_{n,\alpha,\rho}(x,\cdot)\right|\right|_1
  + \sum_{k=2}^{N}\left(q_{n,\alpha,\rho}\right)^k(x,V_{2\rho^{n-1}}).
\end{eqnarray*}
If $\rho^n-1 \leq |x|$, then
$\left(q_{n,\alpha,\rho}\right)^k(x,V_{2\rho^{n-1}}) = 0$ as long as $k
\leq (1-3/\rho)/(2\alpha)$. By first choosing $\rho$ large enough, then
$\alpha$ small enough and estimating the higher summands again
with Proposition~\ref{def:super-localclt}, we deduce that for such $x$ and
all large $n$,
\begin{equation*}
\sum_{k=1}^{\infty}\left(q_{n,\alpha,\rho}\right)^k(x,V_{2\rho^{n-1}}) \leq 1/10.
\end{equation*}
Together with~\eqref{eq:main-transience-1}, we have for large $n$, $\omega
\in A_n$ and $\rho^n-1 \leq |x|\leq \rho^n+1$,
\begin{equation}
\label{eq:main-transience-3}
  \Prw_{x,\omega}\left(\xi_{M} \in V_{\rho^{n-1}+2\alpha\rho^n}\mbox{ for
        some } 0\leq M\leq N-1\right)\leq C(\alpha,\rho)N^2n^{-9} + 1/10 \leq 1/5.
\end{equation}
Let $B$ be the event that the walk $\left(\xi_k\right)_{k\geq
  0}$ leaves $V_{\rho^{n+1}+2\alpha\rho^n}$ before reaching
$V_{\rho^{n-1}+2\alpha\rho^n}$.  From~\eqref{eq:main-transience-2}
and~\eqref{eq:main-transience-3} we deduce that
$\Prw_{x,\omega}\left(B\right) \geq 3/5$, provided $n$ is large enough,
$\omega \in A_n$ and $\rho^n-1\leq |x|\leq \rho^n +1$. But on $B$, the
underlying basic RWRE $\left(X_k\right)_{k\geq 0}$ clearly leaves $V_{\rho^{n+1}}$
before reaching $V_{\rho^{n-1}}$. Hence if $\omega \in \{\liminf A_n\}$,
there exists $m_0=m_0(\omega)\in\mathbb{N}$ such that
\begin{equation*}
\Prw_{x,\omega}\left(\tau_{V_{\rho^{n+1}}}< T_{V_{\rho^{n-1}}}\right) \geq 3/5
\end{equation*}
for all $n\geq m_0$, $x$ with $|x|\geq\rho^n-1$ (of course, we may now drop the
constraint $|x| \leq \rho^n+1$). From this property, transience easily
follows.
\begin{figure}
\begin{center}\parbox{5.5cm}{\includegraphics[width=5cm]{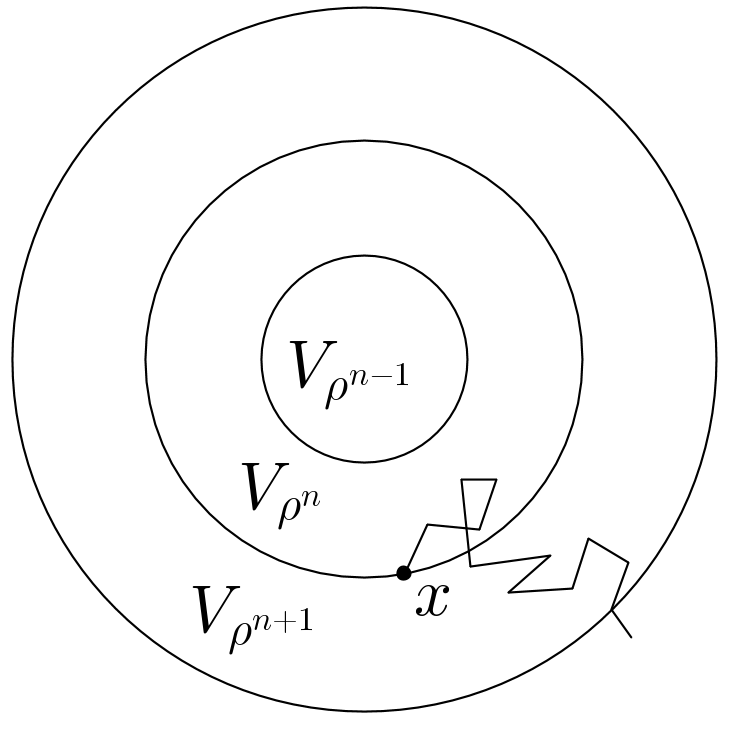}}
\parbox{9.5cm}{
\caption{On a set of environments with mass $1$, the RWRE started at any $x$
  with $|x|\geq \rho^n -1$ leaves the ball $V_{\rho^{n+1}}$ before hitting
  $V_{\rho^{n-1}}$ with probability at least $3/5$. This implies transience
  of the RWRE.}}
\end{center}
\end{figure}
Indeed, for $m, M, k \in \mathbb{N}$ satisfying $M > m\geq m_0$ and $0\leq
k \leq M+1-m$, set
\begin{equation*}
h_M(k) = \sup_{x: |x| \geq \rho^{m+k}-1}\Prw_{x,\omega}\left(T_{V_{\rho^m}}< \tau_{V_{\rho^M}}\right).
\end{equation*}
Then $h_M$ solves the difference inequality
\begin{equation*}
h_M(k) \leq \frac{2}{5}h_M(k-1) + \frac{3}{5}h_M(k+1)
\end{equation*}
with boundary conditions $h_M(0) = 1$, $h_M(M+1-m) = 0$. Further, by
either applying a discrete maximum principle or by a direct computation, we
see that $h_M \leq \overline{h}_M$, where $\overline{h}_M$
is the solution of the difference equality 
\begin{equation}
\label{eq:main-transience-4}
\overline{h}_M(k) = \frac{2}{5}\overline{h}_M(k-1) +
\frac{3}{5}\overline{h}_M(k+1)
\end{equation}
with boundary conditions $\overline{h}_M(0) = 1$, $\overline{h}_M(M+1-m)=0$. Solving
\eqref{eq:main-transience-4}, we get
\begin{equation*}
  \overline{h}_M(k) = \frac{1}{1-\left(3/2\right)^{M+1-m}} +
  \frac{1}{1-\left(2/3\right)^{M+1-m}}\left(\frac{2}{3}\right)^k. 
\end{equation*}
Letting $M\rightarrow\infty$, we deduce that for $|x|\geq \rho^{m+k}$,
\begin{equation}
\Prw_{x,\omega}\left(T_{V_{\rho^m}}< \infty\right) \leq
\lim_{M\rightarrow \infty}\overline{h}_M(k) = \left(\frac{2}{3}\right)^k.
\end{equation}
Together with~\eqref{eq:main-transience-bc}, this proves that for almost all
$\omega\in\Omega$, the random walk is transient under $\Prw_{\cdot,\omega}$.
\end{prooof2}

   \section{Mean sojourn times in the ball}
\label{times}
Using the results about the variational difference of the exit measures and
the estimates of Section~\ref{super}, we provide in this section the basis
for the proof of Proposition~\ref{def:times-main-prop}, which then leads to
Theorem~\ref{def:times-thm2}. Recall that we work under Assumption {\bf A2}.
 \subsection{Preliminaries}
 Given three real numbers $a\leq b$ and $R$, we write $[a,b]\cdot R$ for
 the interval $[aR, bR]$.  Recall the definition of $h_L$ and the
 corresponding coarse graining scheme on $V_L$ from
 Section~\ref{smoothbad-cgs}. In this part, we take a closer look at
 movements in balls $V_t(x)$ inside $V_L$, where $t>0$ is large. As in
 Section~\ref{smoothbad-cgs}, we let
\begin{equation*}
s_t = \frac{t}{(\log t)^3} \quad\mbox{and}\quad r_t= \frac{t}{(\log t)^{15}}.
\end{equation*}
We transfer the coarse graining schemes on $V_L$ in the obvious way to
$V_t(x)$. We write $\Ph_t^x$ for the transition probabilities in $V_t(x)$
belonging to $((h_t^x(y))_{y\in V_t(x)},p_\omega)$, where $h_t^x(\cdot)$
stands for $h_{t,r_t}(\cdot-x)$, which is defined
in~\eqref{eq:smoothbad-hLr}. The kernel $\ph_t^x$ is defined similarly,
with $p_\omega$ replaced by $p^{\mbox{\tiny RW}}$.

For the corresponding Green's functions we use the expressions $\Gh_t^x$ and
$\gh_t^x$.  If we do not keep $x$ as an index, we
always mean $x=0$ as before. Notice that for $y,z \in V_t(x)$, we have
$\ph_t^x(y,z) = \ph_t(y-x,z-x)$ and $\gh_t^x(y,z) =
\gh_t(y-x,z-x)$. Plainly, this is in general not true for $\Ph_t^x$ and
$\Gh_t^x$.
  
We will readily use the fact that for
simple random walk starting in $y\in V_{L}(x)$ (cf.~\cite{LawLim},
Proposition 6.2.6),
\begin{equation}
\label{eq:times-srw}
L^2-|y|^2\leq
\Erw_{y}\left[\tau_{V_L(x)}\right]\leq (L+1)^2-|y|^2.
\end{equation}
Define the ``coarse grained'' RWRE sojourn times
\begin{equation*}
  \Lambda_L(x) = 1_{V_L}(x)\, 
  \frac{1}{h_L(x)}\int\limits_{\mathbb{R}_+}\varphi\left(\frac{t}{h_L(x)}\right)\Erw_{x,\omega}\left[\tau_{V_t(x)\cap
      V_L}\right]\dt t,    
\end{equation*}
and the analog for simple random walk,
\begin{equation*}
  \lambda_L(x) = 1_{V_L}(x)\,
      \frac{1}{h_L(x)}\int\limits_{\mathbb{R}_+}\varphi\left(\frac{t}{h_L(x)}\right)\Erw_x\left[\tau_{V_t(x)\cap
        V_L}\right]\dt t.  
\end{equation*} 
We will also consider the corresponding quantities $\Lambda_t^x$, $\lambda_t^x$ for
balls $V_t(x)$. For example, 
\begin{equation*}
\Lambda_t^x(y) = 1_{V_t(x)}(y)\,\frac{1}{h_t^x(y)}\int\limits_{\mathbb{R}_+}\varphi\left(\frac{s}{h_t^x(y)}\right)\Erw_{y,\omega}\left[\tau_{V_s(y)\cap
  V_t(x)}\right]\dt s.  
\end{equation*}
We often let kernels operate on mean sojourn times from the left. As an example,
\begin{equation*}
\Gh_{L,r}\Lambda_L(x) = \sum_{y\in V_L}\Gh_{L,r}(x,y)\Lambda_L(y).
\end{equation*}
The basis for our inductive scheme is established by 
\begin{lemma}
\label{times-keylemma}
For environments $\omega\in\mathcal{P}_\varepsilon$,
$x\in\mathbb{Z}^d$,
\begin{equation*}
  \Erw_{x,\omega}\left[\tau_L\right] = \Gh_{L,r}\Lambda_L(x).
\end{equation*}
In particular, 
\begin{equation*}
\Erw_{x}\left[\tau_L\right] = \gh_{L,r}\lambda_L(x).
\end{equation*}
\end{lemma}
\begin{prooof}
  We take a probability space $(\Xi, \mathcal{A}, \pQ)$ carrying
  independently for each $x\in V_L$ a family of independent real-valued
  random variables $(\xi_x^{(n)})_{n\in \mathbb{N}}$, distributed according
  to $\frac{1}{h_L(x)}\varphi\left(\frac{t}{h_L(x)}\right)\dt t$. For the
  sake of convenience set $\xi_x^{(n)} = 1$ for all
  $x\in\mathbb{Z}^d\backslash V_L$ and all $n\in\mathbb{N}$.
  Define the filtration $\pG_n =
  \sigma\left(X_0,\ldots,X_n,\xi_{X_0}^{(0)}, \ldots, \xi_{X_{n-1}}^{(n-1)}\right)$. Here, $X_n$ is the
  projection on the $n$th component of the first factor of
  $\left(\mathbb{Z}^d\right)^{\mathbb{N}}\times \Xi$. Then $(X_n,\pG_n)$
  is a Markov chain on
  $\left(\left(\mathbb{Z}^d\right)^{\mathbb{N}}\times\Xi,
    \pG\otimes\mathcal{A}, \Prw_{x,\omega}\otimes\pQ\right)$ with
  transition kernel $p_{\omega}$ and starting point $x$. With $T_0 =
  0$, and iteratively
\begin{equation*}
T_{n+1} = \inf\left\{m>
T_n: X_m\notin
V_{\xi_{X_{T_n}}^{(T_n)}}\left(X_{T_n}\right)\right\}\wedge \tau_L,
\end{equation*}
one shows by induction that $T_n$ is a stopping time with respect to
$\pG_k$. Moreover, in $V_L$, the coarse grained chain running with
transition kernel $\Ph(\omega)$ can be obtained from $X_n$ by looking at
times $T_n$, that is by considering $(X_{T_n})_{n\geq 0}$. Denote by
$\widetilde{\Erw}_{x,\omega}$ the expectation with respect to
$\widetilde{\Prw}_{x,\omega}= \Prw_{x,\omega}\otimes\pQ$. Then, using the strong Markov property in the next to last equality,
\begin{eqnarray*}
\Erw_{x,\omega}\left[\tau_L\right]&=&\sum_{z\in
  V_L}\Erw_{x,\omega}\left[\sum_{n=0}^\infty 1_{\{z\}}(X_n)1_{\{n <
  \tau_L\}}\right] = \sum_{z\in V_L}\widetilde{\Erw}_{x,\omega}\left[\sum_{n=0}^\infty 1_{\{z\}}(X_n)1_{\{n <
  \tau_L\}}\right]\\
&=&\sum_{z\in
  V_L}\widetilde{\Erw}_{x,\omega}\left[\sum_{n=0}^\infty\sum_{k=T_n}^{T_{n+1}-1}
  1_{\{z\}}(X_k)\right]\\
&=&\sum_{z\in
  V_L}\widetilde{\Erw}_{x,\omega}\left[\sum_{n=0}^\infty\left(\sum_{y\in
    V_L}1_{\{y\}}\left(X_{T_n}\right)\right)\sum_{k=T_n}^{T_{n+1}-1}
1_{\{z\}}(X_k)\right]\\
&=&\sum_{y\in
  V_L}\sum_{n=0}^{\infty}\widetilde{\Erw}_{x,\omega}\left[1_{\{y\}}\left(X_{T_n}\right)
\widetilde{\Erw}_{x,\omega}\left[\sum_{z\in
    V_L}\sum_{k=T_n}^{T_{n+1}-1}1_{\{z\}}\left(X_k\right)\bigm|\pG_{T_n}\right]\right]\\ 
&=& \sum_{y\in
  V_L}\sum_{n=0}^\infty\widetilde{\Erw}_{x,\omega}\left[1_{\{y\}}\left(X_{T_n}\right)\right]\Lambda_L(y)
=  \Gh_{L,r}\Lambda_L(x).
\end{eqnarray*}
\end{prooof}
Note that the proof of the statement does not depend on the particular
form of the coarse graining scheme.
\subsection{Good and bad points}
As in our study of exit laws, we introduce the terminology of good and
bad points, but now with respect to both space and time. It turns out that
we need simultaneous control over two levels, which is reflected in a stronger
notion of ``goodness''.
\subsubsection{Space-good and space-bad points}
We say that $x\in V_L$ is {\it
  space-good}, if
\begin{itemize}
\item $x\in V_L\backslash \badP_L$, that is $x$ is good in the sense of
  Section~\ref{goodandbad}.
\item If $\dL(x) > 2s_L$, then additionally for all $t\in [h_L(x),2h_L(x)]$ and for all
  $y\in V_t(x)$,
  \begin{itemize}
  \item For all $t'\in [h_t^x(y), 2h_t^x(y)]$,
    $||(\Pi_{V_{t'}(y)}-\pi_{V_{t'}(y)})(y,\cdot)||_1 \leq \delta$.
  \item If $t-|y-x|  > 2r_t$, then additionally
   \begin{equation*}
     \left|\left|(\Ph_t^x -
         \ph_t^x)\ph_t^x(y,\cdot)\right|\right|_1 \leq (\log h_t^x(y))^{-9}.  
   \end{equation*}
 \end{itemize}
\end{itemize}
A point $x\in V_L$ which is not space-good is called {\it space-bad}.  The
set of all space-bad points inside $V_L$ is denoted by $\badPs$. We
classify the environments into $\goods = \{\badPs = \emptyset\}$ and
$\bads= \left\{\badPs\neq\emptyset\right\}$. Notice that $\badP_L \subset
\badPs$ and $\goods\subset \good$.
As an immediate consequence of the definition,
\begin{lemma}
\label{def:times-superlemma3}
There exists $C>0$ such that if $\delta >0$ is small, then
on $\goods$,
\begin{enumerate} 
 \item $\Gh_{L,r_L} \preceq C\Gamma_{L,r_L}$.
 \item If $x\in V_L$ with $\dL(x) > 2s_L$, then for all $t\in[h_L(x),
   2h_L(x)]$,
   \begin{equation*}
  \Gh_t^x \preceq C\Gamma_{t,r_t}(\cdot-x,\cdot-x).
\end{equation*}
\end{enumerate}
\end{lemma}
\begin{prooof}
  (i) Since $\goods\subset\good$, we have $\Gh = \Ghg$ on $\goods$, and
  Lemma~\ref{def:superlemma} can be applied.\\
  (ii) Take $x$ and $t$ as in the statement. On $\goods$, the kernel $\Gh_t^x$
  coincides with its goodified version, since within $V_t(x)$, there are no
  bad points. The claim now follows again from
  Lemma~\ref{def:superlemma}.
\end{prooof}
\begin{lemma}
\label{def:times-lemmabads}
If $L_1$ is large enough, then $\ctwo(\delta,L_1)$ implies that for
$L_1\leq L \leq L_1(\log L_1)^2$,
\begin{equation*}
\pP(\bads) \leq \exp\left(-(2/3)(\log L)^{2}\right).
\end{equation*}
\end{lemma}
\begin{prooof}
One can proceed as in the proof of Lemma~\ref{def:smoothbad-lemmamanybad}. We omit the details.
\end{prooof}
\subsubsection{Time-good and time-bad points}
We will also judge points inside $V_L$ according to their influence on the
time the RWRE spends in the ball. Remember the definitions of $f_\eta$ and
condition $\ctime(\eta,L_1)$ from Section~\ref{mainresults-meantime}. We
fix $0<\eta<1$.  For points in the bulk, we again shall control two levels. We say that a point $x\in V_L$ is {\it
  time-good}  if the following holds:
\begin{itemize}
\item For all $x\in V_L$, $t\in[h_L(x),2h_L(x)]$,
\begin{equation*}\Erw_{x,\omega}\left[\tau_{V_{t}(x)}\right]\in
  \left[1-f_\eta(s_L),\, 1+f_\eta(s_L)\right]\cdot
  \Erw_x\left[\tau_{V_{t}(x)}\right]
\end{equation*}  
\item If $\dL(x) > 2s_L$, then additionally for all $t\in
  [h_L(x),2h_L(x)]$, $y\in V_t(x)$ and for all
  $t'\in[h_t^x(y),2h_t^x(y)]$,
\begin{equation*}
\Erw_{y,\omega}\left[\tau_{V_{t'}(y)}\right]\in
  \left[1-f_\eta(s_t),\, 1+f_\eta(s_t)\right]\cdot
  \Erw_y\left[\tau_{V_{t'}(y)}\right]. 
\end{equation*}  
\end{itemize}
A point $x\in V_L$ which is not time-good is called {\it time-bad}.  We
denote by $\badPt = \badPt(\omega)$ the set of all time-bad points inside
$V_L$. Recall the definition $\mathcal{D}_L$ from
Section~\ref{smoothbad}. We let 
$\onebadt= \left\{\badPt\subset D\mbox{ for some }
  D\in\mathcal{D}_L\right\}$, $\manybadt={\left(\onebadt\right)}^c$, and 
$\goodt = \{\badPt = \emptyset\}\subset\onebadt$.

\subsubsection{Important remark}
The second point in the definition of time-good provides control over
coarse grained mean times on the preceding level, which will be crucial for
the proof of Lemma~\ref{def:times-auxillemma1}. Let us look at the first point.
If $x\in V_L$ is time-good and $\dL(x) > r_L$, then by definition of the coarse-graining,
  \begin{equation*}
    \Lambda_L(x) \in \left[1-f_\eta(s_L),\, 1+f_\eta(s_L)\right]\cdot\lambda_L(x).
  \end{equation*}
If $x\in V_L$ is time-good and $\dL(x) \leq r_L$, then at least
  \begin{equation*}
    \Lambda_L(x) \leq (1+f_\eta(s_L))\Erw_x\left[\tau_{V_{r_L}(x)}\right].
  \end{equation*}
Due to~\eqref{eq:times-srw}, this implies 
\begin{equation*}
\Lambda_L(x) \leq C(\log L)^{-6}L^2\quad\mbox{for all time-good } x\in V_L.
\end{equation*}
\begin{figure}
\begin{center}\parbox{3.5cm}{\includegraphics[width=3cm]{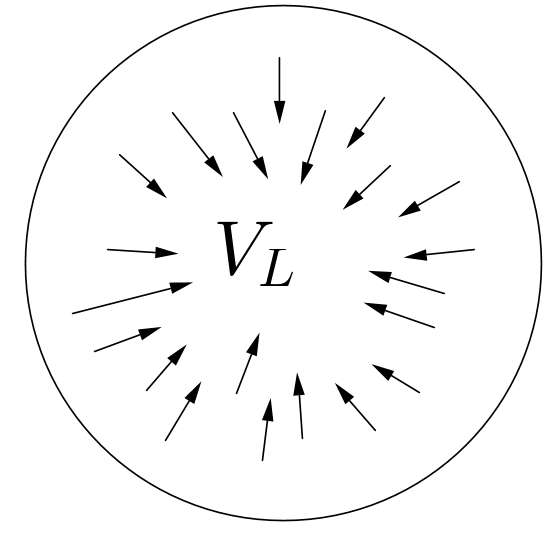}}
\parbox{10cm}{
\caption{A ``trap'': Starting at the origin, the walker is
  pushed back to the center, no
  matter in which direction he walks. On average, he needs time of
  order $\exp(cL)$ to leave the ball.}}
\end{center}
\end{figure}
However, time-bad points could possibly be very bad and give rise to a 
sojourn time which is visible on many subsequent larger scales. For example,
assume that all transition probabilities inside a ball of radius $L$ have
the tendency to push the walker towards the center of the ball (see Figure~\thefigure). Then the
mean sojourn time will be of order $\exp(cL)$ for some $c>0$. The probability of
such an event should however be exponentially small in the volume $L^d$. Of
course, between this extreme case and a well-behaved environment, there are
many intermediate configurations.  One needs to show that ``very
(time-)bad'' environments do not occur too often, which seems to be a
challenging problem. This is the point where Assumption {\bf A2} helps
out. It allows us to concentrate on the event
\begin{equation}
\label{eq:times-badnessbound}
\nottoobadt = \left\{\omega\in \Omega : \Lambda_L(x) \leq (\log L)^{-2}L^2\,\,\mbox{for all } x\in V_L\right\}.
\end{equation}
\begin{lemma}
\label{def:times-lemmanottoobadt}
If {\bf A2} holds, then for $L$ sufficiently large,
\begin{equation*}
\pP\left(\left(\nottoobadt\right)^c\right) \leq (1/2) L^{-6d}.
\end{equation*}
\end{lemma}
\begin{prooof}
First notice that with
\begin{equation*}
  E_L = \left\{\omega \in\Omega: \mbox{for all }x\in
    V_L,\mbox{with }t=2h_L(x),\,\Erw_{x,\omega}\left[\tau_{V_t(x)}\right]
    \leq (\log L)^{-2}L^2\right\}
\end{equation*}
we have $E_L \subset \nottoobadt$.
As $2h_L(x)\leq s_L/10<(\log L)^{-3}L$, the complement of $E_L$ is bounded
under {\bf A2} by $\pP(E_L^c) \leq CL^dL^{-8d}$.

\end{prooof}
\begin{remark}
\label{remark-balanced}
(i) On a certain class of environments, we can easily bound the mean
time the RWRE spends the a ball. Fix a unit vector $e\in\{e_i\}_{i=1}^d$
from the canonical basis of $\mathbb{Z}^d$. We consider an environment
$\omega\in\mathcal{P}_{\e}$ such that for each 
$x\in\mathbb{Z}^d$, $\omega_x(e) = \omega_x(-e)$, i.e. the environment is
{\it balanced}
in direction $e$. 
In such a case,
\begin{equation*}
M_n = (X_n\cdot e)^2 - \sum_{k=0}^{n-1}\left(\omega_{X_k}(e) +
  \omega_{X_k}(-e)\right)
\end{equation*}
is a $\Prw_{0,\omega}$-martingale with respect to the
filtration generated by the walk $(X_n)_{n\geq 0}$. By the stopping
theorem, $\Erw_{0,\omega}\left[M_{n\wedge\tau_L}\right] = 0$.
Since $\omega_{X_k}(e) + \omega_{X_k}(-e) \geq 1/d -2\e > 0$, it follows that
\begin{equation*}
\Erw_{0,\omega}\left[n\wedge \tau_L\right] \leq
{\left(1/d-2\e\right)}^{-1}\Erw_{0,\omega}\left[(X_{n\wedge\tau_L}\cdot
  e)^2\right].
\end{equation*}
Therefore,
\begin{equation*}
\Erw_{0,\omega}\left[\tau_L\right]\leq \frac{d}{1-2\e d}(L+1)^2,
\end{equation*}
and {\bf A2} is trivially satisfied. 

However, for measures $\mu$ which are
invariant under rotations, the class of such environments has positive
measure under $\pP_{\mu}$ only if $\mu$ is supported on the subset of
symmetric transition probabilities
$\{q\in\mathcal{P}_{\varepsilon}: q(+e_i)=
q(-e_i) \mbox{ for all }i=1,\ldots,d\}$, implying $\omega_x(e)
=\omega_x(-e)$ for all unit vectors $e$ and
$x\in\mathbb{Z}^d$ almost surely. In this case, $|X_n|^2-n$ is a quenched
martingale, and $L^2\leq\Erw_{0,\omega}\left[\tau_L\right]\leq (L+1)^2$ for
almost all environments.
\\
(ii) Before proceeding, let us mention that Assumption ${\bf A2}$ can be
expressed in terms of hitting
probabilities. For example, if there exists $\rho>0$ such that for $L$ large,
\begin{equation*}
  \pP\left(\inf_{x\in V_L}\Prw_{x,\omega}\left(\tau_L \leq (\log
      L)^{3}L^2\right) \geq \rho\right) \geq 1-L^{-8d},
\end{equation*}
then ${\bf A2}$ holds. Indeed, on the event $\{\inf_{x\in V_L}\Prw_{x,\omega}\left(\tau_L \leq
  (\log L)^3L^2\right)\geq \rho\}$,
 \begin{equation*}
 \Erw_{0,\omega}\left[\tau_L\right] \leq (\log L)^3 L^2 +
 \sum_{k=1}^\infty \Prw_{0,\omega}\left(\tau_L > k(\log L)^3
  L^2\right)(\log L)^3 L^2.
\end{equation*}
By the Markov property it follows that on this event,
\begin{equation*}
  \Prw_{0,\omega}\left(\tau_L > k(\log L)^3L^2\right) \leq (1-\rho)^k,
\end{equation*}
whence for large $L$, 
\begin{equation*}
\Erw_{0,\omega}\left[\tau_L\right] \leq (1/\rho)(\log L)^3 L^2 \leq (\log
L)^4L^2.
\end{equation*}
\end{remark}
Let us continue by showing that we can forget about environments with space-bad
points or widely spread time-bad points.
\begin{lemma}
\label{def:times-lemmamanybad}
If $L_1$ is large, then $\ctwo(\delta,L_1)$, $\ctime(\eta,L_1)$ imply
that for $L$ with $L_1\leq L \leq L_1(\log L_1)^2$,
\begin{equation*}
\pP\left(\bads\cup\manybadt\right) \leq
    (1/2) L^{-6d}. 
\end{equation*}
\end{lemma}
\begin{prooof}
  We have $\pP\left(\bads\cup\manybadt\right) \leq \pP\left(\bads\right) +
  \pP\left(\manybadt\right)$. The first summand is bounded by
  Lemma~\ref{def:times-lemmabads}. For the second, it follows from the
  definition of time-badness, $f$ and \eqref{eq:times-srw} that if
  $x\in\badPt$ and $L$ is large, then either
  \begin{equation*}
  \Erw_{x,\omega}\left[\tau_{V_t(x)}\right]\notin\left[1-f_\eta(t),\, 1+f_\eta(t)\right]\cdot
  \Erw_x\left[\tau_{V_{t}(x)}\right]
\end{equation*}
for some $t\in[h_L(x),2h_L(x)]\cap\mathbb{N}$, or, if $\dL(x)>2s_L$, 
  \begin{equation*}
  \Erw_{y,\omega}\left[\tau_{V_{t'}(y)}\right]\notin\left[1-f_\eta(t'),\, 1+f_\eta(t')\right]\cdot
  \Erw_y\left[\tau_{V_{t'}(y)}\right]
\end{equation*}
for some $y\in V_{2h_L(x)}(x)$, $t'\in[h_{h_L(x)}^x(y),2h_{2h_L(x)}^x(y)]\cap\mathbb{N}$.

Now notice that for all $x\in V_L$, we have $h_L(x) \geq r_L/20$. Moreover,
if $\dL(x) > 2s_L$, then $h_L(x)= s_L/20$, whence for all $y\in V_t(x)$,
$t\in [h_L(x),\, 2h_L(x)]$, it follows that $h_t^x(y) \geq
r_{(s_L/20)}/20$. We conclude that under $\ctime(\eta,L_1)$,
  \begin{equation*}
\pP\left(x\in\badPt\right) \leq s_L\left(r_L/20\right)^{-6d} +
CL^ds_{s_L}\left(r_{s_L/20}/20\right)^{-6d},  
\end{equation*}
and therefore
 \begin{equation*}
  \pP\left(\manybadt\right) \leq
  CL^{4d+2}\left(r_{s_L/20}/20\right)^{-12d} \leq (1/3)L^{-6d}. 
  \end{equation*}
\end{prooof}

\subsection{Estimates on mean times}
It remains to deal with environments
$\omega\in\goods\cap\onebadt\cap\nottoobadt$.  In contrast to the estimates
on exit measures, we treat all these environments at once. The main
statement of this section, Lemma~\ref{def:times-mainlemma}, can therefore
be seen as the analog for sojourn times of both
Lemmata~\ref{def:lemma-goodpart} and~\ref{def:lemma-badpart} . In the
following, we will always assume that $\delta$ and $L$ are such that
Lemma~\ref{def:times-superlemma3} can be applied. We start with two
auxiliary statements. Here, the difference estimates on the coarse
grained Green's functions from Section~\ref{super-difference} play a
crucial role.
\begin{lemma}
\label{def:times-auxillemma1}
Let $0\leq \alpha < 3$ and $x, y\in
V_{L-2s_L}\backslash\badPt$ with $|x-y| \leq (\log
s_L)^{-\alpha}\,s_L$. On $\goods$, 
\begin{equation*}
\left|\Lambda_L(x) - \Lambda_L(y)\right| \leq C(\log\log s_L)(\log s_L)^{-\alpha}
  s_L^2.
\end{equation*}
\end{lemma}
\begin{prooof}
The claim follows if we show that for all $t\in
\left[(1/20)s_L,(1/10)s_L\right]$,
\begin{equation*}
  \left|\Erw_{x,\omega}\left[\tau_{V_t(x)}\right] -
    \Erw_{y,\omega}\left[\tau_{V_t(y)}\right]\right| \leq C(\log\log t)(\log t)^{-\alpha}
  t^2.
\end{equation*}
Set $t' = \left(1-20(\log t)^{-\alpha}\right)t$. Then $V_{t'}(x) \subset
  V_t(x)\cap V_t(y)$. Further, let $B = V_{t'-2s_t}(x)$. By Lemma~\ref{times-keylemma},
\begin{equation}
\label{eq:times-auxillemma1-1}
\Erw_{x,\omega}\left[\tau_{V_t(x)}\right] = \Gh_t^x1_B\Lambda_t^x(x) +
 \Gh_t^x1_{V_{t}(x)\backslash B}\Lambda_t^x(x).
\end{equation}
Since $x\in
V_{L-2s_L}\backslash\badPt$, it follows that $\Lambda_t^x(z)\leq C(\log
t)^{-6}t^2$, for all $z\in V_t(x)$. Moreover, since $\omega \in \goods$, we
have by Lemma~\ref{def:times-superlemma3} $\Gh_t^x\preceq C\Gamma_{t,r_t}(\cdot-x,\cdot-x)$. Thus,
Lemma~\ref{def:super-gammalemma} (iv) yields
\begin{equation*}
\Gh_t^x1_{V_{t}(x)\backslash B}\Lambda_t^x(x)\leq
C\Gamma_{t,r_t}\left(0,V_t\backslash V_{t'-2s_t}\right)\,(\log
t)^{-6}t^2 \leq  (\log t)^{-\alpha}t^2.   
\end{equation*} 
for $L$ (and therefore also $t$) sufficiently large. 
Concerning $\Erw_{y,\omega}\left[\tau_{V_t(y)}\right]$, we split again into
\begin{equation*}
  \Erw_{y,\omega}\left[\tau_{V_t(y)}\right]= \Gh_t^y1_B\Lambda_t^y(y) + \Gh_t^y1_{V_t(y)\backslash B}\Lambda_t^y(y).
\end{equation*}
As above, the second summand is bounded by $(\log t)^{-\alpha}t^2$. For $z \in
B$, we have $h_t^x(z) = h_t^y(z) = (1/20)s_t$. In particular,
$\Ph_t^x(z,\cdot) = \Ph_t^y(z,\cdot)$, and also $\Lambda_t^x(z) =
\Lambda_t^y(z)$. Since both $x$ and $y$ are contained in
$B\subset V_t(x)\cap V_t(y)$, the strong Markov property gives
\begin{equation*}
\Gh_t^y(y,z) =
\Gh_t^x(y,z) +b(y,z),
\end{equation*}
where
\begin{equation*}
  b(y,z) = \Erw_{y,\Ph_t^y (\omega)}\left[\Gh_t^y(\tau_B,z);\,
    \tau_B < \infty\right] -\Erw_{y,\Ph_t^x(\omega)}\left[\Gh_t^x(\tau_B,z);\,
    \tau_B < \infty\right].
\end{equation*}
Therefore,
\begin{eqnarray*}
  \lefteqn{\left|\Erw_{x,\omega}\left[\tau_{V_t(x)}\right]-\Erw_{y,\omega}\left[\tau_{V_t(y)}\right]\right|}\\
  &\leq& 2(\log t)^{-\alpha}t^2 +\sum_{z\in B}\left(\left|\Gh_t^x(x,z)-\Gh_t^x(y,z)\right| + |b(y,z)|\right)\Lambda_t^x(z).
\end{eqnarray*}
The quantity $\Lambda_t^x(z)$ is estimated as above. For the sum over $|b(y,z)|$, we
notice that if $w\in V_t(y)\backslash B$, then $t-|w-y|\leq C(\log
t)^{-\alpha}t$. We can
use twice Lemma~\ref{def:super-gammalemma} (v) to get
\begin{equation*}
\sum_{z\in B}|b(y,z)| \leq \sup_{v\in
  V_t(x)\backslash B}\Gh_t^x(v,B) + \sup_{w\in
  V_t(y)\backslash B}\Gh_t^y(w,B) \leq C(\log t)^{6-\alpha}.
\end{equation*}
Finally, for the sum over the Green's function difference, we recall that
$\Gh_{t}^x$ coincides with its goodified version, so we may apply
Lemma~\ref{def:super-greendifference}.  Doing so $O\left((\log
  t)^{3-\alpha}\right)$ times gives
\begin{equation*}
\sum_{z\in B}\left|\Gh_{t}^x(x,z)-\Gh_{t}^x(y,z)\right| \leq
C(\log\log t)(\log t)^{6-\alpha}.
\end{equation*}
This proves the statement.
\end{prooof}

\begin{lemma}
\label{def:times-auxillemma2}
Set $\Delta = 1_{V_L}(\Ph_{L,r_L}-\ph_{L,r_L})$. On $\goods\cap\onebadt\cap\nottoobadt$,
\begin{equation*}
\sup_{x\in V_L}\left|\Gh_{L,r_L}\Delta\gh_{L,r_L}\Lambda_L(x)\right| \leq C(\log L)^{-5/3}L^2.
\end{equation*} 
\end{lemma}
\begin{prooof}
We have
\begin{equation*}
\Gh\Delta\gh\Lambda_L(x) = \Gh\Delta\ph\gh\Lambda_L(x) + 
\Gh\Delta\Lambda_L(x)= A_1 + A_2.
\end{equation*}
By Lemma~\ref{def:times-superlemma3}, $\Gh = \Ghg \preceq
C\Gamma$. Therefore, with $B_1= V_{L-2r_L}$, we bound $A_1$ by
\begin{eqnarray*}
  \left|A_1\right| &\leq& \left|\Gh1_{B_1}\Delta\ph\gh\Lambda_L(x)\right| + \left|\Gh1_{B_1^c}\Delta\ph\gh\Lambda_L(x)\right|\\
  &\leq& \left|\sum_{v\in B_1, w\in V_L}\Gh(x,v)\Delta\ph(v,w)\sum_{y\in
    V_L}\left(\gh(w,y)-\gh(v,y)\right)\Lambda_L(y)\right|+ C(\log L)^{-2}L^2 \\
  &\leq& C(\log L)^{-5/3}L^2,
 \end{eqnarray*}
 where in the next to last inequality we have used the bound on
 $\Lambda_L(y)$ coming from~\eqref{eq:times-badnessbound},
 Lemma~\ref{def:super-gammalemma} (iv), (v) and in the last additionally
 Lemma~\ref{def:super-greendifference}.
 For the term $A_2$, we let 
\begin{equation*}
U(\badPt)= \left\{v\in V_L : |\Delta(v,w)|
   > 0\mbox{ for some } w\in\badPt\right\}
\end{equation*}
and define $B= V_{L-5s_L}\backslash U(\badPt)$. We split into
\begin{equation*}
  A_2 = \Gh1_{B}\Delta\Lambda_L(x) + \Gh1_{B^c}\Delta\Lambda_L(x).
\end{equation*}
Lemma~\ref{def:super-gammalemma}
(iv) and an analogous application of Corollary~\ref{def:super-cor} with $U(\badPt)$ instead of $\badP_L$ yield
\begin{equation*}
\Gh(x,U(\badPt)\cup\sh_L(5s_L))\leq
C\log\log L.
\end{equation*}
Since $\Lambda_L(y)\leq (\log L)^{-2}L^2$, this estimates the
second summand of $A_2$. For the first one, 
\begin{equation*}
\Gh1_B\Delta\Lambda_L(x) \leq C\Gamma(x,B)\sup_{v\in
  B}\left|\Delta\Lambda_L(v)\right|.
\end{equation*}
Since $\Gamma(x,B) \leq C(\log L)^6$, the claim follows we show that for
$v\in B$,  
\begin{equation}
\label{eq:times-mainlemma-1}
\left|\Delta\Lambda_L(v)\right| \leq C(\log L)^{-8}L^2,
\end{equation}
which, by definition of $\Delta$, in turn follows if for all
$t\in[h_L(v), 2h_L(v)]$,
\begin{equation*}
\left|\left(\Pi_{V_t(v)}-\pi_{V_t(v)}\right)\Lambda_L(v)\right| \leq C(\log L)^{-8}L^2.
\end{equation*}
Notice that on $B$, $h_L(\cdot) = (1/20)s_L$. We now fix $v\in B$ and
$t\in[(1/20)s_L, (1/10)s_L]$. Set $\Delta'=
1_{V_t(v)}(\Ph_t^v-\ph_t^v)$ and $B'=
V_{t-2r_t}(v)$. By expansion~\eqref{eq:prel-pbe1},
\begin{equation}
\label{eq:times-mainlemma-2}
\left(\Pi_{V_t(v)}-\pi_{V_t(v)}\right)\Lambda_L(v) =
\Gh_t^v1_{B'}\Delta'\pi_{V_t(v)}\Lambda_L(v) +
\Gh_t^v1_{V_t(v)\backslash B'}\Delta'\pi_{V_t(v)}\Lambda_L(v).
\end{equation}
Since $\pi_{V_t(v)} = \ph_t^v\pi_{V_t(v)}$, we get
\begin{eqnarray*}
\left|\Gh_t^v1_{B'}\Delta'\pi_{V_t(v)}\Lambda_L(v)\right| &\leq& 
  \Gh_t^v(v,B')\sup_{w\in
    B'}||\Delta'\ph_t^v(w,\cdot)||_1\sup_{y\in \partial V_t(v)}\Lambda_L(y)\\
&\leq& C(\log s_L)^6\sup_{w\in B'}(\log h_t^v(w))^{-9}(\log L)^{-6}L^2 \\
&\leq& C(\log L)^{-9}L^2.
\end{eqnarray*}
Here, in the next to last inequality we have used the fact that $v$ is
space-good, all $y\in\partial V_t(v)$ are time-good, and
Lemma~\ref{def:super-gammalemma} (v). The last inequality follows from the
bound $h_t^v(w) \geq (1/20)r_{s_L/20}$.  For the second summand
of~\eqref{eq:times-mainlemma-2}, Lemma~\ref{def:super-gammalemma} (iv) gives
$\Gh_t^v(v,V_t(v)\backslash B') \leq C$, whence
\begin{equation*}
\left|\Gh_t^v1_{V_t(v)\backslash
  B'}\Delta'\pi_{V_t(v)}\Lambda_L(v)\right| \leq C \sup_{w\in V_t(v)\backslash
  B'}\left|\Delta'\pi_{V_t(v)}\Lambda_L(w)\right|. 
\end{equation*}
Fix $w\in V_t(v)\backslash B'$. Set $\eta =
  \dist(w,\partial V_t(v))\leq 2r_t + \sqrt{d}$ and choose $y_w\in\partial V_t(v)$ such that $|w-y_w| =
  \eta$. With
\begin{equation*}
I(y_w) = \left\{y\in\partial
    V_t(v) : |y-y_w|\leq (\log L)^{-5/2}s_L\right\},
\end{equation*}
we write
\begin{eqnarray}
\label{eq:times-mainlemma-3}
\lefteqn{\Delta'\pi_{V_t(v)}\Lambda_L(w)}\nonumber\\
& =& \sum_{y\in\partial
  V_t(v)}\Delta'\pi_{V_t(v)}(w,y)\left(\Lambda_L(y)-\Lambda_L(y_w)\right)\nonumber\\
&=& \sum_{y \in
  I(y_w)}\Delta'\pi_{V_t(v)}(w,y)\left(\Lambda_L(y)-\Lambda_L(y_w)\right)\nonumber\\
&&\, +
\sum_{y\in\partial
  V_t(v)\backslash
  I(y_w)}\Delta'\pi_{V_t(v)}(w,y)\left(\Lambda_L(y)-\Lambda_L(y_w)\right).
 \end{eqnarray}
For $y\in I(y_w)$, Lemma~\ref{def:times-auxillemma1} yields
$|\Lambda_L(y)-\Lambda_L(y_w)| \leq C(\log L)^{-7/3}s_L^2$. Therefore, 
\begin{equation*}
\sum_{y \in
  I(y_w)}\left|\Delta'\pi_{V_t(v)}(w,y)\right|\left|\Lambda_L(y)-\Lambda_L(y_w)\right|
\leq C(\log L)^{-8}L^2.
\end{equation*}
It remains to handle the second term of~\eqref{eq:times-mainlemma-3}. To
this end, let $U(w) = \{u\in V_t(v) : |\Delta'(w,u)|> 0\}$.
Using for $y\in \partial V_t(v)\backslash I(y_w)$ the simple bound
$\left|\Lambda_L(y)-\Lambda_L(y_w)\right|\leq \Lambda_L(y) + \Lambda_L(y_w)
\leq C(\log L)^{-6}L^2$,
\begin{eqnarray*}
\lefteqn{\left|\sum_{y\in\partial
    V_t(v)\backslash
    I(y_w)}\Delta'\pi_{V_t(v)}(w,y)\left(\Lambda_L(y)-\Lambda_L(y_w)\right)\right|}\\
&\leq& C(\log L)^{-6}L^2\sup_{u\in U(w)}\pi_{V_t(v)}\left(u,\partial
V_t(v)\backslash I(y_w)\right).
\end{eqnarray*}
If $u\in U(w)$ and $y\in\partial V_t(v)\backslash I(y_w)$, then 
\begin{equation*}
|u-y| \geq |y-y_w| -|y_w-u| \geq (\log L)^{-5/2}s_L - 3r_t \geq (1/2)(\log L)^{-5/2}s_L.
\end{equation*}
For such $u$, we get by Lemma~\ref{def:hittingprob} (ii)  and
  Lemma~\ref{def:hittingprob-technical}
\begin{eqnarray*}
\pi_{V_t(v)}\left(u,\partial V_t(v)\backslash I(y_w)\right) &\leq& C r_t
\sum_{y\in \partial V_t(v)\backslash I(y_w)}\frac{1}{|u-y|^d} \leq
Cr_t(\log L)^{5/2}(s_L)^{-1}\\
& \leq& C(\log L)^{-9}.
\end{eqnarray*}
This bounds the second term
of~\eqref{eq:times-mainlemma-3}. We have proved~\eqref{eq:times-mainlemma-1} and hence the lemma.
\end{prooof}
Now it is easy to prove
\begin{lemma}
\label{def:times-mainlemma}
There exists $L_0=L_0(\eta)$ such that for $L\geq L_0$ and environments $\omega\in \goods\cap\onebadt\cap\nottoobadt$,
\begin{equation*}
  \Erw_{0,\omega}\left[\tau_L\right] \in
    \left[1-f_\eta(L),\,1+f_\eta(L)\right]\cdot\Erw_0\left[\tau_L\right].
\end{equation*}
\end{lemma}
\begin{prooof}
  By Lemma~\ref{times-keylemma} and perturbation
  expansion~\eqref{eq:prel-pbe1}, with $\Delta = 1_{V_L}(\Ph-\ph)$,
\begin{equation*}
  \Erw_{0,\omega}\left[\tau_L\right] = \Gh\Lambda_L(0)
  = \gh\Lambda_L(0) +\Gh\Delta\gh\Lambda_L(0) = A_1 + A_2.
\end{equation*}
Set $B= V_{L-r_L}\backslash \badPt$. The term $A_1$ we split into
\begin{equation*}
A_1 = \gh1_B\Lambda_L(0) + \gh1_{V_L\backslash B}\Lambda_L(0).
\end{equation*}
Since $\gh(0,V_L\backslash B) \leq C$ and $\Lambda_L(x) \leq (\log
L)^{-2}L^2$, the second summand of $A_1$ can be bounded by $O\left((\log
  L)^{-2}\right)\Erw_0[\tau_L]$. The main contribution comes from the first
summand. First notice that
\begin{equation*}
\gh1_B\lambda_L(0) = \Erw_0\left[\tau_L\right]\left(1+O\left((\log L)^{-6}\right)\right).
\end{equation*}
Further, we have for $x\in B$,
\begin{equation*}
  \Lambda_L(x) \in \left[1-f_\eta\left((\log L)^{-3}L\right),\, 1+f_\eta\left((\log L)^{-3}L\right)\right]\cdot\lambda_L(x).
\end{equation*}
Collecting all terms, we conclude that
\begin{equation*}
\begin{split}
  A_1 \in &\left[1 - O\left((\log L)^{-2}\right)- f_\eta\left((\log
      L)^{-3}L\right),\, 1 + O\left((\log L)^{-2}\right) + f_\eta\left((\log
      L)^{-3}L\right)\right]\\
&\times\Erw_0\left[\tau_L\right]. 
\end{split}
\end{equation*}
Lemma~\ref{def:times-auxillemma2} bounds $A_2$ by $O((\log
L)^{-5/3})\Erw_0[\tau_L]$. Since for $L$ sufficiently large,
\begin{equation*}
f_\eta(L) > f_\eta\left((\log L)^{-3}L\right) +C(\log L)^{-5/3},
\end{equation*}
we arrive at
\begin{equation*}
  \Erw_{0,\omega}\left[\tau_L\right] = A_1 + A_2 \,\in\, 
  \left[1-f_\eta(L),\,1+f_\eta(L)\right] \cdot\Erw_0\left[\tau_L\right].
\end{equation*}
\end{prooof}

   \section{Proofs of the main results on sojourn times}
\label{proofmaintimes}
\begin{prooof2}{\bf of Proposition~\ref{def:times-main-prop}:}
  (i) From Lemmata~\ref{def:times-lemmanottoobadt},~\ref{def:times-lemmamanybad}
  and~\ref{def:times-mainlemma} we deduce that for large $L_0$, if $L_1\geq
  L_0$ and $L_1\leq L\leq L_1(\log L_1)^2$, we have under
  $\ctwo(\delta,L_1)$ and $\ctime(\eta,L_1)$
  \begin{eqnarray*}
    \lefteqn{\pP\left(\Erw_{0,\omega}\left[\tau_L\right] \notin \left[1-f(L),\,
          1+f(L)\right]\cdot\Erw_0\left[\tau_L\right]\right)}\\  
    &\leq& \pP\left(\bads\cup\manybadt\right) + \pP\left((\nottoobadt)^c\right)\\
    &&+\ 
    \pP\left(\left\{\Erw_{0,\omega}\left[\tau_L\right] \notin
        \left[1-f(L),\,1+f(L)\right]\cdot\Erw_0[\tau_L]\right\}\right.\\
    &&\quad\quad\,\, \left.\cap\goods\cap\onebadt\cap\nottoobadt\right)\\
    &\leq& L^{-6d}. 
  \end{eqnarray*}
  By Proposition~\ref{def:main-prop}, if $\delta>0$ is small,
  $\ctwo(\delta,L)$ holds under $\cone(\e)$ for all large $L$, provided
  $\e\leq\e_0(\delta)$. This proves part (i) of the proposition.\\
  (ii) We take $L_0$ from part (i). By
  choosing $\e$ small enough, we can guarantee that $\ctime(\eta,L_0)$
  holds. Then, by what we just proved, $\ctime(\eta,L)$
  holds for all $L\geq L_0$. Recalling~\eqref{eq:times-srw}, we therefore
  have for large $L\geq L_0$
\begin{eqnarray*}
  \lefteqn{\pP\left(\sup_{x: |x| \leq L^{3}}\sup_{y\in
        V_L(x)}\Erw_{y,\omega}\left[\tau_{V_L(x)}\right] \notin
      [1-\eta,1+\eta]\cdot L^2\right)}\\
  &\leq&
  CL^{3d}\,\pP\left(\sup_{y\in V_L}\Erw_{y,\omega}\left[\tau_L\right] \notin
    [1-\eta,1+\eta]\cdot L^2\right)\\
  &\leq&CL^{3d}\,\pP\left(\Erw_{0,\omega}\left[\tau_L\right]  <
    (1-\eta)\cdot L^2\right)+ CL^{3d}\,\pP\left(\sup_{y\in V_L}\Erw_{y,\omega}\left[\tau_L\right]  >
    (1+\eta)\cdot L^2\right)\\
  &\leq& CL^{-3d} +
  CL^{4d}\,\pP\left(\Erw_{0,\omega}\left[\tau_L\right]>(1+\eta)\cdot
    L^2\right) \leq L^{-d}.
\end{eqnarray*}  
\end{prooof2}
\begin{prooof2}{\bf of Corollary~\ref{def:times-cormoments}:}
First let $k=1$. Using Proposition~\ref{def:times-main-prop} (ii) and Borel-Cantelli, we
obtain for $\pP$-almost all $\omega$
\begin{equation}
\label{eq:times-cormoments-1}
\limsup_{L\rightarrow\infty}\sup_{x: |x| \leq L^{3}}\sup_{y\in
    V_L(x)}\Erw_{y,\omega}\left[\tau_{V_L(x)}\right]/L^2 \leq 2.
\end{equation}
For the rest of the proof, take an environment $\omega$ 
satisfying~\eqref{eq:times-cormoments-1}.  Assume $k \geq 2$. Then
\begin{eqnarray*}
\Erw_{y,\omega}\left[\tau^k_{V_L(x)}\right] &=&
\sum_{l_1,\ldots,l_k\geq 0}\Prw_{y,\omega}\left(\tau_{V_L(x)} >
  l_1,\ldots,\tau_{V_L(x)}>l_k\right)\\
&\leq& k!\sum_{0\leq l_1\leq\ldots\leq
  l_k}\Prw_{y,\omega}\left(\tau_{V_L(x)} > l_k\right).
\end{eqnarray*}
By the Markov property, using the case $k=1$ and induction in the last step,
\begin{eqnarray*}
\lefteqn{\sum_{0\leq l_1\leq\ldots\leq
  l_k}\Prw_{y,\omega}\left(\tau_{V_L(x)} > l_k\right)}\\ 
&=& \sum_{0\leq l_1\leq\ldots\leq
  l_{k-1}}\Erw_{y,\omega}\left[\sum_{l=0}^\infty\Prw_{X_{l_{k-1}},\omega}\left(\tau_{V_L(x)} >
    l\right);\, \tau_{V_L(x)}> l_{k-1}\right]\\
&\leq& \sup_{z\in V_L(x)}\Erw_{z,\omega}\left[\tau_{V_L(x)}\right]\sum_{0\leq l_1\leq\ldots\leq
  l_{k-1}}\Erw_{y,\omega}\left[\tau_{V_L(x)}> l_{k-1}\right]\\
&\leq& 2^kL^{2k},
\end{eqnarray*}
if $L = L(\omega)$ is sufficiently large.
\end{prooof2}

\begin{prooof2}{\bf of Theorem~\ref{def:times-thm2}:}
  Both statements are proved in the same way, so we restrict ourselves to
  the $\limsup$. 
  Set $\overline{\tau}_{V_L(x)} = \tau_{V_L(x)}/L^2$ and
  \begin{equation*}
    B_1
    = \left\{\limsup_{L\rightarrow\infty}\sup_{x: |x| \leq
        L^{3}}\sup_{y\in
        V_L(x)}\Erw_{y,\omega}\left[\overline{\tau}_{V_L(x)}\right]\in[1-\eta,1+\eta]\right\}.
  \end{equation*}
  By Proposition~\ref{def:times-main-prop} and Borel-Cantelli it follows that
  $\pP\left(B_1\right) = 1$ if $\e\leq \e_0$. Moreover, on $B_1$ the conclusion of
  Corollary~\ref{def:times-cormoments} holds
  true. Corollary~\ref{def:main-transience} tells us that for small enough
  $\e$, on a set $B_2$ of full measure the RWRE satisfies~\eqref{eq:main-transience-eq} and is
  therefore transient. Let $B = B_1\cap B_2$ and define  
\begin{equation*}
  \xi = \left\{\begin{array}{l@{\quad \mbox{for\ }}l}
      \limsup_{L\rightarrow\infty}\sup_{x: |x| \leq
        L^{3}}\sup_{y\in
        V_L(x)}\Erw_{y,\omega}\left[\overline{\tau}_{V_L(x)}\right]& \omega\in B\\
      0 & \omega\in\Omega\backslash B\end{array}\right..
\end{equation*}
Choose an bijective
enumeration function $g:\mathbb{Z}^d\rightarrow\mathbb{N}$ with $g(0) =
0$ and $g(x) < g(y)$ whenever $|x|< |y|$. Let $\mathcal{N}$ denote the
collection of all $\pP$-null sets in $\pF$ and set $\pF_n' =
\sigma\left(\mathcal{N}, Z_n, Z_{n+1},\ldots\right)$, where
$Z_k:\Omega\rightarrow\mathcal{P}$, $Z_k(\omega) =
\omega_{g^{-1}(k)}$, is the projection on the
$g^{-1}(k)$-th component. Let $\mathcal{T} =
\cap_n\pF_n'$ be the (completed) tail $\sigma$-field.  We show that $\xi$
is measurable with respect to $\mathcal{T}$, implying that $\xi$ is
$\pP$-almost surely constant. Take $\omega\in B$. 
We claim that for each fixed ball $V_l$ around the origin, $T_l$ its hitting time,
\begin{equation}
\label{eq:times-thm2-1} 
\xi(\omega) = \underbrace{\limsup_{L\rightarrow\infty}\sup_{x: |x| \leq
    L^{3}}\sup_{y\in V_L(x)\backslash
    V_l}\Erw_{y,\omega}\left[\overline{\tau}_{V_L(x)};\,
    T_l = \infty\right]}_{= \xi_l(\omega)}.
\end{equation}
But then also
\begin{equation*}
\xi(\omega) = \lim_{l\rightarrow\infty}\xi_l(\omega).
\end{equation*}
Since $\xi_l$ depends only on the random variables $\omega_x$ with $|x| >
l$, $\xi$ is in fact measurable with respect to $\mathcal{T}$, provided the
above representation holds true. Therefore, we only have to
prove~\eqref{eq:times-thm2-1}. Obviously, $\xi(\omega)\geq
\xi_l(\omega)$. For the other direction, by the Markov property in the
first inequality,
\begin{eqnarray*}
  \Erw_{y,\omega}\left[\overline{\tau}_{V_L(x)}\right] &=&
  \Erw_{y,\omega}\left[\overline{\tau}_{V_L(x)};\, \tau_{V_L(x)} \leq
    L\right] + \Erw_{y,\omega}\left[\overline{\tau}_{V_L(x)};\,
    \tau_{V_L(x)}>L\right]\\
  &\leq& 2L^{-1} + \Erw_{y,\omega}\left[\Erw_{X_L,\omega}\left[\overline{\tau}_{V_L(x)}\right];\,
    \tau_{V_L(x)}>L\right]\\
  &\leq& 2L^{-1} +
  \Erw_{y,\omega}\left[\Erw_{X_L,\omega}\left[\overline{\tau}_{V_L(x)};\, T_l<\infty\right];\,
    \tau_{V_L(x)}>L\right]\\
  && +\ 
  \Erw_{y,\omega}\left[\Erw_{X_L,\omega}\left[\overline{\tau}_{V_L(x)};\,
      T_l=\infty\right];\, 
    \tau_{V_L(x)}>L\right]. 
\end{eqnarray*} 
Clearly,
\begin{equation*}
\Erw_{y,\omega}\left[\Erw_{X_L,\omega}\left[\overline{\tau}_{V_L(x)};\, T_l=\infty\right];\,
  \tau_{V_L(x)}>L\right] \leq \sup_{y\in V_L(x)\backslash
  V_l}\Erw_{y,\omega}\left[\overline{\tau}_{V_L(x)};\, T_l=\infty\right], 
\end{equation*}
so $\xi(\omega)\leq \xi_l(\omega)$ will follow if we show that
\begin{equation}
\label{eq:times-thm2-2}
\limsup_{L\rightarrow\infty}\sup_{x: |x| \leq
    L^{3}}\sup_{y\in V_L(x)}\Erw_{y,\omega}\left[\Erw_{X_L,\omega}\left[\overline{\tau}_{V_L(x)};\, T_l<\infty\right];\,
    \tau_{V_L(x)}>L\right] = 0. 
\end{equation}
By Cauchy-Schwarz in the first and Corollary~\ref{def:times-cormoments}
in the last inequality, for large $L$,
\begin{eqnarray*} 
  \lefteqn{\Erw_{y,\omega}\left[\Erw_{X_L,\omega}\left[\overline{\tau}_{V_L(x)};\, T_l<\infty\right];\,
      \tau_{V_L(x)}>L\right]}\\
  &\leq&\Erw_{y,\omega}\left[\Erw_{X_L,\omega}\left[\overline{\tau}_{V_L(x)}^2\right]^{1/2}\Prw_{X_L,\omega}\left(T_l<\infty\right)^{1/2};\,\tau_{V_L(x)}>L\right]\\  
  &\leq&\sup_{z\in
    V_L(x)}\Erw_{z,\omega}\left[\overline{\tau}_{V_L(x)}^2\right]^{1/2}\Erw_{y,\omega}\left[\Prw_{X_L,\omega}\left(T_l<\infty\right)^{1/2}\right]\\
  &\leq& 3\Erw_{y,\omega}\left[\Prw_{X_L,\omega}\left(T_l<\infty\right)^{1/2}\right].
\end{eqnarray*} 
For the probability inside the expectation, note that as a consequence
of~\eqref{eq:main-transience-eq}, for each $\vartheta > 0$ we can choose
$K=K(\omega,l)$ such that
\begin{equation*}
  \sup_{z: |z| \geq K}\Prw_{z,\omega}\left(T_l < \infty\right) \leq \vartheta^2/81.
\end{equation*}
Therefore, replacing the probability by $1$ on $\{|X_L|< K\}$,
\begin{equation*}
\Erw_{y,\omega}\left[\Prw_{X_L,\omega}\left(T_l<\infty\right)^{1/2}\right]
\leq \vartheta/9 + \Prw_{y,\omega}\left(|X_L|< K\right).
\end{equation*}
Using again~\eqref{eq:main-transience-eq}, there exists $K' =
K'(\omega,K)$ such that
\begin{equation*}
\sup_{y: |y| \geq K'}\Prw_{y,\omega}\left(|X_L|< K\right)\leq\sup_{y: |y| \geq
  K'}\Prw_{y,\omega}\left(T_K<\infty\right) \leq \vartheta/9.
\end{equation*}
 On the other hand, for each fixed $K'>0$ and $K>0$, transience also implies
\begin{equation*}
\sup_{y\in V_{K'}}\Prw_{y,\omega}\left(|X_L| < K\right)\leq \vartheta/9,
\end{equation*}
if $L$ is large enough. Altogether, we have shown that for $L$ sufficiently
large,
\begin{equation*}
\sup_{x: |x|\leq L^3}\sup_{y\in V_L(x)}\Erw_{y,\omega}\left[\Erw_{X_L,\omega}\left[\overline{\tau}_{V_L(x)};\, T_l<\infty\right];\,
    \tau_{V_L(x)}>L\right] \leq \vartheta.
\end{equation*}
Since $\vartheta$ can be chosen arbitrarily small, this
shows~\eqref{eq:times-thm2-2} from which we deduce~\eqref{eq:times-thm2-1}.
\end{prooof2}

   \section{Appendix}
\subsection{Some difference estimates}
In this section we collect some difference estimates of (non)-smoothed
exit distributions needed to prove Lemma~\ref{def:est-phi}
(i) and (iii). The first technical lemma compares the exit measure on
 $\partial V_L$ of simple random walk to that on $\partial C_L$ of standard
 Brownian motion. 
\begin{lemma}
\label{def:app-kernelest-technical}
Let $\beta,\eta >0$ with $3\eta<\beta<1$. For large
$L$, there exists a constant $C = C(\beta,\eta) > 0$ such that for
$A\subset\mathbb{R}^d$, $A^\beta = \{y\in\mathbb{R}^d : \dist(y,A)\leq
L^\beta\}$ and $x\in V_L$ with $\dL(x) > L^\beta$, the
following holds.
%
\begin{enumerate}
\item $\pi_L(x,A) \leq
\pibm_L\left(x,A^\beta\right)\left(1+CL^{-(\beta-3\eta)}\right) + L^{-(d+1)}.$
\item $\pibm_L(x,A) \leq \pi_L(x,A^{\beta})\left(1+CL^{-(\beta-3\eta)}\right) + L^{-(d+1)}.$
\end{enumerate}
\end{lemma}
\begin{prooof}
  (i) Set $L' = L + L^\eta$, $L''= L +2L^\eta$ and denote by $A_{'}^\beta$
  the image of $A^\beta$ on $\partial C_{L'}$ under the map $y\mapsto
  (L'/L)y$. With $\hat{x} = (L'/L)x$, using the Poisson kernel
  representation~\eqref{eq:est-poissonkernel} in the second equality,
\begin{equation*}
  \pibm_L\left(x,A^\beta\right) =
  \pibm_{L'}\left(\hat{x},A_{'}^\beta\right) =
  \int\limits_{A_{'}^\beta}\frac{\left((L')^2 -
      |\hat{x}|^2\right)|x-y|^d}{\left((L')^2-|x|^2\right)|\hat{x}-y|^d}\pibm_{L'}(x,dy).
\end{equation*}
Since $|x| \leq L+1 -L^\beta$ and $0<\eta < \beta < 1$, an evaluation of
the integrand shows
\begin{equation}
\label{eq:app-kernelest-technical-eq1}
  \pibm_L\left(x,A^\beta\right) \geq\pibm_{L'}(x,A_{'}^\beta)\left(1-C(\beta,\eta)L^{-(\beta-\eta)}\right)
\end{equation}
for some positive constant $C(\beta,\eta)$. By~\cite{ZAI}, Corollary 1, for
each $k\in\mathbb{N}$ there exists a constant $C_1=C_1(k) > 0$ such that
for each integer $n\geq 1$, one can construct on the same probability space
a Brownian motion $W_t$ with covariance matrix $d^{-1}I_d$ as well as
simple random walk $X_n$, both starting in $x$ and satisfying (with
$\mathbb{Q}$ denoting the probability measure on that space)
\begin{equation}
\label{eq:app-kernelest-technical-eq2}
  \mathbb{Q}\left(\max_{0\leq m\leq n}\left|X_m-W_m\right| > C_1\log
    n\right) \leq C_1 n^{-k}.
\end{equation}
Choose $k >(2/5)(d+1)$ and let $C_1(k)$ be the corresponding constant. The
following arguments hold for sufficiently large $L$. By standard results on the
oscillation of Brownian paths, 
\begin{equation}
\label{eq:app-kernelest-technical-eq3}
\mathbb{Q}\left(\sup_{0\leq t\leq L^{5/2}}\left|W_{\lfloor
      t\rfloor}-W_t\right| > (5/2)C_1\log L\right) \leq (1/3)L^{-(d+1)}.
\end{equation}
With
\begin{equation*}
B_1= \left\{\sup_{0\leq t\leq L^{5/2}}\left|X_{\lfloor
      t\rfloor}-W_t\right| \leq 5C_1\log L\right\},
\end{equation*}
we deduce from~\eqref{eq:app-kernelest-technical-eq2}
and~\eqref{eq:app-kernelest-technical-eq3} that
\begin{equation*}
  \mathbb{Q}\left({B}^c_1\right) \leq (2/3)L^{-(d+1)}.
\end{equation*}
Let $\tau' = \inf\left\{t\geq 0 : W_t \notin C_{L'}\right\}$ and
$B_2= \{\tau'\vee\tau_{L''} \leq L^{5/2}\}$. We claim that 
\begin{equation}
\label{eq:app-kernelest-technical-eq4}
\mathbb{Q}\left({B}^c_2\right) \leq
(1/3)L^{-(d+1)}.
\end{equation}
By the central limit theorem, one finds a constant $c>0$ with
$\mathbb{Q}\left(\tau_{L''} \leq {(L'')}^2\right)\geq c$ for $L$ large.  By
the Markov property, we obtain $\mathbb{Q}\left(\tau_{L''} > L^{5/2}\right)
\leq (1-c)^{L^{1/3}}$. A similar bound holds for the probability
$\mathbb{Q}\left(\tau' > L^{5/2}\right)$, and~\eqref{eq:app-kernelest-technical-eq4} follows. Since $\pibm_L$ is
unchanged if the Brownian motion is replaced by a Brownian motion with
covariance $d^{-1}I_d$, we have
\begin{eqnarray}
\label{eq:app-kernelest-technical-eq5}
  \pibm_{L'}(x,A_{'}^\beta) &\geq& \mathbb{Q}\left(X_{\tau_L}\in A,\, W_{\tau'}\in A_{'}^\beta\right)\\ 
  &\geq& \Prw_x\left(X_{\tau_L}\in A\right) - \mathbb{Q}\left(X_{\tau_L}\in
    A,\, W_{\tau'}\notin A_{'}^\beta,\, B_1\cap B_2\right) - L^{-(d+1)}.\nonumber
\end{eqnarray}
Let $U=\left\{z\in\mathbb{Z}^d : \dist(z,(\partial C_{L'}\backslash A_{'}^\beta))
\leq 5C_1\log L\right\}$. 
Then
\begin{equation*}
\mathbb{Q}\left(X_{\tau_L}\in A,\, W_{\tau'}\notin A_{'}^\beta,\, B_1\cap B_2\right) \leq
\Prw_x\left(X_{\tau_L}\in A,\, T_U < \tau_{L''}\right).
\end{equation*}
By the strong Markov property,
\begin{equation*}
  \Prw_x\left(X_{\tau_L}\in A,\, T_U < \tau_{L''}\right)\leq
  \Prw_x\left(X_{\tau_L}\in A\right)\sup_{y\in A}\Prw_y\left( T_U < \tau_{L''}\right).
\end{equation*}
Further, there exists a constant $c>0$ such that for $y\in A$ and $z\in U$,
we have $|y-z| \geq cL^{\beta}$ and $\dist_{L''}(z) \leq\dist_{L''}(y) \leq
2L^\eta$. Therefore, an application of first Lemma~\ref{def:hittingprob}
(ii) and then Lemma~\ref{def:hittingprob-technical} yields
\begin{equation*}
  \Prw_y\left(T_U < \tau_{L''}\right) \leq C L^{2\eta}\sum_{z\in
    U}\frac{1}{{|y-z|}^d} \leq  CL^{2\eta}(\log L) L^{-\beta}\leq CL^{-(\beta-3\eta)},
\end{equation*}
uniformly in $y\in A$. Going back
to~\eqref{eq:app-kernelest-technical-eq5}, we arrive at
\begin{equation*}
  \pibm_{L'}\left(x,A_{'}^\beta\right) \geq \pi_L(x,A)\left(1-CL^{-(\beta-3\eta)}\right)-L^{-(d+1)}.
\end{equation*}
Together with~\eqref{eq:app-kernelest-technical-eq1}, this shows (i).\\
(ii) The ideas are the same as in (i), so we only sketch the proof. Set
$L'=L-L^\eta$, $L'' = L+L^\eta$. Denote by $A_{'}$ the image of $A$ on
$\partial{C}_{L'}$ under $y\mapsto (L'/L)y$. Similar to~\eqref{eq:app-kernelest-technical-eq1}, one finds
\begin{equation*}
  \pibm_L\left(x,A\right) \leq\pibm_{L'}(x,A_{'})\left(1+C(\beta,\eta)L^{-(\beta-\eta)}\right).
\end{equation*}
With $B_1$, $B_2$, $\tau'$ and $W_t$, $\mathbb{Q}$ defined as above,
$\Pbm_x$ the law of $W_t$ conditioned on $W_0=x$,
\begin{equation*}
\pi_L(x,A^{\beta}) \geq \Pbm_x\left(W_{\tau'}\in A_{'}\right) - \mathbb{Q}\left(W_{\tau'}\in A_{'},\,X_{\tau_L}\notin
A^{\beta},\,B_1\cap B_2\right) - L^{-(d+1)}.
\end{equation*}
Then, with $U=\left\{z\in\mathbb{R}^d : \dist(z,(\partial C_{L}\backslash A^{\beta}))
\leq 5C_1\log L\right\}$, $\tau'' = \inf\left\{t\geq 0 : W_t \notin C_{L''}\right\}$, 
\begin{equation*}
\mathbb{Q}\left(W_{\tau'}\in A_{'},\,X_{\tau_L}\notin
  A^{\beta},\, B_1\cap B_2\right)\leq \Pbm_x\left(W_{\tau'}\in
  A_{'}\right)\sup_{y\in A_{'}}\Pbm_y\left(T_U<\tau''\right).
\end{equation*}
Using the hitting estimates for Brownian motion from
Lemma~\ref{def:hittingprob-bm}, one obtains for $y\in A_{'}$ 
\begin{equation*}
\Pbm_y\left(T_U<\tau''\right)\leq CL^{-(\beta-3\eta)}.
\end{equation*}
Altogether, (ii) follows.
\end{prooof}

We write $\phbm_{\psi}(x,z)$ for the density of
$\phbm_{\psi}(x,dz)$ with respect to $d$-dimensional Lebesgue measure,
i.e. for $\psi = (m_x)_{x\in\mathbb{R}^d}$, 
\begin{equation}
\label{eq:app-cgpsibmdens}
\phbm_{\psi}(x,z)=\frac{1}{m_x}\varphi\left(\frac{|z-x|}{m_x}\right)\pibm_{C_{|z-x|}}(0,z-x).
 \end{equation}
\begin{lemma}
\label{def:app-kernelest}
There exists a constant $C> 0$ such that 
for large $L$, $\psi=(m_y)\in\mathcal{M}_L$, $x,x'\in\{(2/3)L\leq |y|\leq
(3/2)L\}\cap\mathbb{Z}^d$ and any $z,z'\in\mathbb{Z}^d$, 
\begin{enumerate}
\item $\ph_{\psi}(x,z) \leq CL^{-d}$. 
\item $\phbm_{\psi}(x,z)\leq CL^{-d}$.
\item $|\ph_{\psi}(x,z) - \ph_{\psi}(x',z)| \leq
  C|x-x'|L^{-(d+1)}\log L$.
\item $|\ph_{\psi}(x,z) - \ph_{\psi}(x,z')| \leq C|z-z'|L^{-(d+1)}\log L$.
\item $|\phbm_{\psi}(x,z) - \phbm_{\psi}(x',z)| \leq C|x-x'|L^{-(d+1)}$.
\item $|\phbm_{\psi}(x,z) - \phbm_{\psi}(x,z')|
  \leq C|z-z'|L^{-(d+1)}$.
\item $|\ph_{\psi}(x,z)-\phbm_{\psi}(x,z)| \leq L^{-(d+1/4)}$.
\end{enumerate}
\end{lemma}
\begin{corollary}
\label{def:app-kernelest-cor}
In the situation of the preceding lemma,
 \begin{enumerate}
 \item
 \begin{eqnarray*}
    \lefteqn{|\ph_{\psi}(x,z) - \ph_{\psi}(x',z)|}\\ 
    &\leq& C\min\left\{|x-x'|L^{-(d+1)}\log L,\,
      |x-x'|L^{-(d+1)}+
      L^{-(d+1/4)}\right\}.
  \end{eqnarray*}
\item
  \begin{eqnarray*}
    \lefteqn{|\ph_{\psi}(x,z) - \ph_{\psi}(x,z')|}\\ 
    &\leq&
    C\min\left\{|z-z'|L^{-(d+1)}\log L,\,|z-z'|L^{-(d+1)}+
      L^{-(d+1/4)}\right\}.
  \end{eqnarray*}
\end{enumerate}
\end{corollary}
\begin{proof}
Combine (iii)-(vii).
\end{proof}
\begin{remark}
The condition on $x$ and $x'$ in the lemma is only to ensure that both
points lie in the domain of $\psi$.
\end{remark}
\begin{prooof2}{\bf of Lemma~\ref{def:app-kernelest}:}
  (i), (ii) This follows from the definition of $\ph_{\psi}$, $\phbm_{\psi}$ together with
  Lemma~\ref{def:lemmalawler} (i) and the explicit form of the Poisson
  kernel~\eqref{eq:est-poissonkernel}, respectively. \\
  (iii), (iv) We can restrict ourselves to the case
  $|x-x'| = 1$ as otherwise we take a shortest path connecting $x$ with $x'$
  inside $\{(2/3)L\leq |y|\leq (3/2)L\}$ and apply the result for distance $1$
  $O\left(|x-x'|\right)$ times. We have
  \begin{eqnarray*}
    \lefteqn{\ph_{\psi}(x,z)-\ph_{\psi}(x',z)}\\
    &=& \left(1-\frac{m_x}{m_{x'}}\right)\ph_{\psi}(x,z) +
    \frac{1}{m_{x'}}\int_{\mathbb{R}_+}\left(\varphi\left(\frac{t}{m_x}\right)
      -\varphi\left(\frac{t}{m_{x'}}\right)\right)\pi_{V_t(x)}(x,z)\dt t\\
    && +\ \frac{1}{m_{x'}}\int_{\mathbb{R}_+}\varphi\left(\frac{t}{m_{x'}}\right)
    \left(\pi_{V_t(x)}(x,z)-\pi_{V_t(x')}(x',z)\right)\dt t\\
    &=& I_1 + I_2 + I_3.
  \end{eqnarray*}
  Using the fact that $\psi\in\mathcal{M}_L$ and part (i) for
  $\ph_{\psi}(x,z)$, it follows that $|I_1| \leq CL^{-(d+1)}$. Using
  additionally the smoothness of $\varphi$ and, by
  Lemma~\ref{def:lemmalawler} (i), $|\pi_{V_t(x)}(x,z)| \leq CL^{-(d-1)}$,
  we also have $|I_2| \leq CL^{-(d+1)}$. It remains to handle $I_3$. By
  translation invariance of simple random walk, $\pi_{V_t(x)}(x,z) =
  \pi_{V_t}(0,z-x)$. In particular, both (iii) and (iv) will follow if we
  prove that
  \begin{equation}
    \label{eq:a4-1}
    \left|\int_{\mathbb{R}_+}\varphi\left(\frac{t}{m_{x}}\right)\left(\pi_{V_t}(0,z-x)-\pi_{V_t}(0,z-x')\right)\dt
      t \right|
    \leq C L^{-d}\log L
  \end{equation}
  for $x,x'$ with $|x-x'| = 1$. By definition of $\mathcal{M}_L$, $m_x\in
  (L/10,5L)$. We may therefore assume that $L/10 < |y-z| < 10L$ for
  $y=x,x'$. Due to the smoothness of $\varphi$ and the fact that the
  integral is over an interval of length at most $2$,~\eqref{eq:a4-1} will
  follow if we show
  \begin{equation*}
    \left|\int_{L/10}^{10L}\left(\pi_{V_t}(0,z-x)-\pi_{V_t}(0,z-x')\right)\dt
      t\right|
    \leq C L^{-d}\log L.
  \end{equation*} 
  We set $J = \{t > 0 : z-x\in \partial V_t\}$ and $J' = \{t>0 :
  z-x'\in\partial V_{t'}\}$, where
  \begin{equation*}
    t' = t'(t) = \left|t\frac{(z-x)}{|z-x|}-(x'-x)\right|.
  \end{equation*}
  $J$ is an interval of length at most $1$, and $J'$ has the same length up
  to order $O(L^{-1})$. Furthermore, $|J\Delta J'|$ is of order
  $O(L^{-1})$, and $\left|\frac{\textup{d}}{\textup{d}t}t'\right| = 1 +
  O(L^{-1})$. Using that both $\pi_{V_t}(0,z-x)$ and $\pi_{V_t}(0,z-x')$
  are of order $O(L^{-(d-1)})$, it therefore suffices to prove
  \begin{equation}
    \label{eq:a4-2}
    \left|\int_{J\cap
        J'}\left(\pi_{V_t(x)}(x,z)-\pi_{V_{t'}(x')}(x',z)\right)\dt t\right| \leq CL^{-d}\log L.
  \end{equation}
  Write $V$ for $V_t(x)$ and $V'$ for $V_{t'}(x')$. By a first exit
  decomposition,
  \begin{equation*}
    \pi_V(x,z) \leq \pi_{V'}(x,z) + \sum_{y\in V\backslash V'}\Prw_x\left(T_y
      < \tau_V\right)\pi_V(y,z).   
  \end{equation*}
  By Lemma~\ref{def:lemmalawler} (ii), we can replace $\pi_{V'}(x,z)$ by
  $\pi_{V'}(x',z) + O(L^{-d})$. For $y\in V\backslash V'$ we have by
  Lemma~\ref{def:hittingprob} (ii) $\pi_V(y,z) = O(|y-z|^{-d})$ and
  $\Prw_x\left(T_y < \tau_V\right) = O(L^{-(d-1)})$, uniformly in $t
  \in J\cap J'$. Further, using $|x-x'| = 1$, we have with $r= |z-x|$
  \begin{equation*}
    \bigcup\limits_{t\in J\cap J'}(V\backslash V') \subset V_r(x)\backslash
    V_{r-2}(x') \subset x + \sh_r(3),
  \end{equation*}
  and for any $y\in \sh_r(3)$, it follows by a geometric consideration that
  \begin{equation*}
    \int_{J\cap J'}1_{\{y\in V\backslash V'\}}\dt t \leq C\frac{|y-z|}{L}.
  \end{equation*}
  Altogether, applying Lemma~\ref{def:hittingprob-technical} in the last
  step,
\begin{eqnarray*}
  \lefteqn{\int_{J\cap J'}\pi_{V}(x,z)\dt t}\\ 
  &\leq& 
  \int_{J\cap J'}\pi_{V'}(x',z)\dt t + O(L^{-d}) + CL^{-(d-1)}\sum_{y\in x +
    \sh_r(3)}\frac{1}{|y-z|^{d}}\frac{|y-z|}{L}\\ 
  &\leq& \int_{J\cap J'}\pi_{V'}(x',z)\dt t +  CL^{-d}\log L.
\end{eqnarray*}
The reverse inequality, proved in the same way, then implies~\eqref{eq:a4-2}.\\
(v) We can assume $|x-x'| \leq 1$. Then the claim follows from 
  \begin{eqnarray*}
    \lefteqn{\left|\phbm_{\psi}(x,z) - \phbm_{\psi}(x',z)\right|}\\
    &=&
    \frac{1}{d\,alpha}\left|\frac{1}{m_x}\varphi\left(\frac{|z-x|}{m_x}\right)\frac{1}{|z-x|^{d-1}}-
      \frac{1}{m_{x'}}\varphi\left(\frac{|z-x'|}{m_{x'}}\right)\frac{1}{|z-x'|^{d-1}}\right|. 
  \end{eqnarray*}
  (vi) This is proved in the same way as (v).\\
  (vii) Fix $\alpha = 2/3$, $\beta = 1/3$, and let
  $0<\eta<1/40$. Set $A = C_{L^\alpha}(z)$ and $A^{\mathbb{Z}}= A\cap
  \mathbb{Z}^d$. By part (iv), we have
  \begin{equation}
    \label{eq:a7-1}
    \ph_{\psi}(x,z) \leq \frac{1}{|A^{\mathbb{Z}}|}\ph_{\psi}\left(x,A^{\mathbb{Z}}\right)+
    CL^{-(d+1-\alpha)}\log L.
  \end{equation}
  Further,
  \begin{equation}
    \label{eq:a7-2}
    \ph_{\psi}\left(x,A^{\mathbb{Z}}\right) =
    \frac{1}{m_x}\int_{L/10}^{10L}
    \varphi\left(\frac{t}{m_x}\right)\pi_{V_t(x)}\left(x,A^{\mathbb{Z}}\right)\dt
    t. 
  \end{equation}
  By Lemma~\ref{def:app-kernelest-technical} (i), it follows that for $t
  \in (L/10,10L)$
  \begin{equation*}
    \pi_{V_t(x)}\left(x,A^{\mathbb{Z}}\right) \leq
    \pibm_{V_t(x)}\left(x,A^\beta\right)\left(1+CL^{-(\beta-3\eta)}\right)
    + CL^{-(d+1)}, 
  \end{equation*}
  where $A^\beta = C_{L^\alpha + L^\beta}(z)$ and the constant $C$ is
  uniform in $t$. If we plug the last line into~\eqref{eq:a7-2} and use
  part (ii) and (vi), we arrive at
  \begin{eqnarray*}
    \ph_\psi\left(x,A^{\mathbb{Z}}\right) &\leq&
    \phbm_\psi(x,A^\beta)\left(1+CL^{-(\beta-3\eta)}\right) + CL^{-(d+1)}\\
    &\leq& \phbm_\psi(x,A)\left(1+CL^{-(\beta-3\eta)}\right) +
    CL^{-d}L^{(d-1)\alpha+\beta}\\
    &\leq & |A|\cdot\phbm_\psi(x,z) + CL^{d\alpha}L^{-(d+\beta-3\eta)}.
  \end{eqnarray*}
  Notice that in our notation, $|A|$ is the volume of $A$, while $|A^{\mathbb{Z}}|$ is
  the cardinality of $A^{\mathbb{Z}}$. From {\it Gauss} we have learned that $|A| =
  |A^{\mathbb{Z}}| + O\left(L^{(d-1)\alpha}\right)$.  Going back to~\eqref{eq:a7-1},
  this implies
  \begin{equation*}
    \ph_{\psi}(x,z) \leq \phbm_{\psi}(x,z) + L^{-(d+1/4)},
  \end{equation*}
  as claimed. To prove the reverse inequality, we can follow the same steps,
  replacing the random walk estimates by those of Brownian motion and vice
  versa.
\end{prooof2}

\subsection{Proof of Lemma~\ref{def:est-phi}}
\begin{prooof2}{\bf of Lemma~\ref{def:est-phi}:}
  (i) Set $\alpha = 2/3$, $\beta= 1/3$ and $\eta =
  \dist(x,\partial V_L)$. Choose $y_1\in\partial V_L$ such that $|x-y_1| =
  \eta$. First assume $\eta \leq L^{\beta}$. The following estimates are
  valid for $L$ large. Write
  \begin{equation*}
    \phi_{L,\psi}(x,z) = \sum_{y\in \partial V_L:\atop |y-y_1| \leq
      L^\alpha}\pi_L(x,y)\ph_{\psi}(y,z) + \sum_{y\in \partial V_L:\atop |y-y_1|>
      L^\alpha}\pi_L(x,y)\ph_{\psi}(y,z)  = I_1 + I_2.
  \end{equation*}
  For $I_2$, notice that $|y-y_1| > L^\alpha$ implies $|y-x| >
  L^{\alpha}/2$. Using Lemmata~\ref{def:app-kernelest} (i),
  \ref{def:hittingprob} (iii) in the first and
  Lemma~\ref{def:hittingprob-technical} in the second inequality, we have
  \begin{equation}
    \label{eq:est-phi-eq-1-1}
    I_2 \leq C\eta L^{-d}\sum_{y\in\partial V_L:\atop|y-y_1| >
      L^\alpha}\frac{1}{|x-y|^d} \leq C\eta L^{-(d+\alpha)} \leq L^{-(d+1/4)}.
  \end{equation}
  For $I_1$, we first use Lemma~\ref{def:app-kernelest} part (iii) to
  deduce
  \begin{equation*}
    \ph_{\psi}(y,z) \leq \ph_{\psi}(y_1,z) + CL^{-(d+1-\alpha)}\log L.
  \end{equation*}
  Therefore by part (vii),
  \begin{equation*}
    I_1 \leq \ph_{\psi}(y_1,z) + L^{-(d+1/4)} \leq \phbm_{\psi}(y_1,z) + 2L^{-(d+1/4)}.
  \end{equation*}
  From the Poisson
  formula~\eqref{eq:est-poissonkernel} we deduce much as
  in~\eqref{eq:est-phi-eq-1-1} that
  \begin{equation*}
    \int\limits_{y\in\partial C_L: |y-y_1| > L^\alpha}\pibm_L(x,dy) \leq L^{-1/4}.
  \end{equation*}
  Using Lemma~\ref{def:app-kernelest} (ii) in the first and (v) in the
  second inequality, we conclude that
  \begin{eqnarray*}
    \phbm_{\psi}(y_1,z) &\leq& \phbm_{\psi}(y_1,z) \int\limits_{y\in\partial C_L:
      |y-y_1| \leq L^\alpha}\pibm_L(x,dy) + CL^{-(d+1/4)}\\ 
    &\leq& \int\limits_{y\in\partial C_L:
      |y-y_1| \leq L^\alpha}\pibm_L(x,dy)\phbm_{\psi}(y,z) +
    CL^{-(d+1/4)}\\
    &\leq& \phibm_{L,\psi}(x,z) + CL^{-(d+1/4)}.
  \end{eqnarray*}
  Now we look at the case $\eta > L^\beta$. We take a cube $U_1$ of radius
  $L^\alpha$, centered at $y_1$, and set $W_1 = \partial V_L\cap U_1$.
  Then we can find a partition of $\partial V_L\backslash W_1$ into
  disjoint sets $W_i = \partial V_L\cap U_i$, $i=2,\ldots,k_L$, where $U_i$
  is a cube such that for some $c_1,c_2> 0$ depending only on $d$, 
  \begin{equation*}
    c_1L^{\alpha(d-1)}\leq |W_i| \leq c_2L^{\alpha(d-1)}. 
  \end{equation*}
  For $i \geq 2$, we fix an arbitrary $y_i\in W_i$.  Let $W_i^\beta =
  \{y\in\mathbb{R}^d : \dist(y,W_i) \leq L^\beta\}$.  Applying first
  Lemma~\ref{def:app-kernelest} (iii) and then
  Lemma~\ref{def:app-kernelest-technical} (i) gives
  \begin{eqnarray}
    \label{eq:est-phi-eq-1-2}
    \phi_{L,\psi}(x,z) &\leq& \sum_{i=1}^{k_L}\pi_L(x,W_i)\ph_{\psi}(y_i,z)+ L^{-(d+1/4)}\nonumber\\
    &\leq&
    \sum_{i=1}^{k_L}\pibm_L(x,W_i^\beta)\ph_{\psi}(y_i,z)\left(1+L^{-1/4}\right)
    +L^{-(d+1/4)}.
  \end{eqnarray}
  As the $W_i^\beta$ overlap, we refine them as follows: Set $\tilde{W}_1
  = W_1^\beta\cap \partial C_L$, and split $\partial C_L\backslash
  \tilde{W}_1$ into a collection of disjoint measurable sets $\tilde{W}_i
  \subset \partial C_L\cap W_i^\beta$, $i=2,\ldots,k_L$, such that
  $\cup_{i=1}^{k_L} \tilde{W}_i = \partial C_L$ and
  $|(W_i^\beta\cap\partial C_L)\backslash \tilde{W_i}| \leq
  C_1L^{\alpha(d-2)+\beta}$ for some $C_1=C_1(d)$. By construction we can
  find constants $c_3, c_4 > 0$ such that $|\tilde{W}_i| \geq
  c_3L^{\alpha(d-1)}$ and, for $i=2,\ldots,k_L$,
  \begin{equation*}
    \inf\limits_{y\in W_i^\beta}|x-y| \geq c_4\sup\limits_{y\in \tilde{W}_i}|x-y|,
  \end{equation*}
  which implies by~\eqref{eq:est-poissonkernel} that
  \begin{equation*}
    \sup\limits_{y\in W_i^\beta}\pibm_L(x,y) \leq c_4^{-1}\inf\limits_{y\in
      \tilde{W}_i}\pibm_L(x,y).
  \end{equation*}
  For $i=1,\ldots, k_L$ we then have
  \begin{equation*}
    \pibm_L(x,W_i^\beta) \leq
    \pibm_L(x,\tilde{W}_i)\left(1+C_1c_3^{-1}L^{\beta-\alpha}\right)\leq \pibm_L(x,\tilde{W}_i)\left(1+L^{-1/4}\right). 
  \end{equation*}
  Plugging the last line into~\eqref{eq:est-phi-eq-1-2},
  \begin{equation*}
    \phi_{L,\psi}(x,z) \leq \sum_{i=1}^{k_L}\pibm_L(x,\tilde{W}_i)\ph_{\psi}(y_i,z)\left(1+L^{-1/4}\right)+ L^{-(d+1/4)}.
  \end{equation*}
  A reapplication of Lemma~\ref{def:app-kernelest} (iii), (vii) and then
  (ii) yields
  \begin{eqnarray*}
    \phi_{L,\psi}(x,z) &\leq& \sum_{i=1}^{k_L}\int_{\tilde{W}_i}\pibm_L(x,dy)\ph_{\psi}(y,z)+ L^{-(d+1/4)}\nonumber\\
    &\leq& \sum_{i=1}^{k_L}\int_{\tilde{W}_i}\pibm_L(x,dy)\phbm_{\psi}(y,z)\left(1+L^{-1/4}\right)
    +L^{-(d+1/4)}\\
    &=& \phibm_{L,\psi}(x,z) + CL^{-(d+1/4)}.
  \end{eqnarray*}
  The reverse inequality in both the cases $\eta \leq L^\beta$ and $\eta >
  L^{\beta}$ is obtained similarly.\\
  (ii) Let $\psi = (m_y)_y\in\mathcal{M}_L$ and $z\in\mathbb{Z}^d$. For
  $y\in\mathbb{R}^d$ with $L/2 < |y| < 2L$ we set
  \begin{equation}
    \label{eq:app-phidens} 
    g(y,z) = 
    \frac{1}{m_y}\varphi\left(\frac{|z-y|}{m_y}\right)\pibm_{C_{|z-y|}}(0,z-y).
  \end{equation}
  Then
  \begin{equation*}
    \phibm_{L,\psi}(x,z) = \int_{\partial C_L}\pibm_L(x,dy)g(y,z).
  \end{equation*}
  Choose a cutoff function
  $\chi\in C^{\infty}\left(\mathbb{R}^d\right)$ with compact support in
  $\{x\in\mathbb{R}^d : 1/2 < |x| < 2\}$ such that $\chi \equiv 1$ on
  $\{2/3\leq|x|\leq 3/2\}$. Setting $m_v = 1$ for $v\notin \{L/2 < |x| <
  2L\}$, we define 
  \begin{equation*}
    \tilde{g}(y,z) = g(Ly,z)\chi(y),\ \ y\in\mathbb{R}^d.
  \end{equation*}
  By~\eqref{eq:est-poissonkernel} we have the representation
  \begin{equation*}
    \tilde{g}(y,z) =
    \frac{1}{d\,\alpha\, m_{Ly}}{|z-Ly|}^{-d+1}\varphi\left(\frac{|z-Ly|}{m_{Ly}}\right)\chi(y). 
  \end{equation*}
  Notice that $\tilde{g}(\cdot,z)\in C^4\left(\mathbb{R}^d\right)$, with
  $\tilde{g}(y,z) = 0$ if $|z-Ly| \notin (L/5,10L)$ or
  $|y|\notin\left(1/2,2\right)$. 
 The Poisson integral $u(\overline{x},z)
    = \phibm_{L,\psi}(x,z),\ x = L\overline{x}$, solves the Dirichlet problem
  \begin{equation}
  \label{eq:app-DP}
    \left\{\begin{array}{r@{\:=\:}l@{,\quad}l}
        \Delta_{\overline{x}}u(\overline{x},z) & 0 &\overline{x}\in
        C_1\\
        u(\overline{x},z) &\tilde{g}(\overline{x}, z)
        &\overline{x}\in\partial C_1
      \end{array}\right..
  \end{equation}
 where $\Delta_{\overline{x}}$ is the Laplace operator with respect
 to $\overline{x}$.
  Moreover, by Corollary 6.5.4 of Krylov~\cite{KRY}, $u(\cdot,z)$ is smooth on
  $\overline{C}_1$. Write
  \begin{equation*} {|u(\cdot,z)|}_k =
    \sum_{i=0}^{k}{\left|\left|D^iu(\cdot,z)\right|\right|}_{C_1}.
  \end{equation*}
  Theorem 6.3.2 in the same book shows that for
  some $C>0$ independent of $z$
  \begin{equation*} {|u(\cdot,z)|}_3 \leq C{|\tilde{g}(\cdot,z)|}_4.
  \end{equation*}
  A direct calculation shows that $\sup_{z\in\mathbb{R}^d}{|\tilde{g}(\cdot,z)|}_4 \leq
  CL^{-d}$. Now the claim follows from
  \begin{equation*}
    {\left|\left|D^i\phibm_{L,\psi}(\cdot,z)\right|\right|}_{C_L} =
    L^{-i}{\left|\left|D^iu(\cdot,z)\right|\right|}_{C_1}.  
  \end{equation*}
  (iii) Let $x, x'\in V_L\cup\partial V_L$. Choose $\tilde{x}\in V_L$ next to $x$ and $\tilde{x}'\in V_L$
  next to $x'$. Then $|\tilde{x}-x| = 1$ if $x\in\partial{V}_L$ and
  $\tilde{x}=x$ otherwise. By the triangle inequality,
  \begin{eqnarray}
    \label{eq:app-est-phi-eq1}
    \lefteqn{\left |\phi_{L,\psi}(x,z) - \phi_{L,\psi}(x',z)\right|}\nonumber\\ 
    &\leq& \left|\phi_{L,\psi}(x,z) - \phi_{L,\psi}(\tilde{x},z)\right|
    +\left|\phi_{L,\psi}(\tilde{x},z) - \phi_{L,\psi}(\tilde{x}',z)\right|
    +\left|\phi_{L,\psi}(\tilde{x}',z) - \phi_{L,\psi}(x',z)\right|.\nonumber\\
  \end{eqnarray}
  By parts (i) and (ii) combined with the mean value theorem, we get for the
  middle term
  \begin{eqnarray*}
    \lefteqn{\left|\phi_{L,\psi}(\tilde{x},z) - \phi_{L,\psi}(\tilde{x}',z)\right|}\\
    &\leq & \left|\phi_{L,\psi}(\tilde{x},z) - \phibm_{L,\psi}(\tilde{x},z)\right| +
    \left|\phibm_{L,\psi}(\tilde{x},z) - \phibm_{L,\psi}(\tilde{x}',z)\right| +
    \left|\phibm_{L,\psi}(\tilde{x}',z) - \phi_{L,\psi}(\tilde{x}',z)\right|\\
    &\leq & C\left (L^{-(d+1/4)} + |x-x'|L^{-(d+1)}\right).
  \end{eqnarray*}
  If $x\in\partial{V_L}$, then $\phi_{L,\psi}(x,z) = \ph_{\psi}(x,z)$, so
  that we can write the first term of~\eqref{eq:app-est-phi-eq1} as
  \begin{equation*}
    \left|\phi_{L,\psi}(x,z) - \phi_{L,\psi}(\tilde{x},z)\right| =
    \left|\sum_{y\in\partial V_L}\pi_L(\tilde{x},y)\left(\ph_{\psi}(y,z) - \ph_{\psi}(x,z)\right)\right|.
  \end{equation*}
  Set $A = \{y\in\partial V_L : |x-y| > L^{1/4}\}$. Then by
  Lemmata~\ref{def:hittingprob} (iii) and~\ref{def:hittingprob-technical},
  \begin{equation*}
    \pi_L(\tilde{x},A) \leq C\sum_{y\in A}\frac{1}{|x-y|^d} \leq CL^{-1/4}.
  \end{equation*}
  For all $y\in \partial V_L$, we have by Lemma~\ref{def:app-kernelest} (i)
  that $|\ph_{\psi}(y,z)-\ph_{\psi}(x,z)| \leq CL^{-d}$. If $y\in\partial
  V_L\backslash A$, then part (iii) gives
  $|\ph_{\psi}(y,z)-\ph_{\psi}(x,z)| \leq CL^{-(d+3/4)}\log L$. Altogether,
  \begin{equation*}
    \left|\phi_{L,\psi}(x,z) - \phi_{L,\psi}(\tilde{x},z)\right| \leq CL^{-(d+1/4)}.
  \end{equation*}
  The third term of~\eqref{eq:app-est-phi-eq1} is treated in exactly the
  same way.
\end{prooof2}

\subsection{Proof of Lemma~\ref{def:hittingprob}}
We start with an auxiliary lemma, which already includes the upper bound of
part (iii). 
\begin{lemma}
  \label{def:lemma1}
 Let $x \in V_L$, $y \in \partial V_L$, and set $t = |x-y|$. 
 \begin{enumerate}
  \item
    \begin{equation*}
      \Prw_x\left(X_{\tau_L} = y\right) \leq C\:{\dL(x)}^{-d+1}.
    \end{equation*}
  \item 
    \begin{equation*}
      \Prw_x(X_{\tau_L} = y) \leq C\frac{\max\{1,\dL(x)\}}{|x-y|}\max_{x'\in\partial
        V_{t/3}(y)\cap V_L}\Prw_{x'}(X_{\tau_L} = y).
    \end{equation*}
  \item 
    \begin{equation*}
      \Prw_x(X_{\tau_L} = y) \leq C\frac{\max\{1,\dL(x)\}}{|x-y|^d}.
    \end{equation*}
  \end{enumerate}
\end{lemma}
\begin{prooof}
  (i) We can assume that $s = \dL(x) \geq 6$. If $s' = \lfloor s/3\rfloor$, then $\partial V_{s'}(x) \subset V_{L-s'}$. Using
  Lemma~\ref{def:lemmalawler} (iii), we
  compute for any $y' \in V_L$ with $|y-y'| = 1$,
  \begin{equation*}
    \Prw_{y'}\left(T_{\partial V_{s'}(x)} < \tau_L\right) \leq
    \Prw_{y'}\left(T_{V_{L-s'}} < \tau_L\right) \leq Cs^{-1}. 
  \end{equation*}
  By Lemma~\ref{def:hittingprob} (i) it follows that uniformly in
  $z\in\partial{V}_{s'}(x)$,
  \begin{equation*}
    \Prw_z\left(T_x < \tau_L\right) \leq \Prw_z\left(T_x < \infty\right) \leq C
    (s')^{-d+2}\leq C s^{-d+2}.
  \end{equation*}
  Thus, by the strong Markov property at $T_{\partial V_{s'}(x)}$,
  \begin{equation*}
    \Prw_{y'}\left(T_x < \tau_L\right) \leq Cs^{-d+1}.
  \end{equation*}
  Since by time reversibility of simple random walk
  \begin{equation*}
    \Prw_x\left(X_{\tau_L} = y\right) = \sum_{y'\in V_L,\atop |y'-y| =
      1}\Prw_x\left(X_{\tau_L} = y,\, X_{\tau_L-1} = y'\right) =
    \frac{1}{2d}\sum_{y'\in V_L,\atop |y'-y| = 1} \Prw_{y'}\left(T_x <
      \tau_L\right), 
  \end{equation*}  
  the claim is proved.\\
  (ii) We may assume that $t=|x-y| >100d$ and $\dL(x) <
   t/100$. Choose a point $x'$ outside $V_L$ such that $V_{t/10}(x')\cap V_L
  =\emptyset$ and $|x-x'| \leq \dL(x)+t/10 + \sqrt{d}$. Then $|x-x'| \leq
  t/5$. Furthermore, since $|x'-y|\geq 4t/5$, 
\begin{equation*}
     \left(V_{t/4}(x')\cup\partial V_{t/4}(x')\right) \cap V_{t/3}(y) = \emptyset.
   \end{equation*}
We apply twice the strong Markov property and obtain
   \begin{equation*}
     \Prw_x(X_{\tau_L} = y) \leq \Prw_x\left(\tau_{V_{t/4}(x')} <
      T_{V_{t/10}(x')}\right) \max_{z\in\partial V_{t/3}(y)\cap V_L}\Prw_z(X_{\tau_L}=y).
   \end{equation*}
Evaluating the expression in Lemma~\ref{def:lemmalawler} (iii) shows
\begin{equation*}
\Prw_x\left(\tau_{V_{t/4}(x')} < T_{V_{t/10}(x')}\right)\leq C\frac{\max\{1,\dL(x)\}}{t},
\end{equation*}
which concludes the proof of part (ii).\\
  (iii) By (ii) it suffices to prove that for some constant $K$ and for
  all $l\geq 1$
  \begin{equation}
    \label{eq:eq1}
    \max_{z\in\partial V_{l/3}(y)\cap V_L}\Prw_z(X_{\tau_L}=y) \leq Kl^{-d+1}.
  \end{equation}
  Let $c_1$ and $c_2$ be the constants from (i) and (ii), respectively.
  Define $\eta = 3^{-d}c_2^{-1}$ and $K = \max\{3^{d(d-1)}c_2^{d-1},
  c_1\eta^{-d+1}\}$.  For $l \leq 3^{d}c_2$ there is nothing to prove since
  $Kl^{-d+1} \geq 1$.  Thus let $l > 3^{d}c_2$, and choose $l_0$ with $l_0
  < l \leq 2l_0$. Assume that~\eqref{eq:eq1} is proved for all $l' \leq l_0$. We show
  that~\eqref{eq:eq1} also holds for $l$. For $z$ with $\dL(z) \geq \eta
  l$, it follows from (i) that
  \begin{equation*}
    \Prw_z(X_{\tau_L}=y) \leq c_1\eta^{-d+1}l^{-d+1} \leq Kl^{-d+1}.
  \end{equation*}
  If $1 \leq \dL(z) < \eta l$, then by (ii) and the fact that $l/3 \leq
  l_0$
  \begin{eqnarray*}
    \Prw_z(X_{\tau_L}=y) &\leq&
    c_2\frac{\max\{1,\dL(z)\}}{|z-y|}\max_{z'\in\partial V_{t/9}(y)\cap
      V_L}\Prw_z(X_{\tau_L}=y)\nonumber\\ 
    &\leq& c_23\eta K \left(l/3\right)^{-d+1} \leq K l^{-d+1}.
  \end{eqnarray*}
  If $\dL(z) < 1$, then again by (i)
  \begin{equation*}
    \Prw_z(X_{\tau_L}=y) \leq c_23l^{-1} K \left(l/3\right)^{-d+1} \leq Kl^{-d+1}.
  \end{equation*}
  This proves the claim.
\end{prooof}
\begin{prooof2}{\bf of Lemma~\ref{def:hittingprob}:} 
(i) follows from Proposition 6.4.2 of~\cite{LawLim}.\\
(ii) We consider different cases. If $|x-y| \leq \dL(y)/2$, then $\dL(x) \geq
  \dL(y)/2$ and thus by Lemma~\ref{def:hittingprob} (i)
  \begin{equation*}
    \Prw_x\left(T_{V_a(y)} < \tau_L\right) \leq \Prw_x\left(T_{V_a(y)}<\infty\right) \leq
    C{\left(\frac{a}{|x-y|}\right)}^{d-2}\leq C\frac{a^{d-2}\dL(y)\dL(x)}{|x-y|^d}.
  \end{equation*}
  For the rest of the proof we assume that $|x-y| > \dL(y)/2$. Set $a' =
  \dL(y)/5$. First we argue that in the case $1\leq a\leq a'$, we only have
  to prove the bound for $a'$. Indeed, if $\dL(y)/6 \leq a < a'$, we get an
  upper bound by replacing $a$ by $a'$. For $1 \leq a < \dL(y)/6$, the
  strong Markov property together with Lemma~\ref{def:hittingprob} (i)
  yields
  \begin{eqnarray*}
    \Prw_x\left(T_{V_a(y)} <
      \tau_L\right)&\leq&\max_{z\in\partial(\mathbb{Z}^d\backslash V_{a'}(y))}\Prw_z\left(T_{V_a(y)} <
      \tau_L\right)\Prw_x\left(T_{V_{a'}(y)} < \tau_L\right)\nonumber\\
    &\leq& C{\left(\frac{a}{a'-1}\right)}^{d-2}\:\frac{{(a')}^{d-2}\dL(y)\max\{1,\dL(x)\}}{|x-y|^d}\nonumber\\ 
    &\leq & C\frac{a^{d-2}\dL(y)\max\{1,\dL(x)\}}{|x-y|^d}.
  \end{eqnarray*}
  Now we prove the claim for $a = \dL(y)/5$. We take a point $y'\in\partial
  V_L$ closest to $y$. If $|x-z| \geq |x-y|/2$ for all $z\in V_a(y')$, then
  by Lemma~\ref{def:lemma1} (iii)
  \begin{equation*}
    \max_{z\in V_a(y')}\Prw_x\left(X_{\tau_L}=z\right)
    \leq C\,2^d\frac{\max\{1,\dL(x)\}}{|x-y|^d}.
  \end{equation*}
  As a subset of $\mathbb{Z}^d$, $V_{a}(y')\cap\partial V_L$ contains on
  the order of $\dL(y)^{d-1}$ points. Therefore, by
  Lemma~\ref{def:lemmalawler} (i), we deduce that there exists some $\delta >
  0$ such that
  \begin{equation*}
    \min_{x'\in V_a(y)}\Prw_{x'}\left(X_{\tau_L} \in V_{a}(y')\right) \geq \delta.
  \end{equation*}
  We conclude that
  \begin{eqnarray}
    \label{eq:eq2}
    \frac{a^{d-1}\max\{1,\dL(x)\}}{|x-y|^d} &\geq& c\Prw_{x}\left(X_{\tau_L} \in
      V_{a}(y')\right) \geq c\Prw_{x}\left(X_{\tau_L} \in V_{a}(y'),\,
      T_{V_a(y)} < \tau_L\right)\nonumber\\
    &=&c\sum_{x'\in V_a(y)}\Prw_{x}\left(X_{T_{V_a(y)}}  = x',\,
      T_{V_a(y)} < \tau_L\right)\Prw_{x'}\left(X_{\tau_L} \in
      V_a(y')\right)\nonumber\\
    &\geq& c\,\delta\cdot \Prw_{x}\left(T_{V_a(y)} < \tau_L\right).
  \end{eqnarray}
  On the other hand, if $|x-z| < |x-y|/2$ for some $z\in V_a(y')$, then
  \begin{equation*}
    |x-y| \leq |x-z| + |z-y'| + |y'-y| \leq 2\dL(y) + |x-y|/2
  \end{equation*}
  and thus
  \begin{equation}
    \label{eq:eq3}
    \dL(y)/2 < |x-y| \leq 4\dL(y).
  \end{equation}
  If $\dL(x) \geq 4\dL(y)/5$, we use Lemma~\ref{def:hittingprob} (i)
  again. For $\dL(x) < 4\dL(y)/5$, we get by Lemma~\ref{def:lemmalawler}
  (iii)
  \begin{equation*}
    \Prw_x\left(T_{V_a(y)} < \tau_L\right) \leq \Prw_x\left(T_{V_{L-4\dL(y)/5}}<\tau_L\right) \leq
    C\frac{\max\{1,\dL(x)\}}{\dL(y)}.  
  \end{equation*}
  Together with~\eqref{eq:eq3}, this proves the claim in this
  case. Altogether, we have proved the bound for $1\leq a\leq \dL(y)/5$. It
  remains to handle the case $\max\{1,\dL(y)/5\} \leq a$.  If $z\in
  V_{6a}(y)$, we have that
  \begin{equation*}
    |x-y| \leq |x-z| + 6a
  \end{equation*}
  and thus, using $|x-y| > 7a$,
  \begin{equation*}
    |x-y| \leq 7|x-z|. 
  \end{equation*}
  Therefore Lemma~\ref{def:lemma1} (iii) yields
  \begin{equation*}
    \max_{z\in V_{6a}(y)}\Prw_x\left(X_{\tau_L}=z\right)
    \leq C\frac{\max\{1,\dL(x)\}}{|x-z|^d} \leq 7^dC\frac{\max\{1,\dL(x)\}}{|x-y|^d}.
  \end{equation*}
  Again by Lemma~\ref{def:lemmalawler} (i), we find some $\delta > 0$ such that
  \begin{equation*}
    \min_{x'\in V_a(y)}\Prw_{x'}\left(X_{\tau_L} \in V_{6a}(y)\right) \geq \delta.
  \end{equation*}
  A similar argument to~\eqref{eq:eq2}, with $V_a(y')$ replaced by
  $V_{6a}(y)$, finishes the proof of (ii).\\
(iii) It only remains to prove the lower bound. Let $t=|x-y|$. First assume
$t\geq L/2$. Then Lemma~\ref{def:lemmalawler} (iii) gives
\begin{equation*}
\Prw_x\left(T_{V_{2L/3}}<\tau_L\right) \geq c\frac{\dL(x)}{t},
\end{equation*}
and the claim follows from the strong Markov property and Lemma~\ref{def:lemmalawler} (i).
Now assume $t<L/2$. Let $x'\in V_L$ such that $V_t(x')\subset V_L$ and
$y\in\partial V_t(x')$. If $\dL(x)> t/2$, there is by
Lemma~\ref{def:lemmalawler} (i) a strictly positive
probability to exit the ball $V_{t/2}(x)$ within $V_{2t/3}(x')$. Since by
the same lemma,
\begin{equation}
\label{eq:eq4}
\inf_{z\in V_{2t/3}(x')}\Prw_{z}\left(\tau_L=y\right) \geq ct^{-(d-1)},
\end{equation}
we obtain the claim in this
case again by applying the strong Markov property. Finally, assume $\dL(x) \leq
t/2$. Then a careful evaluation of the expression in Lemma~\ref{def:lemmalawler} (iii) shows
\begin{equation*}
\Prw_x\left(T_{V_{L-2t/3}}<\tau_L\right) \geq c\frac{\dL(x)}{t},
\end{equation*}
and
\begin{eqnarray*}
\Prw_x\left(\tau_L=y\right)&\geq&
\Prw_x\left(\tau_L=y,\,T_{L-2t/3}<\tau_L,\,T_{V_{2t/3}(x')}<\tau_L\right)\\
&\geq& c\frac{\dL(x)}{t}\Prw_x\left(\tau_L=y\,|\,T_{L-2t/3}<\tau_L,\,
  T_{V_{2t/3}(x')}<\tau_L\right)\\
&&\times\Prw_x\left(T_{V_{2t/3}(x')}<\tau_L\,|\,T_{L-2t/3}<\tau_L\right).
\end{eqnarray*}
By a simple geometric consideration and again Lemma~\ref{def:lemmalawler}
(i), the second probability on the right side is bounded from below by some
$\delta>0$, and the first probability
has already been estimated in~\eqref{eq:eq4}.
\end{prooof2}

\subsection{Proofs of Propositions~\ref{def:super-localclt} and~\ref{def:super-behaviorgreen}}
Since $\ph_m(x,y) = \ph_m(0,y-x)$, it suffices to look at
$\ph_m(x)=\ph_m(0,x)$ and $\gh_{m,\mathbb{Z}^d}(x) =
\gh_{m,\mathbb{Z}^d}(0,x)$. Recall the definition of $\gamma_m$ from Section~\ref{clt}.

\begin{prooof2}{{\bf of Proposition~\ref{def:super-localclt}:}}
For bounded $m$, that is $m\leq m_0$ for some $m_0$, the result is
a special case of~\cite{LawLim}, Theorem 2.1.1. Also, for $n\leq
n_0$ and all $m$, the statement follows from Lemma~\ref{def:app-kernelest}
(i). We therefore have to prove the proposition only for large $n$ and $m$.
To this end, let $B_m = [-\sqrt{\gamma_m}\,\pi,\,\sqrt{\gamma_m}\,\pi]^d$, and for $\theta\in B_m$ set
\begin{equation*}
\phi_m(\theta) = \sum_{y\in\mathbb{Z}^d}e^{i\theta\cdot y/\sqrt{\gamma_m}}\ph_m(y).
\end{equation*} 
The Fourier inversion formula gives
\begin{equation*}
\ph_m^n(x) = \frac{1}{(2\pi)^d\gamma_m^{d/2}}
\int\limits_{B_m}e^{-ix\cdot\theta/\sqrt{\gamma_m}}[\phi_m(\theta)]^n\dt\theta. 
\end{equation*}
We decompose the integral into
\begin{equation*}
(2\pi)^d\gamma_m^{d/2}n^{d/2}\ph_m^n(x) = I_0(n,m,x)+\ldots+ I_3(n,m,x), 
\end{equation*}
where, with $\beta = \sqrt{n}\;\theta$,
\begin{eqnarray*}
I_0(n,m,x) &=& \int\limits_{\mathbb{R}^d}e^{-ix\cdot\beta/\sqrt{n\gamma_m}}e^{-|\beta|^2/2}\dt\beta,\\
I_1(n,m,x) &=& \int\limits_{|\beta|\leq
  n^{1/4}}e^{-ix\cdot\beta/\sqrt{n\gamma_m}}\left([\phi_m(\beta/\sqrt{n})]^n-e^{-|\beta|^2/2}\right)\dt\beta,
\\ 
I_2(n,m,x) &=& -\int\limits_{|\beta|>n^{1/4}}e^{-ix\cdot\beta/\sqrt{n\gamma_m}}e^{-|\beta|^2/2}\dt\beta,\\ 
I_3(n,m,x) &=& n^{d/2}\int\limits_{n^{-1/4}<|\theta|,\, \theta \in B_m}e^{-ix\cdot\theta/\sqrt{\gamma_m}}[\phi_m(\theta)]^n\dt\theta.\\
\end{eqnarray*}
By completing the square in the exponential, we get
\begin{equation*}
I_0(n,m,x)= (2\pi)^{d/2}\exp\left(-\frac{|x|^2}{2n\gamma_m}\right).
\end{equation*}
For $I_1$ and $|\beta|\leq n^{1/4}$, we expand $\phi_m$ in a series around
the origin,
\begin{eqnarray}
\label{eq:app-lclt1}
\phi_m(\beta/\sqrt{n}) &=& 1 - |\beta|^2/2n + |\beta|^4O\left(n^{-2}\right),\nonumber\\
\log \phi_m(\beta/\sqrt{n}) & = & - |\beta|^2/2n + |\beta|^4O\left(n^{-2}\right).
\end{eqnarray}
Therefore,
\begin{equation*}
[\phi_m(\beta/\sqrt{n})]^n = e^{-|\beta|^2/2}\left(1+|\beta|^4O\left(n^{-1}\right)\right),
\end{equation*}
so that
\begin{equation*}
\left|I_1(n,m,x)\right| \leq O\left(n^{-1}\right)\int\limits_{|\beta|\leq
  n^{1/4}}e^{-|\beta|^2/2}|\beta|^4\dt \beta = O\left(n^{-1}\right).
\end{equation*}
Similarly, $I_2$ is bounded by
\begin{equation*}
|I_2(n,m,x)| \leq C \int_{n^{1/4}}^\infty r^{d-1}e^{-r^2/2}\dt r = O\left(n^{-1}\right).
\end{equation*}
Concerning $I_3$, we follow closely~\cite{BZ}, proof of
Proposition B1, and split the integral further into
\begin{eqnarray*}
n^{-d/2}I_3(n,m,x) &=& \int\limits_{n^{-1/4}<|\theta| \leq a} +
\int\limits_{a<|\theta| \leq A} +
\int\limits_{A<|\theta| \leq m^{\alpha}} +
\int\limits_{m^{\alpha}<|\theta|,\, \theta \in B_m}\\
& =& \left(I_{3,0} + I_{3,1} +
I_{3,2} + I_{3,3}\right)(n,m,x),
\end{eqnarray*}
where $0<a<A$ and $\alpha\in (0,1)$ are constants that will be chosen
in a moment, independently of $n$ and $m$.
By~\eqref{eq:app-lclt1}, we can find $a>0$ such that for
$|\beta|\leq a\sqrt{n}$, 
$\log\phi_m(\theta) \leq  -|\theta|^2/3$ (recall that $\beta = \sqrt{n}\,\theta$). Then
\begin{equation*}
|I_{3,0}(n,m,x)|\leq C\int_{n^{-1/4}}^\infty r^{d-1}e^{-nr^2/3}\dt r = O\left(n^{-(d+2)/2}\right).
\end{equation*}
As a consequence of Lemma~\ref{def:lemmalawler} (i) and of our coarse
graining, it follows that for any $0<a < A$, one has for some $0 < \rho =
\rho(a,A) < 1$, uniformly in $m$,
\begin{equation*}
\sup_{a\leq |\theta| \leq A}|\phi_m(\theta)| \leq \rho.            
\end{equation*}
Using this fact,
\begin{equation*}
|I_{3,1}(n,m,x)| \leq CA^{d}\rho^n = O\left(n^{-(d+2)/2}\right).
\end{equation*}
To deal with the last two integrals is more delicate
since we have to take into account the $m$-dependency. First,
\begin{equation*}
|I_{3,2}(n,m,x)| \leq \int\limits_{A<|\theta| \leq
  m^\alpha}\left|\phi_m(\theta)\right|^n\dt \theta.
\end{equation*}
We bound the integrand pointwise. Since $\ph_m(\cdot)$ is invariant
under rotations preserving $\mathbb{Z}^d$, it suffices to look at $\theta$
with all components positive. Assume $\theta_1 =
\max\left\{\theta_1,\ldots,\theta_d\right\}$. Set $M=\lfloor 2\pi\sqrt{\gamma_m}/\theta_1\rfloor$ and $K =
\lfloor 5m/M\rfloor$. Notice that 
$\ph_m(x) > 0$ implies $|x| < 2m$. By taking $A$ large enough, we can
assume that on the domain of integration, $M \leq m$. First,
\begin{equation*}
\phi_m(\theta) =
\sum_{\left(x_2,\ldots,x_d\right)}\exp\left(\frac{i}{\sqrt{\gamma_m}}\sum_{s=2}^dx_s\theta_s\right)
\sum_{j=1}^K\sum_{x_1=-2m
  + (j-1)M}^{-2m +
  jM-1}\exp\left(\frac{ix_1\theta_1}{\sqrt{\gamma_m}}\right)\ph_m(x).
\end{equation*}
Inside the $x_1$-summation, we write for each $j$ separately
\begin{equation*}
\ph_m(x) = \ph_m(x)-\ph_m(x^{(j)})+\ph_m(x^{(j)}),
\end{equation*}
where $x^{(j)} = (-2m + (j-1)M,x_2,\ldots,x_d)$. By Corollary~\ref{def:app-kernelest-cor}, 
\begin{equation*}
  \left|\ph_m(x)-\ph_m(x^{(j)})\right| \leq C\left|\frac{x_1+2m
      -(j-1)M}{m}\right|^{1/2}m^{-d}.
\end{equation*}
Thus,
\begin{equation*}
\left|\sum_{x_1=-2m+ (j-1)M}^{-2m
    +jM-1}\exp\left(\frac{ix_1\theta_1}{\sqrt{\gamma_m}}\right)\left(\ph_m(x)-\ph_m(x^{(j)})\right)\right| \leq
C\theta_1^{-3/2}m^{-d+1},
\end{equation*}
and
\begin{equation*}
  \left|\sum_{j=1}^K\sum_{x_1=-2m
      + (j-1)M}^{-2m +
      jM-1}\exp\left(\frac{ix_1\theta_1}{\sqrt{\gamma_m}}\right)\left(\ph_m(x)-\ph_m(x^{(j)})\right)\right|
  \leq C\theta_1^{-1/2}m^{-d+1}.
\end{equation*}
On our domain of integration, $0<\left(\theta_1/\sqrt{\gamma_m}\right) \leq
Cm^{\alpha-1} < 2\pi$ for large $m$. Therefore,
\begin{eqnarray*}
  \left|\sum_{j=1}^K\ph_m(x^{(j)})\sum_{x_1 = -2m +(j-1)M}^{-2m +
      jM-1}\exp\left(\frac{ix_1\theta_1}{\sqrt{\gamma_m}}\right)\right| &\leq&
  CKm^{-d}\left|\frac{1-\exp\left(i\theta_1M/\sqrt{\gamma_m}\right)}{1-\exp\left(i\theta_1/\sqrt{\gamma_m}\right)}\right|\\
&\leq& C|\theta|m^{-d},
\end{eqnarray*}
and altogether for sufficiently large $A$, $m$ and $n$, 
\begin{equation*}
\int\limits_{A<|\theta| \leq
  m^\alpha}\left|\phi_m(\theta)\right|^n\dt \theta \leq C_1^n\int\limits_{A<|\theta| \leq m^\alpha}\left(\frac{1}{\sqrt{|\theta|}} +\frac{|\theta|}{m}\right)^n\dt \theta = O\left(n^{-(d+2)/2}\right).
\end{equation*}
For $I_{3,3}$ we again assume all components of $\theta$ positive and
$\theta_1 = \max\{\theta_1,\ldots,\theta_d\}$. Since
\begin{equation*}
\ph_m(x) = \sum_{y=-2m}^{x_1}\left(\ph_m(y,x_2,\ldots,x_d)-\ph_m(y-1,x_2,\ldots,x_d)\right),
\end{equation*}
we have 
\begin{eqnarray*}
  \lefteqn{\left|\phi_m(\theta)\right|}\\ 
&\leq& Cm^{d-1}
  \left|\sum_{x_1=-2m}^{2m}\exp\left(\frac{ix_1\theta_1}{\sqrt{\gamma_m}}\right)\sum_{y=-2m}^{x_1}\left(\ph_m(y,x_2,\ldots,x_d)-\ph_m(y-1,x_2,\ldots,x_d)\right)\right|
  \\ 
  &\leq&
  Cm^{d-1}\sum_{y=-2m}^{2m}\left|\ph_m(y,x_2,\ldots,x_d)-\ph_m(y-1,x_2,\ldots,x_d)\right|\left|\sum_{x_1=y}^{2m}\exp\left(\frac{ix_1\theta_1}{\sqrt{\gamma_m}}\right)\right|.
\end{eqnarray*}
The sum over the exponentials is estimated by $C m/|\theta|$, so that again
with Corollary~\ref{def:app-kernelest-cor},
\begin{equation*}
\left|\phi_m(\theta)\right| \leq C_2 m^{1/2}|\theta|^{-1}.
\end{equation*}
Hence, for $\alpha$ close to $1$ and large $n$, $m$,
\begin{equation*}
\int\limits_{m^{\alpha}<|\theta|,\, \theta \in
  B_m}\left|\phi_m^n(\theta)\right|\dt\theta \leq C_2^nm^{n/2+\alpha(d-n)}= O\left(n^{-(d+2)/2}\right).
\end{equation*}
\end{prooof2}
For Proposition~\ref{def:super-behaviorgreen}, we still need a large deviation estimate.
\begin{lemma}[Large deviation estimate]
\label{def:app-ldestim}
There exist constants $c_1,c_2 > 0$ such that for $|x| \geq 3m$
\begin{equation*}
\ph_m^n(x) \leq c_1m^{-d}\exp\left(-\frac{|x|^2}{c_2nm^2}\right).
\end{equation*}
\end{lemma}
\begin{prooof}
  Write $\Prw$ for $\Prw_{0,\ph_m}$ and $\Erw$ for the expectation with
  respect to $\Prw$, and denote by $X_n^j$ the $j$th component of the random
  walk $X_n$ under $\Prw$. For $r
  > 0$,
\begin{eqnarray*}
\sum_{y: |y| \geq r}\ph_m^n(y) 
&\leq& \sum_{j=1}^d\Prw(|X_n^j| \geq d^{-1/2}r)\nonumber\\
&=& 2d\Prw(X_n^1 \geq d^{-1/2}r).
\end{eqnarray*}
We claim that 
\begin{equation*}
\Prw(X_n^1 \geq d^{-1/2}r) \leq \exp\left(-\frac{r^2}{8dnm^2}\right).
\end{equation*}
By the martingale maximal inequality for all $t, \lambda > 0$,
\begin{equation*}
\Prw(X_n^1 \geq \lambda) \leq e^{-t\lambda}\Erw\left[\exp(tX_n^1)\right] =
e^{-t\lambda}\left(\Erw\left[\exp(tX_1^1)\right]\right)^n. 
\end{equation*}
Since $X_1^1 \in (-2m,2m)$ and $x\rightarrow e^{tx}$ is convex, it follows that
\begin{equation*}
\exp(tX_1^1)\leq \frac{1}{2}\frac{(2m-X_1^1)}{2m}e^{-2tm} + \frac{1}{2}\frac{(2m+X_1^1)}{2m}e^{2tm}. 
\end{equation*}
Therefore, using the symmetry of $X_1^1$,
\begin{equation*}
\Erw\left[\exp\left(tX_n^1\right)\right] \leq
\left(\frac{1}{2}e^{-2tm}+\frac{1}{2}e^{2tm}\right)^n = \cosh^n(2tm)\leq e^{2nt^2m^2},
\end{equation*}
and
\begin{equation*}
\Prw(X_n^1 \geq d^{-1/2}r)\leq e^{-td^{-1/2}r}e^{2nt^2m^2}.
\end{equation*}
Putting $t = r/(4\sqrt{d}\,nm^2)$ we get
\begin{equation*}
\Prw(X_n^1 \geq d^{-1/2}r)\leq \exp\left({-\frac{r^2}{8dnm^2}}\right).
\end{equation*}
From this it follows that
\begin{eqnarray*}
\ph_m^n(x) = \sum_{y:|y| \geq |x|-2m}\ph_m^{n-1}(y)\ph_m(x-y)&\leq&
\frac{c_1}{m^d}\exp\left({-\frac{(|x|-2m)^2}{8d(n-1)m^2}}\right)\\
&\leq&
\frac{c_1}{m^d}\exp\left({-\frac{|x|^2}{c_2nm^2}}\right).
\end{eqnarray*}
\end{prooof}
\begin{prooof2}{\bf of Proposition~\ref{def:super-behaviorgreen}:}
(i) follows from Proposition~\ref{def:super-localclt}. For (ii), we set 
\begin{equation*}
N = N(x,m) = \frac{|x|^2}{\gamma_m}\left(\log\frac{|x|^2}{\gamma_m}\right)^{-2}.
\end{equation*}
We split $\gh_{m,\mathbb{Z}^d}(x)$ into
\begin{equation*}
\gh_{m,\mathbb{Z}^d}(x) = \sum_{n=1}^\infty\ph_m^n(x) =
\sum_{n=1}^{\lfloor N\rfloor}\ph_m^n(x)+
\sum_{n=\lfloor N\rfloor +1}^{\infty}\ph_m^n(x). 
\end{equation*}
For the first sum on the right, we use the large deviation estimate from
Lemma~\ref{def:app-ldestim},
\begin{equation*}
\sum_{n=1}^{\lfloor N\rfloor}\ph_m^n(x) \leq c_1m^{-d}\sum_{n=1}^{\lfloor
  N\rfloor}\exp\left(-\frac{|x|^2}{c_2\,nm^2}\right) = O\left(|x|^{-d}\right).
\end{equation*}
In the second sum, we replace the transition probabilities by the
expressions obtained in Proposition~\ref{def:super-localclt}. 
The error terms are estimated by
\begin{equation*}
  \sum_{n=\lfloor N\rfloor +1}^\infty O\left(m^{-d}n^{-(d+2)/2}\right) =
  O\left(|x|^{-d}\left(\log\frac{|x|^2}{\gamma_m}\right)^d\right). 
\end{equation*}
Putting $t_n = 2\gamma_mn/|x|^2$, we obtain for the main part
\begin{eqnarray*}
  \lefteqn{\sum_{n=\lfloor
    N\rfloor + 1}^\infty\frac{1}{(2\pi\gamma_mn)^{d/2}}\exp\left(-\frac{|x|^2}{2\gamma_mn}\right)}\\
  &=&
  \frac{|x|^{-d+2}}{2\pi^{d/2}\gamma_m}\sum_{n=\lfloor N\rfloor +1}^\infty
  t_n^{-d/2}\exp(-1/t_n)(t_n-t_{n-1})\nonumber\\  
  &=& \frac{|x|^{-d+2}}{2\pi^{d/2}\gamma_m}\int_0^\infty t^{-d/2}\exp(-1/t)\mbox{d}t
  + O\left(|x|^{-d}\right).
\end{eqnarray*}
This proves the statement for $|x| \geq 3m$ with
\begin{equation*}
c(d) = \frac{1}{2\pi^{d/2}}\int_0^\infty t^{-d/2}\exp(-1/t)\mbox{d}t.
\end{equation*}
\end{prooof2}
\subsection*{Acknowledgments}
I would like to thank Erwin Bolthausen and Ofer Zeitouni for many helpful
discussions and constant support.



\end{document}